\documentclass[reqno]{amsart}[utf8, 11pt]%
\usepackage{amsmath}
\usepackage{amsfonts}
\usepackage{amssymb}
\usepackage{mathrsfs}
\usepackage{graphicx}
\usepackage{hyperref}
\usepackage{stmaryrd}
\usepackage{enumerate}
\usepackage{bbm}
\usepackage{color}%
\usepackage{esint}
\usepackage{mathtools}
\usepackage{etoolbox}
\DeclareFontFamily{U}{matha}{\hyphenchar\font45}
\DeclareFontShape{U}{matha}{m}{n}{
	<5> <6> <7> <8> <9> <10> gen * matha
	<10.95> matha10 <12> <14.4> <17.28> <20.74> <24.88> matha12
}{}
\DeclareSymbolFont{matha}{U}{matha}{m}{n}
\DeclareFontSubstitution{U}{matha}{m}{n}

\DeclareFontFamily{U}{mathx}{\hyphenchar\font45}
\DeclareFontShape{U}{mathx}{m}{n}{
	<5> <6> <7> <8> <9> <10>
	<10.95> <12> <14.4> <17.28> <20.74> <24.88>
	mathx10
}{}
\DeclareSymbolFont{mathx}{U}{mathx}{m}{n}
\DeclareFontSubstitution{U}{mathx}{m}{n}

\makeatletter
\def\l@subsection{\@tocline{2}{0pt}{2.5pc}{5pc}{}}
\makeatother

\DeclareMathDelimiter{\vvvert}{0}{matha}{"7E}{mathx}{"17}
\apptocmd{\thebibliography}{\renewcommand{\sc}{}}{}{}
\usepackage{xpatch,amsthm}
\makeatletter
\xpatchcmd{\@thm}{\fontseries\mddefault\upshape}{}{}{} 
\makeatother
\newtheorem{theorem}{Theorem}

\newtheorem{case}[theorem]{Case}
\newtheorem{claim}[theorem]{Claim}

\newtheorem{corollary}[theorem]{Corollary}

\newtheorem{definition}[theorem]{Definition}

\newtheorem{lemma}[theorem]{Lemma}

\newtheorem*{question*}{Question}
\newtheorem{proposition}[theorem]{Proposition}
\newtheorem{remark}[theorem]{Remark}

\newcommand{\X}{\mathfrak X}
\newcommand{\Xu}{\mathfrak X u}

\usepackage{newunicodechar} 
\newunicodechar{​}{\hspace{0pt}}

\begin{document}
\author[A. Mallick]{Arka Mallick}
\email{arkamallick@iisc.ac.in (A. Mallick)}

\author[S. Sil]{Swarnendu Sil}
\email{swarnendusil@iisc.ac.in (S.Sil)}

\address{Department of Mathematics\\
Indian Institute of Science\\
Bangalore, India}


	\title{A unified approach to nonlinear Stein theorems}	
	

	\begin{abstract}
		We study regularity for the variable exponent quasilinear $\mathfrak{p}\left(x\right)$-Laplace type equation on domains in Heisenberg groups and Euclidean spaces. We establish borderline continuity estimates for the appropriate first order derivatives of solutions. More precisely, we show that for any weak solution $u \in 	\mathbb{E}W^{1, \mathfrak{p}(\cdot)}\left(\Omega\right)$  of 
		\begin{align*}
			\operatorname{div}_{\mathbb{E}} \left( \mathfrak{a}(x) \lvert \nabla_{\mathbb{E}} u \rvert^{\mathfrak{p}\left(x\right)-2} \nabla_{\mathbb{E}} u \right) = f \qquad \text{ in } \Omega, 
		\end{align*}
		where $\Omega \subset \mathbb{E}^{n}$, $\mathfrak{a}$ is a uniformly positive bounded scalar function and the exponent function $\mathfrak{p}$ is uniformly bounded away from $1$ and $\infty,$ $\nabla_{\mathbb{E}}u$ is continuous in $\Omega$ as soon as $f \in L^{\left( Q_{\mathbb{E}}, 1\right)}\left(\Omega\right)$ and $\mathfrak{a}, \mathfrak{p}$ satisfies some conditions regarding the summability of their mean-oscillations. Here $\mathbb{E}^{n}$ is either the Heisenberg group $\mathbb{H}_{n}$ or the Euclidean space $\mathbb{R}^{n}$ and $\operatorname{div}_{\mathbb{E}}$, $\nabla_{\mathbb{E}}$, $Q_{\mathbb{E}}$ stands for the corresponding divergence, gradient and homogeneous dimension, respectively. We treat the elliptic and subelliptic cases in a unified manner using Euclidean techniques and our conditions on $\mathfrak{a}$ and $\mathfrak{p}$ are new and weaker than all the known sufficient conditions even in the Euclidean case. However, all the known sufficient conditions imply our conditions, achieving yet another unification.  
	\end{abstract}
	\maketitle 
	\smallskip 
	
	\tableofcontents
	
	\smallskip

	\section{Introduction and main results}
	\subsection{History of the problem}
	\subsubsection{History of the problem in Euclidean spaces}
	Let $U \subset \mathbb{R}^n$ be an open subset. The prototypical quasilinear elliptic equation is the $p$-Laplacian, given by 
	\begin{align}\label{p-laplace}
		-\operatorname{div}\left(\left\lvert \nabla u \right\rvert^{p-2}\nabla u\right) &= f &&\text{ in } U, 
	\end{align}
	where $1 < p < \infty$ is a real number and $f:U \rightarrow \mathbb{R}$ is a given real-valued function. When $f \equiv 0,$ this equation arises as the Euler-Lagrange equation for the $p$-Dirichlet integral $I[u]=\int_{U}\left\lvert \nabla u \right\rvert^{p}/p$ and is linear and uniformly elliptic when $p=2$, but quasilinear otherwise and degenerate elliptic for $p>2$ and singular elliptic for $1< p <2.$ Unlike the case of harmonic functions, $p$-harmonic functions (when $f\equiv 0$) are not in general smooth for $p \neq 2.$ As is well-known, in general, $p$-harmonic functions  are locally $C^{1, \alpha}$ for some $0< \alpha < 1,$ as established by Uraltceva \cite{Uralceva}, Uhlenbeck \cite{uhlenbecknonlinearelliptic},  Evans \cite{Evans_pLaplacian}, DiBenedetto \cite{DiBenedetto-C1-alpha-83}, Lewis \cite{Lewis83}, Tolksdorff \cite{Tolksdorf_regularity} (see also Giaquinta-Modica \cite{Giaquinta_Modica_Uhlenbeck_result}, Hamburger \cite{hamburgerregularity}). 
	
	\par For the inhomogeneous equation, the minimum regularity for $f$ that forces weak solutions of \eqref{p-laplace} to be $C^1$ is a question that has attracted a lot of contributions since 2010. When $p=2,$ combining a celebrated result of Stein \cite{Stein_steintheorem} with the usual linear Calderon-Zygmund estimates \cite{Calderon_Zygmund_CZestimate}, one can prove that $u$ is $C^1$ when $f$ is in the Lorentz space $L^{\left(n,1\right)}.$ This condition on $f$ is known to be sharp (see \cite{Cianchi_SharpLorentzexponent}). Several important contributions by Duzaar-Mingione and Mingione (see \cite{DuzaarMingione_gradientcontinuityestimates}, \cite{Duzaar_Mingione_linear_nonlinear_potentials}, \cite{Duzaar_Mingione_nonlinearpotentials}, \cite{Mingione_gradient_potential_estimate}) culminated in a remarkable \emph{nonlinear Stein theorem} in Kuusi-Mingione \cite{KuusiMingione_Steinequationslinearpotentials} (see also \cite{KuusiMingione_nonlinearStein} for a vectorial version). The result in \cite{KuusiMingione_Steinequationslinearpotentials} states: Let $u \in W^{1,p}\left(U\right)$ be a weak solution of the following equation 
	\begin{align*}
		-\operatorname{div}\left(a(x)\left\lvert \nabla u \right\rvert^{p-2}\nabla u\right) &= f &&\text{ in } U, 
	\end{align*} 
	where $1 < p < \infty$ is a real number and $a: U \rightarrow [\nu, L]$ for some constants $0 < \nu \le L < \infty.$ Then $\nabla u$ is continuous in $U$ as soon as $a$ is Dini continuous and $f \in L^{\left(n,1\right)}\left(U\right).$ Notably, if $a$ is merely uniformly continuous, then one can construct a non-Lipschitz solution even in the linear case, as shown in \cite{JinMazyaVanSchaftingen_dinisharp}.

	On the other hand, the study of `variable growth' problems in the Euclidean case started gaining attention after Zhikov investigated them in \cite{Zhikov_Lavrentiev}, \cite{Zhikov_higherintegrability}, where he studied the variational problem on bounded open subsets $U \subset \mathbb{R}^n$
	\begin{align*}
		\inf \left\lbrace \int_{U} \frac{1}{p(x)}\left\lvert \nabla u \right\rvert^{p(x)}\ \mathrm{d}x: u \in u_{0} + W_{0}^{1, p\left(\cdot\right)}\left( U \right)\right\rbrace.  
	\end{align*}
	The Euler-Lagrange equation is the so-called $p(x)$-Laplacian, i.e. 
	\begin{align*}
		-\operatorname{div} \left(\left\lvert  \nabla u \right\rvert^{p(x)-2} \nabla u \right) &= 0 &&\text{ in } U. 
	\end{align*}
	Apart from being mathematically interesting, these types of variable growth problem arise in a variety of applications, e.g. in the theory of electrorheological fluids, thermistor problem, fluid flow in porous media, magnetostatics, etc. Zhikov \cite{Zhikov_Lavrentiev} showed that the minimizers need not have any additional regularity if the exponent function $p$ is allowed to be too irregular (see \cite{Balci_Diening_Surnachev_LavrentievGap} and the references therein for the state-of-the-art at present). On the other hand, when $p$ is regular enough, one can establish regularity results for the equation 
	\begin{align*}
		-\operatorname{div} \left(\left\lvert  \nabla u \right\rvert^{p(x)-2} \nabla u \right) &= f &&\text{ in } U. 
	\end{align*}
	H\"{o}lder continuity and $L^{q}$ estimates of $\nabla u$ are proved by Acerbi and Mingione in \cite{Acerbi_Mingione_C1alpha} and \cite{Acerbi_Mingione_CZforpxLaplacian}, respectively. The question of $C^1$ regularity is tackled in Ok \cite{Ok_pxStein} (see also \cite{Ok_C1minima_px}), where the main result states: Let $u \in W^{1,p(\cdot)}\left(U\right)$ be a weak solution of the following equation 
	\begin{align*}
		-\operatorname{div}\left(a(x)\left\lvert \nabla u \right\rvert^{p(x)-2}\nabla u\right) &= f &&\text{ in } U, 
	\end{align*} 
	where $p:U \rightarrow [\gamma_{1}, \gamma_{2}]$ and $a:U \rightarrow [\nu, L]$ are continuous functions  with $1 < \gamma_{1} \leq \gamma_{2} < \infty$ and $ 0 < \nu < L < \infty.$ Then $\nabla u$ is continuous in $U$ if $a$ is Dini, $p$ is log-Dini and $f \in L^{\left(n,1\right)}\left(U\right).$ Recently, Baroni \cite{Baroni_pxStein} proved a related result for the homogeneous equation (i.e. with $f \equiv 0$) with different assumptions on $a$ and $p$. Baroni's main result states: Let $p:U \rightarrow [\gamma_{1}, \gamma_{2}]$ and $a:U \rightarrow [\nu, L]$ be measurable functions  with $1 < \gamma_{1} \leq \gamma_{2} < \infty$ and $ 0 < \nu < L < \infty.$ If $u \in W^{1,p(\cdot)}\left(U\right)$ is a weak solution to  
	\begin{align*}
		-\operatorname{div}\left(a(x)\left\lvert \nabla u \right\rvert^{p(x)-2}\nabla u\right) &= 0 &&\text{ in } U, 
	\end{align*} 
	then $\nabla u$ is continuous if $\nabla a \in L^{(n,1)}$ and $\nabla p \in L^{(n,1)}\log L.$
	
	\subsubsection{History of the problem in Heisenberg groups}
	On the subelliptic side, the PDE of interest for the constant exponent $p$ is the subelliptic $p$-Laplacian, given by 
	\begin{align*}
		\operatorname{div}_{\mathbb{H}}\left[ \lvert \X u \rvert^{p-2} \X u\right]    = f   &&\text{ in } \Omega,
	\end{align*}
	where $\Omega \subset \mathbb{H}^{n}$ is open, $\Xu$ denotes the horizontal gradient of a function $u:\Omega \subset \mathbb{H}^{n} \rightarrow \mathbb{R}$ and $\operatorname{div}_{\mathbb{H}}$ denotes the horizontal divergence. Unlike the Euclidean derivatives, which commute, here the horizontal vector fields $\mathfrak{X}_{1}, \ldots, \mathfrak{X}_{2n}$ do not always commute and this presents a significant obstacle to proving regularity for horizontal derivatives. 
	
	\par Regularity results for this problem have been investigated in a number of important works. Notable contributions include the works of Capogna-Danielli-Garofalo \cite{Capogna_Danielli_Garofalo}, Capogna \cite{Capogna_regularity, Capogna_quasiconformal}, Capogna-Garofalo \cite{Capogna_Garofalo_partialregularity_step2} , Domokos \cite{Domokos_differentiabilityofTu}, Domokos-Manfredi \cite{Domokos_Manfredi_pnear2}, Mingione-Zatorska-Goldstein-Zhong \cite{Mingione_ZatorskaGoldstein_Zhong} (see also the books Capogna-Danielli et.al. \cite{Capogna_Danielli_book} and Riccotti \cite{Ricciotti_pLaplaceHeisenberg}). H\"{o}lder continuity of the horizontal gradient in full generality was established in the unpublished work by Zhong \cite{zhong2018regularityvariationalproblemsheisenberg} for $p>2$, which marks an important breakthrough. This was extended to the full range $1 < p < \infty$ in  Mukherjee-Zhong \cite{Mukherjee_Zhong}. Several subsequent contributions are Mukherjee-Sire \cite{Mukherjee_Sire}, Mukherjee \cite{Mukherjee_Lipschitz, Mukherjee_C1alpha}, Citti-Mukherjee \cite{Citti_Mukherjee_Step2}. A  relatively recent survey can be found in Capogna-Citti-Zhong \cite{Capogna_Citti_Zhong}. 
	
	In spite of this progress, the question of minimal regularity of $f$ to ensure continuity of $\Xu$ remained open for a while. The main reason for this discrepancy is the lack of excess decay estimates. In the Euclidean case, the excess decay estimates for $p$-harmonic functions are well-known and different versions are established in Hamburger \cite{hamburgerregularity}, DiBenedetto-Manfredi \cite{DiBenedettoManfredi} and Lieberman \cite{Lieberman_expoentdecreasing}. However, while the method of Zhong \cite{zhong2018regularityvariationalproblemsheisenberg} and Mukherjee-Zhong \cite{Mukherjee_Zhong} established an `energy-to-excess decay estimate', the `excess-to-excess decay estimate' remained unknown in the subelliptic context. This significant barrier is finally broken through in \cite{Mallick_Sil_ExcessDecay}, where the present authors have established a Euclidean-like excess-to-excess decay estimate. As a consequence, a subelliptic analogue for the Kuusi-Mingione's nonlinear Stein theorem is established in \cite{Mallick_Sil_ExcessDecay}. Later, the present authors proved a subelliptic analogue of Baroni's result for variable-exponent setting in \cite{Mallick_Sil_continuityvariablegrowth}. 
	
	\subsection{Main results}
	The following two results are our main contributions. 
	\begin{theorem}\label{elliptic main them}
			Let $\Omega \subset \mathbb{R}^{n}$ be open and bounded. Let $\mathfrak{p}:\Omega \rightarrow [\gamma_{1}, \gamma_{2}]$ and $\mathfrak{a}:\Omega \rightarrow [\nu, L]$ be measurable functions  with $1 < \gamma_{1} \leq \gamma_{2} < \infty$ and $ 0 < \nu < L < \infty,$ with the property that there exists $\tau_{0} \in (0, 1/4)$ such for every $ 0 < \tau < \tau_{0} $ and every $1< q< \infty,$ we have 
			\begin{align*}
				\lim\limits_{r \rightarrow 0}	\sup\limits_{x \in K} \sum\limits_{j=0}^{\infty}\left(\fint_{B_{\tau^{j}r}(x)} \left\lvert \mathfrak{a} - \left(\mathfrak{a}\right)_{x, \tau^{j}r} \right\rvert^{q}\right)^{\frac{1}{q}} = 0, \text{ and }
				\end{align*}
				\begin{align*}  
				\lim\limits_{r \rightarrow 0}	\sup\limits_{x \in K} \sum\limits_{j=0}^{\infty}\log \left(\frac{1}{\tau^{j}r}\right)\left(\fint_{B_{\tau^{j}r}(x)} \left\lvert \mathfrak{p} - \left(\mathfrak{p}\right)_{x, \tau^{j}r} \right\rvert^{q}\right)^{\frac{1}{q}} = 0, 
			\end{align*}
			whenever $K \subset \subset \Omega$ is a compact subset. If $u \in W_{\text{loc}}^{1,p\left(\cdot\right)} \left(\Omega\right)$ is a local weak solution to the equation 		\begin{align*}
				\operatorname{div}\left[ \mathfrak{a}(x) \lvert \nabla u \rvert^{\mathfrak{p}\left(x\right)-2} \nabla u\right]    = f   &&\text{ in } \Omega, 
			\end{align*}
			for some $f \in L^{\left(n,1\right)}\left( \Omega\right),$ then $\nabla u$ is continuous in $\Omega.$
		\end{theorem}
		\begin{theorem}\label{subelliptic main them}
			Let $\Omega \subset \mathbb{H}^{n}$ be open and bounded. Let $\mathfrak{p}:\Omega \rightarrow [\gamma_{1}, \gamma_{2}]$ and $\mathfrak{a}:\Omega \rightarrow [\nu, L]$ be measurable functions  with $1 < \gamma_{1} \leq \gamma_{2} < \infty$ and $ 0 < \nu < L < \infty,$ with the property that there exists $\tau_{0} \in (0, 1/4)$ such for every $ 0 < \tau < \tau_{0} $ and every $1< q< \infty,$ we have 
			\begin{align*}
				\lim\limits_{r \rightarrow 0}	\sup\limits_{x \in K} \sum\limits_{j=0}^{\infty}\left(\fint_{B_{\tau^{j}r}(x)} \left\lvert \mathfrak{a} - \left(\mathfrak{a}\right)_{x, \tau^{j}r} \right\rvert^{q}\right)^{\frac{1}{q}} = 0, \text{ and }
				\end{align*}
				\begin{align*} 
				\lim\limits_{r \rightarrow 0}	\sup\limits_{x \in K}\sum\limits_{j=0}^{\infty}\log \left(\frac{1}{\tau^{j}r}\right)\left(\fint_{B_{\tau^{j}r}(x)} \left\lvert \mathfrak{p} - \left(\mathfrak{p}\right)_{x, \tau^{j}r} \right\rvert^{q}\right)^{\frac{1}{q}} = 0, 
			\end{align*}
			whenever $K \subset \subset \Omega$ is a compact subset.  If $u \in HW_{\text{loc}}^{1,p\left(\cdot\right)} \left(\Omega\right)$ is a local weak solution to the equation 		\begin{align*}
				\operatorname{div}_{\mathbb{H}}\left[ \mathfrak{a}(x) \lvert \X u \rvert^{\mathfrak{p}\left(x\right)-2} \X u\right]    = f   &&\text{ in } \Omega, 
			\end{align*}
			for some $f \in L^{\left(Q,1\right)}\left( \Omega\right),$ then $\X u$ is continuous in $\Omega.$
	\end{theorem}
We present a unified treatment of both the elliptic and subelliptic cases, to illustrate that the perturbation arguments we use in insensitive to commutativity. The summability assumptions on the mean oscillations of $\mathfrak{a}$ and $\mathfrak{p}$ is new already in the Euclidean case.  These assumptions are verified (see Proposition \ref{derivative to summability} and Proposition \ref{dini to summability}) if we assume either of the following 
	\begin{itemize}
		\item \emph{Ok's condition \cite{Ok_pxStein}:} $\mathfrak{a}$ is Dini and $\mathfrak{p}$ is log-Dini (for both elliptic and subelliptic cases).  
		\item \emph{Baroni's condition \cite{Baroni_pxStein}:} $\X \mathfrak{a} \in L^{(Q,1)}$ and $\X \mathfrak{p} \in L^{(Q,1)}\log L$ for the subelliptic case and $\nabla \mathfrak{a} \in L^{(n,1)}$ and $\nabla \mathfrak{p} \in L^{(n,1)}\log L$ for the elliptic case. 
	\end{itemize}
	Thus, already in the Euclidean case, our results extend Baroni's results to inhomogeneous equations. Moreover, our conditions are more general than either of these conditions. As a trivial example, one can allow the cross-possibilities, i.e. $\mathfrak{a}$ is Dini and $\nabla \mathfrak{p} \in L^{(n,1)}\log L$ or $\nabla \mathfrak{a} \in L^{(n,1)}$ and $\mathfrak{p}$ is log-Dini. Furthermore, any condition that interpolates between the pointwise conditions of Ok and the first derivatives conditions of Baroni would also satisfy our assumptions.

	The main technical novelty of our work is a new linearization.  Our linearized comparison estimates are linearized at $L^{\gamma}$ scale, where $\gamma \le \min \lbrace \gamma_{1}, \gamma^{'}_{2}\rbrace.$  This is in contrast to Baroni \cite{Baroni_pxStein}, where estimates are linearized at $L^1$ scale. Ok \cite{Ok_pxStein} also used estimates at $L^{\gamma}$ scale, but since we do not have pointwise control of oscillations of $\mathfrak{a}$ and $\mathfrak{p}$, deriving these linearized comparison estimates are significantly more involved in our case. As the lack of pointwise control for oscillations of $\mathfrak{a}$ and $\mathfrak{p}$ rules out comparing with homogeneous equations where both $\mathfrak{a}$ and $\mathfrak{p}$ are fixed at $\mathfrak{a}(x_{0})$ and $\mathfrak{p}(x_{0})$ for some fixed $x_{0},$ we adopt a hybrid strategy instead. To be more precise, in each ball of a shrinking family of concentric metric balls, the constant coefficient and constant exponent comparison systems are different, depending on the integral average of $\mathfrak{a}$ and $\mathfrak{p}$ on that ball. To handle this additional complexity, our higher integrability estimates needed to be far more robust and flexible. Due to the presence of a non-zero right hand side, we also need a global higher integrability result (see Theorem \ref{global_higher_integrability}), which is new in the subelliptic case. 
	
	\subsection{Exposition and organization}
	To conclude the introduction, we would like to comment on our exposition.  Even the Euclidean version of our main result covers all the known cases, which are proved in multiple articles, with numerous references to several others for intermediate steps and arguments. 
    
	\par On the other hand, the literature on the regularity theory of quasilinear equations in the subelliptic case is so nascent at present that practically all of our intermediate lemmas are, strictly speaking, new results.  However, a lot of them can be proved by a reasonably routine adaptation of their Euclidean counterparts and an expert reader with sufficient familiarity with the proofs in the Euclidean case can easily fill in the details.  Trying to spell out those adaptations of standard arguments would serve little purpose and would perhaps, make the presentation quite boring and repetitive to read for experts. Instead, we focus on the genuinely new parts of the arguments, where actual novelty is required and would only sketch and sometimes even altogether skip furnishing the details for the rest. 
	\par Thus, we deduce the key linearized comparison estimate in Lemma \ref{linearized comparison} in detail.  Our proof of Theorem \ref{unified main them} is then concluded via an exit-time argument, which, though quite nontrivial in itself, is quite well-known since Kuusi-Mingione \cite{KuusiMingione_nonlinearStein} (see also \cite{KuusiMingione_Steinequationslinearpotentials}). So we have only sketched how to finish off the proof rather than presenting an almost verbatim repetition of their arguments. We believe that this choice, would not only allow us to limit our exposition to a reasonable length, but also improve the readability by a significant factor.

	\section{Notations}\label{notations} 
	\subsection{Notations for Heisenberg groups}
	We shall follow the notations in \cite{Mallick_Sil_ExcessDecay}, \cite{Mallick_Sil_continuityvariablegrowth}. 
	For $n\geq 1$, the \textit{Heisenberg Group} denoted by $\mathbb{H}_{n}$, is identified with the Euclidean space 
	$\mathbb{R}^{2n+1}$ with the group operation 
	\begin{equation}\label{eq:group op}
		x\cdot y\, := \Big(x_1+y_1,\ \dots,\ x_{2n}+y_{2n},\ t+s+\frac{1}{2}
		\sum_{i=1}^n (x_iy_{n+i}-x_{n+i}y_i)\Big)
	\end{equation}
	for every $x=(x_1,\ldots,x_{2n},t),\, y=(y_1,\ldots,y_{2n},s)\in \mathbb{H}_{n}$. The left invariant vector fields 
	\[ X_i=  \partial_{x_i}-\frac{x_{n+i}}{2}\partial_t, \quad
	X_{n+i}=  \partial_{x_{n+i}}+\frac{x_i}{2}\partial_t,\] 
	for every $1\leq i\leq n$ and $\X f  = (X_1f,\ldots, X_{2n}f)$ denotes the \textit{Horizontal gradient} of $f$. 
	 For a vector field $F$, $ \operatorname{div}_{\mathbb{H}} (F)  =  \sum_{i=1}^{2n} X_i f_i $ denotes the 
	\textit{Horizontal divergence} of $F$. The bi-invariant Haar measure of $\mathbb{H}_{n}$ is just the Lebesgue 
	measure of $\mathbb{R}^{2n+1}$. The \textit{Carnot-Carath\`eodory metric} (CC-metric) is a left-invariant metric on $\mathbb{H}_{n},$ denoted by $d_{\text{CC}},$ is defined as the length of the shortest horizontal curves, connecting two points. This metric is equivalent to the \textit{Kor\`anyi metric} $d_{Kor}$, defined by the gauge function $\rho(x,t)\coloneqq (\lvert x \rvert ^4 +t^2)^{1/4}.$ 	
	\subsection{Unified notations}
	We now describe our unified notations. \smallskip

 \textbf{Domain space:} $\mathbb{E}^{n}$ is either $\mathbb{R}^{n} \text{ or } \mathbb{H}_{n}.$ $\mathbb{E}^{n}$ will be equipped with a metric $d_{\mathbb{E}}$, which is just the Euclidean metric when $\mathbb{E}^{n} = \mathbb{R}^{n}$ and the \textit{Kor\`anyi metric} $d_{Kor}$ when $\mathbb{E}^{n} = \mathbb{H}_{n}.$ If not mentioned explicitly, $\Omega \subset \mathbb{E}^{n}$ will always be an open and (metric-)bounded subset.\smallskip 
	
	\textbf{Divergence, Gradient and Homogeneous dimension:} We set 
	\begin{align*}
		\operatorname{div}_{\mathbb{E}} &= \operatorname{div}, & \nabla_{\mathbb{E}} &= 	\nabla  &\text{ and } && Q_{\mathbb{E}} &= n  &\text{ if } \quad\mathbb{E}^{n} = \mathbb{R}^{n} ,\\
		\operatorname{div}_{\mathbb{E}} &= \operatorname{div}_{\mathbb{H}}, & \nabla_{\mathbb{E}} &= 	\X &\text{ and } && Q_{\mathbb{E}} &= 2n+2  &\text{ if } \quad\mathbb{E}^{n} = \mathbb{H}_{n}. 
	\end{align*} 
	
	\textbf{Metric ball and Measure:}  For any $x \in \mathbb{E}^{n}$ and any $r>0,$ $B_{r}(x)$ will denote the metric ball of radius $r$ centered at $x$, defined by $B_r(x) = \left\lbrace y\in\mathbb{E}^{n}: d_{\mathbb{E}}(x,y)<r \right\rbrace.$ 
	The bi-invariant Haar measure on $\mathbb{E}^{n}$ is the Lebesgue measure in either case: $n$-dimensional for $\mathbb{R}^{n}$ and $(2n+1)$-dimensional for $\mathbb{H}_{n}.$ For any measurable set $A \subset \mathbb{E}^{n},$ $\left\lvert A \right\rvert$ will denote the corresponding measure of $A$. For any metric ball $B_r \subset \mathbb{E}^{n}$ with $r>0,$ there exists a constant $c_{n} = c_{n}(n)>0$ such that 
	\begin{align}\label{measure of balls}
		\left\lvert B_{r} \right\rvert = c_n r^{Q_{\mathbb{E}}}. 
	\end{align}
	\textbf{Average and Mean:} For $A \subset \mathbb{E}^{n}$ measurable with $\left\lvert A \right\rvert >0$ and a measurable function $g:A \rightarrow \mathbb{R}^{N}$, taking values in any Euclidean space $\mathbb{R}^{N},$ $N \geq 1, $ its integral average will be denoted by the notation
	\begin{align*}
		\left(g\right)_{A} := \fint_{A} g\left(x\right)\ \mathrm{d}x := \frac{1}{\left\lvert A \right\rvert} \int_{A} g\left(x\right)\ \mathrm{d}x . 
	\end{align*} 
	If $A=B_R(x_0)$ for $x_0\in \mathbb{E}^n$ and $R>0$, then $\left(g\right)_{x_0, R}$ and $\left(g\right)_A$ denotes the same quantity. If $x_0$ is fixed in the context, then $\left(g\right)_{x_0, R}$ and $\left(g\right)_{R}$ denotes the same quantity. An important property of the integral averages that we would use throughout the rest 
	is 
	\begin{align}\label{minimality of mean}
		\left( \fint_{A} \left\lvert g - \left(g\right)_{A}\right\rvert^{q}\ \mathrm{d}x \right)^{\frac{1}{q}} \leq c_{n} 	\left( \fint_{A} \left\lvert g - \xi \right\rvert^{q}\ \mathrm{d}x \right)^{\frac{1}{q}}, 
	\end{align} 
	for some constant $c_{n}>0$, depending only on $n$, for any $ \xi \in \mathbb{R}^{N} $ and any $ 1 \leq q < \infty.$ 
	  
	\textbf{Lebesgue Spaces:}  All integrals and $L^{p}$ spaces would be defined with respect to the bi-invariant Haar measure on $\mathbb{E}^{n}$. Weak derivatives are defined the usual way. For Sobolev spaces, we use the notation 
	\begin{align*}
		\mathbb{E}W^{1, \mathfrak{p}(\cdot)}\left(\Omega\right) = \left\lbrace \begin{aligned}
			&W^{1, \mathfrak{p}(\cdot)}\left(\Omega\right) &&\text{ if } \mathbb{E}^{n} = \mathbb{R}^{n} ,\\
			&HW^{1, \mathfrak{p}(\cdot)}\left(\Omega\right) &&\text{ if } \mathbb{E}^{n} = \mathbb{H}_{n}. 
		\end{aligned}\right. 
	\end{align*}
	Similar notation scheme will be used for $\mathbb{E}W_{0}^{1, \mathfrak{p}(\cdot)}\left(\Omega\right)$ and $\mathbb{E}W_{\text{loc}}^{1, \mathfrak{p}(\cdot)}\left(\Omega\right).$\smallskip 
	
	\textbf{Modulus of Continuity:} For any uniformly continuous function $a:\Omega \rightarrow \mathbb{R},$ the modulus of continuity of $a$, denoted $\omega_{a, \Omega}$ is a concave nondecreasing function $\omega_{a, \Omega}: [0, \operatorname*{diam} \Omega ) \rightarrow [0, \infty),$ defined by 
	\begin{align*}
		\omega_{a, \Omega}\left( r \right):=  \sup\limits_{\substack{x,y \in \Omega,\\ d_{\mathbb{E}}(x,y) \leq r }} \left\lvert a\left(x\right) - a\left( y\right)\right\rvert \qquad \text{ for all  } 0 < r < \operatorname*{diam}\Omega ,
	\end{align*}
	and where $\operatorname*{diam}\Omega$ denotes the (metric) diameter of $\Omega$ and we set $\omega_{a, \Omega}\left(0\right)= 0$ by convention. We would often denote the modulus of continuity by $\omega_{a}$ when the domain is clear and also just $\omega$ when the function is also clear from the context. 
	We say $a:\Omega \rightarrow \mathbb{R}$ is \emph{Dini-continuous} if $a$ is uniformly continuous and 
	\begin{align*}
		\int_{0}^{\frac{1}{2}\operatorname*{diam}\Omega}\omega_{a}\left(\rho\right)\frac{\mathrm{d}\rho}{\rho} < \infty.  
	\end{align*}
	More generally, for $\beta \in \lbrace 0, 1 \rbrace,$ we say $a:\Omega \rightarrow \mathbb{R}$ is \emph{$\log^{\beta}$-Dini continuous} if $a$ is uniformly continuous and 
	\begin{align}\label{log beta Dini}
		\int_{0}^{\frac{1}{2}\operatorname*{diam}\Omega} \omega_{\mathfrak{p}}\left(\rho\right)\log^{\beta} \left(\frac{1}{\rho}\right)\ \frac{\mathrm{d}\rho}{\rho} < \infty.  
	\end{align}
	This reduces to Dini-continuity when $\beta=0$.	As usual, a uniformly continuous function $a:\Omega \rightarrow \mathbb{R}$ is called 
	\begin{itemize}
		\item \emph{H\"{o}lder-continuous} with exponent $\alpha \in (0,1)$ if there exists $C>0$ such that 
		\begin{align*}
			\omega_{a}\left(\rho\right)\le C\rho^{\alpha} \qquad \text{ for all } 0 < \rho < \operatorname{diam} \Omega.
		\end{align*}
		\item \emph{Lipschitz continuous} if if there exists $C>0$ such that 
		\begin{align*}
			\omega_{a}\left(\rho\right)\le C\rho \qquad \text{ for all } 0 < \rho < \operatorname{diam} \Omega.  
		\end{align*}
	\end{itemize} 
	The smallest such constants are denoted by $\left[\mathfrak{a}\right]_{C^{0, \alpha}}$ and $\operatorname{Lip}\left({\mathfrak{a}}\right)$, respectively.  
		
	\textbf{The exponent $\gamma$:}	Given two real numbers $1 < \gamma_{1} \le \gamma_{2} < \infty$ and a homogeneous dimension $Q_{\mathbb{E}},$ we set 
	
	\begin{align}\label{def gamma}
		\gamma:= \min \left\lbrace 1 + \frac{1}{2}\left[ \min \left\lbrace \gamma_{1}, \frac{\gamma_{2}}{\gamma_{2}-1}\right\rbrace - 1 \right], \frac{2Q_{\mathbb{E}}}{2Q_{\mathbb{E}}-1} \right\rbrace. 	
	\end{align}
	
	\section{Preliminaries}\label{prelim}
	\subsection{Algebraic lemmas}
	In this subsection, we shall record some elementary algebraic facts. As is standard in the literature, given any real number $1 < p < \infty,$ we shall use the following auxiliary mapping $V_{p}: \mathbb{R}^{2n} \rightarrow \mathbb{R}^{2n}$ 
	defined by $V_{p}(z) : = \left\lvert z \right\rvert^{(p-2)/2} z $, which is a locally Lipschitz bijection from $\mathbb{R}^{2n}$ into itself.  
	\begin{lemma}\label{prop of V}
		For any $p >1$ there exists a constant $c_{V} \equiv c_{V}(n,p) > 0$ such that 
		\begin{equation}\label{constant cv}
			\frac{\left\lvert z_{1} - z_{2}\right\rvert}{c_{V}} \leq  \frac{\left\lvert V_{p}(z_{1}) - V_{p}(z_{2})\right\rvert}{\left( \left\lvert z_{1} \right\rvert + 
				\left\lvert z_{2}\right\rvert \right)^{\frac{p-2}{2}}} \leq c_{V}\left\lvert z_{1} - z_{2}\right\rvert,
		\end{equation}
		for any $z_{1}, z_{2} \in \mathbb{R}^{2n},$ not both zero. This implies the classical monotonicity estimate 
		\begin{equation}\label{monotonicity}
			\left( \left\lvert z_{1} \right\rvert + \left\lvert z_{2}\right\rvert \right)^{p-2} \left\lvert z_{1} - z_{2}\right\rvert^{2} \leq c(n,p) 
			\left\langle  \left\lvert z_{1} \right\rvert^{p-2} z_{1} - \left\lvert z_{2} \right\rvert^{p-2} z_{2}, z_{1} - z_{2} \right\rangle ,  
		\end{equation}
		with a constant $c(n,p) > 0$ for all $p >1$ and all $z_{1}, z_{2} \in \mathbb{R}^{2n}.$
		Moreover, if $1 < p \leq 2,$ there exists a constant $c \equiv c(n,p) > 0$ such that for any  $z_{1}, z_{2} \in \mathbb{R}^{2n},$
		\begin{equation}\label{v estimate p less 2}
			\left\lvert z_{1} - z_{2}\right\rvert \leq c \left\lvert V_{p}(z_{1}) - V_{p}(z_{2})\right\rvert^{\frac{2}{p}} + c \left\lvert V_{p}(z_{1}) - V_{p}(z_{2})\right\rvert 
			\left\lvert z_{2}\right\rvert^{\frac{2-p}{2}}.
		\end{equation}
	\end{lemma}
	For proofs, we refer \cite[Lemma 2.1]{hamburgerregularity} for \eqref{constant cv} and \eqref{monotonicity}  and \cite[Lemma 2]{KuusiMingione_nonlinearStein} for \eqref{v estimate p less 2}. 

	\subsection{Variable exponent spaces on \texorpdfstring{$\mathbb{E}^{n}$}{En}}\label{variable exponent spaces}
	Let $1 < \gamma_{1} \leq \gamma_{2} < \infty.$ 
	\begin{definition}
		Let $\mathfrak{p}:\Omega \rightarrow [\gamma_{1}, \gamma_{2}], $ uniformly continuous function. We say 
		\begin{itemize}
			\item $\mathfrak{p}$ is \emph{log-H\"{o}lder continuous},  denoted by $\mathfrak{p} \in \mathcal{P}^{\log}_{[\gamma_{1}, \gamma_{2}]}\left(\Omega \right),$ if  
			\begin{align}\label{log holder def}
				\lim\limits_{R \rightarrow 0} \omega_\mathfrak{p}\left(R\right)\log \left(\frac{1}{R}\right) < +\infty.  
			\end{align}
			The finite value of this limit is called the \emph{log-H\"{o}lder} constant of $\mathfrak{p}$ on $\Omega$.
			\item $\mathfrak{p}$ is \emph{vanishing log-H\"{o}lder continuous} if 
			\begin{align}\label{log holder vanishing}
				\lim\limits_{R \rightarrow 0} \omega_{\mathfrak{p}}\left(R\right)\log \left(\frac{1}{R}\right) = 0 . 
			\end{align}
			\item  $\mathfrak{p}$ is \emph{log-Dini continuous} if there exists $R>0$ such that we have 
			\begin{align}\label{log Dini}
				\int_{0}^{\frac{1}{2}\operatorname*{diam}\Omega} \omega_{\mathfrak{p}}\left(\rho\right)\log \left(\frac{1}{\rho}\right)\ \frac{\mathrm{d}\rho}{\rho} < \infty.  
			\end{align}	
		\end{itemize}
	\end{definition}
	Note that \eqref{log beta Dini} reduces to log-Dini continuity when $\beta=1.$ These definitions are standard and indeed defined analogously to the Euclidean setting (see \cite{Diening_et_al_variable_exponent} and \cite[Section 3.1.3]{Mallick_Sil_continuityvariablegrowth} for details) when $\mathbb{E}^{n} = \mathbb{H}_{n}.$  Furthermore, as noted in \cite[Remark 12]{Mallick_Sil_continuityvariablegrowth}, all the three concepts in $\mathbb{H}_n$ are equivalent to corresponding ones in $\mathbb{R}^{2n+1}$	.
	
	\begin{definition}
		The space $L^{\mathfrak{p}(\cdot)}\left( \Omega\right)$ with exponent $\mathfrak{p} \in \mathcal{P}^{\log}_{[\gamma_{1}, \gamma_{2}]}\left(\Omega \right)$ is defined as 
		\begin{align*}
			L^{\mathfrak{p}(\cdot)}\left( \Omega\right):= \left\lbrace u:\Omega \rightarrow \mathbb{R} \text{ measurable}: \int_{\Omega} \left\lvert u\left(x\right)\right\rvert^{\mathfrak{p}\left(x\right)}\ \mathrm{d}x < \infty\right\rbrace. 
		\end{align*}
	\end{definition}
	Equipped with the usual Luxemburg norm $\left\lVert \cdot \right\rVert_{L^{\mathfrak{p}(\cdot)}\left( \Omega\right)}$, $L^{\mathfrak{p}(\cdot)}\left( \Omega\right)$ is a reflexive, separable Banach space (cf. \cite{Diening_et_al_variable_exponent}). The definition extends in the obvious manner to vector-valued functions. When $\mathbb{E}^n= \mathbb{R}^n$, the variable exponent Sobolev spaces are defined as 
	\begin{align}\label{W1px norm}
		W^{1, \mathfrak{p}(\cdot)}\left( \Omega \right):= \left\lbrace u \in L^{\mathfrak{p}(\cdot)}\left( \Omega\right): \left\lVert u \right\rVert_{L^{\mathfrak{p}(\cdot)}\left( \Omega\right)} + \left\lVert \nabla u \right\rVert_{L^{\mathfrak{p}(\cdot)}\left(\Omega; \mathbb{R}^{n}\right)}  < \infty \right\rbrace,
	\end{align}
	where $u$ is implicitly assumed to be weakly differentiable in $U$ and $\nabla u$ stands for the weak gradient of $u.$ The space $	HW^{1, \mathfrak{p}(\cdot)}_{0}\left( \Omega \right)$ denotes the closure of $C^{1}_{c}\left(\Omega\right)$ under the norm \eqref{W1px norm}. When $\mathbb{E}^n= \mathbb{H}_n$, the horizontal variable exponent Sobolev space is 
	\begin{align*}
		HW^{1, \mathfrak{p}(\cdot)}\left( \Omega \right):= \left\lbrace u \in L^{\mathfrak{p}(\cdot)}\left( \Omega\right): \left\lVert u \right\rVert_{L^{\mathfrak{p}(\cdot)}\left( \Omega\right)} + \left\lVert \X u \right\rVert_{L^{\mathfrak{p}(\cdot)}\left(\Omega; \mathbb{R}^{2n}\right)}  < \infty \right\rbrace,
	\end{align*}
    equipped with the norm 
	\begin{align}\label{HW1px norm}
		\left\lVert u \right\rVert_{HW^{1, \mathfrak{p}(\cdot)}\left( \Omega \right)}:= \left\lVert u \right\rVert_{L^{\mathfrak{p}(\cdot)}\left( \Omega\right)} + \left\lVert \X u \right\rVert_{L^{\mathfrak{p}(\cdot)}\left(\Omega; \mathbb{R}^{2n}\right)},
	\end{align}
	where $\X u$ denotes the weak horizontal gradient.  The space $	HW^{1, \mathfrak{p}(\cdot)}_{0}\left( \Omega \right)$ denotes the closure of $C^{1}_{c}\left(\Omega\right)$ under the norm \eqref{HW1px norm}. \emph{We would not assume $\mathfrak{p}$ to be log-H\"{o}lder continuous in $\Omega$. However, our assumption will always imply (see Lemma \ref{local vanishing log holder}) that $\mathfrak{p} \in \mathcal{P}^{\log}_{[\gamma_{1}, \gamma_{2}]}\left(\Omega' \right),$ for every compactly contained open subdomain $\Omega' \subset \subset \Omega$. Thus, we interpret  the notation $ u \in \mathbb{E}W_{\text{loc}}^{1, \mathfrak{p}(\cdot)}\left(\Omega\right)$  as meaning $u \in \mathbb{E}W^{1,\mathfrak{p}(\cdot)}\left(\Omega'\right)$ for every compactly contained open subdomain $\Omega' \subset \subset \Omega$.}	As we are concerned with interior estimates, here onwards we would often disregard this technical point.  
	\subsection{Poincar\'{e}-Sobolev inequalities}
	\begin{proposition}[Poincar\'{e}-Sobolev inequality with means]\label{poincaresobolevwithmeans}
		Let $B_{R} \subset \mathbb{E}^{n}$ be any ball of radius $R>0.$ Let $1 \leq s <  Q_{\mathbb{E}}.$ Then there exists a constant $c >0,$ depending only on $n$ and $s$ such that for any $u \in \mathbb{E}W^{1,s}\left( B_{R} \right),$ we have
		\begin{equation}\label{poincaresobolevineqwithmeans}
			\left( \fint_{B_{R}} \lvert u - \left( u\right)_{B_{R}} \rvert^{\frac{Q_{\mathbb{E}}s}{Q_{\mathbb{E}}-s}} \right)^{\frac{Q_{\mathbb{E}}-s}{Q_{\mathbb{E}}s}} \leq c R \left( \fint_{B_{R}} \lvert \nabla_{\mathbb{E}} u \rvert^{s} \right)^{\frac{1}{s}}. 
		\end{equation}
	\end{proposition}
	The Heisenberg group version of \eqref{poincaresobolevineqwithmeans} follows from \cite{lu-poincare-sobolev},  and \cite{Garofalo-Nhieu-Global-Approximation}.  
	\begin{remark}\label{p uniform constants Sobolev inequalities}
		The constant above depends on $s.$ However, the proofs are constructive and the constant is explicitly quantifiable. Thus, if $1 < \gamma_{1} \leq s \leq \gamma_{2}< \infty,$ the constants can all be chosen (by replacing them with larger or smaller constants if necessary) to depend only $\gamma_{1}, \gamma_{2}$ such that the estimates are uniform in $s \in [\gamma_{1}, \gamma_{2}].$ 
	\end{remark}	
	\subsection{Lorentz-Zygmund spaces}
	\subsubsection{Lorentz spaces:} We start by defining Lorentz spaces. See \cite{Hunt_LorentzSpaces} and \cite{SteinWeiss_Fourieranalysis}.
	\begin{definition}
		For $1\leq p <  \infty,$ and $0 < q \leqslant \infty,$ A measurable function $u: \Omega \rightarrow \mathbb{R}$ is said to belong to the Lorentz space $L^{(p,q)} \left(\Omega\right)$ if 
		$$ \left\lVert u \right\rVert_{L^{(p,q)}\left(\Omega\right)}^{q} 
		:= \int_{0}^{\infty} \left( t^{p} \left\lvert \left\lbrace x \in \Omega : \left\lvert u \right\rvert > t \right\rbrace\right\rvert \right)^{\frac{q}{p}} \frac{\mathrm{d}t}{t} < \infty $$
		for $0 < q < \infty$ and to $L^{(p,\infty)} \left(\Omega\right)$ if 
		$$ \left\lVert u \right\rVert_{L^{(p,\infty)}\left(\Omega\right)}^{p} 
		:= \sup_{t >0} \left( t^{p} \left\lvert \left\lbrace x \in \Omega : \left\lvert u \right\rvert > t \right\rbrace\right\rvert \right) < \infty .$$
	\end{definition}

	
	\subsubsection{$L\log^{\beta}L$ spaces:}
	For $\beta \geq 0 ,$ the Orlicz space $L\log^{\beta}L \left( \Omega; \mathbb{R}^{m}\right)$ is defined as 
	\begin{align*}
		L\log^{\beta}L \left( \Omega; \mathbb{R}^{m}\right): = \left\lbrace f \in L^{1}\left( \Omega; \mathbb{R}^{m}\right): \int_{\Omega} \left\lvert f \right\rvert \log^{\beta}\left( e + \left\lvert f \right\rvert \right)\ \mathrm{d}x < \infty \right\rbrace.  
	\end{align*}
	This is a Banach space equipped with the Luxemburg norm $\lVert \cdot \rVert_{L\log^{\beta}L \left( \Omega\right)}$. Iwaniec (\cite{IwaniecPharmonic}, \cite{Iwaniec_et_al_MappingsBMOboundeddistortion}, \cite{IwaniecVerde_LlogL}) showed that this norm is equivalent to the following 
	\begin{align*}
		\left[ f\right]_{L\log^{\beta}L \left( \Omega\right)}:= \fint_{\Omega} \left\lvert f \right\rvert \log^{\beta}\left( e + \frac{\left\lvert f \right\rvert}{\fint_{\Omega}\left\lvert f \right\rvert\ \mathrm{d}x} \right)\ \mathrm{d}x. 
	\end{align*}
	For any $p >1,$ $L^{p}\left(\Omega; \mathbb{R}^{m}\right)$ embeds into $L\log^{\beta}L \left( \Omega; \mathbb{R}^{m}\right)$ and we have the following inequality (cf. estimate $(28)$ in \cite{Acerbi_Mingione_CZforpxLaplacian}) 
	\begin{align*}
		\fint_{\Omega} \left\lvert f \right\rvert \log^{\beta}\left( e + \frac{\left\lvert f \right\rvert}{\fint_{\Omega}\left\lvert f \right\rvert\ \mathrm{d}x} \right)\ \mathrm{d}x \leq c\left(p, \beta\right) \left( \fint_{\Omega} \left\lvert f \right\rvert^{p}\ \mathrm{d}x\right)^{\frac{1}{p}}. 
	\end{align*}
	In the same spirit, we also have the following lemma (see \cite[Lemma 2.1]{Baroni_Coscia}). 
	\begin{lemma}\label{Llog beta L lemma}
		Let $\zeta >1$, $\delta, \beta, \tau \geq 0$ and $f$ be a nonnegative function in $L^{\zeta}\left(B_{R}\right)$ for some ball $B_{R}$ with $R \leq 1/e.$ Then there exists a constant $c= c(n, \beta, \delta, \tau, \zeta) >0$ such that 
		\begin{align*}
			\fint_{B_{R}}  f \log^{\beta}\left( e + f^{\delta} \right)\ \mathrm{d}x \leq c \left( 1 + R^{\tau} \left\lVert f \right\rVert_{L^{1}\left(B_{R}\right)}\right)^{\beta} \log^{\beta}\left(\frac{1}{R}\right)\left( \fint_{B_{R}} \left\lvert f \right\rvert^{\zeta}\ \mathrm{d}x\right)^{\frac{1}{\zeta}}. 
		\end{align*}
	\end{lemma}
	\subsubsection{$L^{p,q}\log L$ spaces:}
	A good reference for these materials is \cite[Chapter 9]{Pick_Kufner_Fucik_Function_spaces}. 	Let $\Omega \subset \mathbb{E}^{n}$ be a measurable set and let $g: \Omega \rightarrow \mathbb{R}$ be a measurable function. Let $\mu_{f}: [0, +\infty) \rightarrow [0, +\infty)$ denote the  distribution function of $f$, and $f^{\ast}: [0, +\infty) \rightarrow [0, +\infty]$ denote the nonincreasing rearrangement of $f$. 	\begin{definition} Let $1\leq s < \infty,$ $\mu \in (0, \infty]$ and $\beta \in \mathbb{R}.$ Given a measurable function $f:\Omega \rightarrow \mathbb{R},$ we define its  $L^{(s, \mu)}\log^{\beta}L \left(\Omega\right)$ quasinorm as 
		\begin{align*}
			\left\vvvert f \right\vvvert_{L^{(s, \mu)}\log^{\beta}L \left(\Omega\right)} := \left\lbrace \begin{aligned}
				&\left( \int_{0}^{\left\lvert \Omega \right\rvert} \left[ \rho^{\frac{1}{s}}\left( 1+ \left\lvert \log \rho \right\rvert\right)^{\beta}f^{\ast}\left(\rho\right)\right]^{\mu}\ \frac{\mathrm{d}\rho}{\rho}\right)^{\frac{1}{\mu}} &&\text{ if } \mu < \infty, \\
				&\sup\limits_{\rho \in \left(0, \left\lvert \Omega \right\rvert\right)}  \rho^{\frac{1}{s}}\left( 1+ \left\lvert \log \rho \right\rvert\right)^{\beta}f^{\ast}\left(\rho\right) &&\text{ if } \mu = \infty. 
			\end{aligned}\right. 
		\end{align*}
		We say $f \in L^{(s, \mu)}\log^{\beta}L \left(\Omega\right)$ if $	\left\vvvert f \right\vvvert_{L^{(s, \mu)}\log^{\beta}L \left(\Omega\right)}  < \infty.$
	\end{definition}
	This quasinorm  is equivalent to the following norm when $s>1$ and $\mu <\infty$.     
	\begin{align*}
		\left\lVert f \right\rVert_{L^{(s, \mu)}\log^{\beta}L \left(\Omega\right)} := \left( \int_{0}^{\left\lvert \Omega \right\rvert} \left[ \rho^{\frac{1}{s}}\left( 1+ \left\lvert \log \rho \right\rvert\right)^{\beta}f^{\ast\ast}\left(\rho\right)\right]^{\mu}\ \frac{\mathrm{d}\rho}{\rho}\right)^{\frac{1}{\mu}}, 
	\end{align*}
	where $f^{\ast\ast} \left( r\right) = \fint_{0}^{r} f^{\ast}$, for $r \in (0, \infty)$. Moreover, the norm is absolutely continuous, i.e. for any $f \in L^{(s, \mu)}\log^{\beta}L \left(\Omega\right)$ and $E \subset \Omega$ measurable,  we have 
	\begin{align}\label{abs cont Lorentz zygmund norm}
		\left\lVert f \right\rVert_{L^{(s, \mu)}\log^{\beta}L \left(E\right)} \rightarrow 0 \qquad \text{ as } \qquad \left\lvert E \right\rvert \rightarrow 0. 
	\end{align}
		\subsection{Some important series}\label{series section}
		Let $m \in \mathbb{N}$ and let $x_{0} \in \mathbb{E}^{n}$.  Let $r>0$, $\tau \in (0, 1/4)$,  $1 < \theta < Q_{\mathbb{E}}$ and $1 < q < \infty$ be real numbers and let $\beta \in \left\lbrace 0, 1 \right\rbrace.$ Given any $g \in L^{\theta}\left( B_{r}(x_{0}); \mathbb{R}^{m} \right),$ we define 
		\begin{align*}
			\mathcal{S}_{\left[g, \theta, \tau, \beta \right]}\left(x_{0}, r \right) &:= \sum\limits_{j=0}^{\infty} r_{j} \log^{\beta} \left(\frac{1}{r_{j}}\right)\left(  \fint_{B_{r_{j}}\left(x_{0}\right)}\left\lvert   g \right\rvert^{\theta} \right)^{\frac{1}{\theta}} \end{align*} and given any $h \in L^{q}\left( B_{r}(x_{0}); \mathbb{R}^{m} \right),$ we set 
		\begin{align*}
			\mathfrak{S}_{\left[h, q, \tau, \beta\right]}\left(x_{0}, r \right) &:= \sum\limits_{j=0}^{\infty} \log^{\beta} \left(\frac{1}{r_{j}}\right)\left(  \fint_{B_{r_{j}}\left(x_{0}\right)}\left\lvert   h - \left(h\right)_{B_{r_{j}}\left(x_{0}\right)} \right\rvert^{q} \right)^{\frac{1}{q}} ,
		\end{align*}
		where $r_{j} = \tau^{j}r.$ The first series is summable when $g$ is in a Lorentz-Zygmund class. 
		\begin{lemma}\label{Lorentz series estimate lemma combined}
			Let $\beta \in \left\lbrace 0, 1 \right\rbrace$ and $g \in L^{(Q_{\mathbb{E}},1)}\log^{\beta} L\left(\Omega; \mathbb{R}^{m}\right).$  Then for any $1 < \theta < Q_{\mathbb{E}}$ and any $\tau \in (0, 1/4),$ there exists a constant $c \equiv c\left(Q_{\mathbb{E}}, \theta, \tau ,m\right)>0$ such that  
			\begin{align}\label{Lorentz series estimate combined}
				\mathcal{S}_{\left[ g, \theta, \tau, \beta\right]}\left( x_{0}, r \right) \leq c \left\lVert g \right\rVert_{L^{(Q_{\mathbb{E}},1)}\log^{\beta} L\left(B_{r}(x_0); \mathbb{R}^{m}\right)}, 
			\end{align}
			whenever $B_{r}(x_0) \subset \Omega.$
		\end{lemma}
		For $\beta=0,$ Lemma \ref{Lorentz series estimate lemma combined} is proved in \cite[Lemma 1]{KuusiMingione_nonlinearStein} and the proof can easily be adopted for $\beta=1$ (see also \cite{Baroni_pxStein} and \cite[Lemma 2.2]{Baroni2025}). 
		\begin{corollary}\label{vanishing corollary derivative version}
			Let $\beta \in \left\lbrace 0, 1 \right\rbrace$ and $g \in L^{(Q_{\mathbb{E}},1)}\log^{\beta} L\left(\Omega; \mathbb{R}^{m}\right).$ Then given any $\varepsilon>0$, $1< \theta < Q_{\mathbb{E}}$, $\tau \in (0, 1/4)$ and any compact subset $K \subset \Omega,$ there exists a radius $r_{\varepsilon} = r_{\varepsilon}\left(K, \theta, \tau, \left\lVert g \right\rVert_{L^{(Q_{\mathbb{E}},1)}\log^{\beta} L\left(\Omega\right)}\right) >0$ such that if $0 <r \le r_{\varepsilon},$ we have 
			\begin{align*}
				\mathcal{S}_{\left[ g, \theta, \tau, \beta\right]}\left( x, r \right) < \varepsilon \qquad \text{ for a.e. } x \in K .
			\end{align*}
		\end{corollary}
		\begin{proof}
			Observe that since the norm on the left  in \eqref{Lorentz series estimate combined} is absolutely continuous,  $\mathcal{S}_{\left[ g, \theta, \tau, \beta\right]}\left( x_{0}, r \right) \rightarrow 0$ as $r \rightarrow 0$, for a.e. $x_{0} \in \Omega.$ Moreover, as the constant is independent of $x_{0},$ the convergence is locally uniform in $x_{0} \in \Omega.$
		\end{proof}
		Set 
		\begin{align*}
			 m_{\mathbb{E}}=	\left\lbrace \begin{aligned}
				&n &&\text{ if } \mathbb{E}^{n} = \mathbb{R}^{n}, \\
				&2n &&\text{ if } \mathbb{E}^{n} = \mathbb{H}_{n}. 
			\end{aligned}\right. \qquad \text{ and } \qquad \theta^{\ast} := \frac{Q_{\mathbb{E}}\theta}{Q_{\mathbb{E}}- \theta}. 
		\end{align*}
		\begin{proposition}\label{derivative to summability}
			Let $\beta \in \left\lbrace 0, 1 \right\rbrace$ and $g$ be a function in $L^{1}\left(\Omega\right)$ such that $\nabla_{\mathbb{E}}g \in L^{(Q_{\mathbb{E}},1)}\log^{\beta} L\left(\Omega; \mathbb{R}^{m_{\mathbb{E}}}\right).$ Then for any $1 < \theta < Q_{\mathbb{E}}$ and any $\tau \in (0, 1/4),$ we have the estimates 
			\begin{align*}
				\mathfrak{S}_{\left[g, \theta^{\ast}, \tau, \beta\right]}\left(x_{0}, r \right) \leq c	\mathcal{S}_{\left[\nabla_{\mathbb{E}}g, \theta, \tau, \beta\right]}\left(x_{0}, r \right) \le c \left\lVert \nabla_{\mathbb{E}}g \right\rVert_{L^{(Q_{\mathbb{E}},1)}\log^{\beta} L\left(B_{r}(x_0); \mathbb{R}^{m_{\mathbb{E}}}\right)},
			\end{align*}
			whenever $B_{r}(x_0) \subset \Omega.$ In particular, for any $1 < q < \infty$ and $\tau \in (0, 1/4),$ we have 
			\begin{align*}
				\mathfrak{S}_{\left[g, q, \tau, \beta\right]}\left(x_{0}, r \right) \rightarrow 0 \quad \text{ as } r \rightarrow 0 \qquad  \text{ locally uniformly in } x_{0} \in \Omega.	
			\end{align*}
		\end{proposition}
		\begin{proof}
			Clearly the hypothesis implies $g \in \mathbb{E}W^{1, \theta}\left(\Omega\right)$ for any $1 < \theta < Q_{\mathbb{E}}.$ But then by the Poincar\'{e}-Sobolev inequality (see Proposition \ref{poincaresobolevwithmeans}), for every $j \geq 0,$ we have 
			\begin{align*}
				\left(  \fint_{B_{r_{j}}\left(x_{0}\right)}\left\lvert   g - \left(g\right)_{B_{r_{j}}\left(x_{0}\right)} \right\rvert^{\theta^{\ast}} \right)^{\frac{1}{\theta^{\ast}}} 
				\le C r_{j} \left(  \fint_{B_{r_{j}}\left(x_{0}\right)}\left\lvert   \nabla_{\mathbb{E}}g \right\rvert^{\theta} \right)^{\frac{1}{\theta}}.
			\end{align*}
			The rest is just Lemma \ref{Lorentz series estimate lemma combined}. For the last part, choose $\theta$ such that $q \le \theta^{\ast}.$
		\end{proof}
		Continuity properties can also imply summability of these series. 
		\begin{proposition}\label{dini to summability}
			Let $\beta \in \left\lbrace 0, 1 \right\rbrace$ and let $g$ be $\log^{\beta}$-Dini continuous in $\Omega.$ Then for any $1 < q < \infty$ and any $\tau \in (0, 1/4),$ $\mathfrak{S}_{\left[g, q, \tau, \beta \right]}\left(x_{0}, r \right)$ is finite whenever $r \in (0,1/4)$ and $B_{4r}(x_0) \subset \Omega.$  Furthermore,  
			\begin{align*}
				\mathfrak{S}_{\left[g, q, \tau, \beta\right]}\left(x_{0}, r \right) \rightarrow 0 \quad \text{ as } r \rightarrow 0 \qquad  \text{ locally uniformly in } x_{0} \in \Omega.	
			\end{align*}
		\end{proposition}
		\begin{proof}
			By definition of the modulus of continuity, we have 
			\begin{align*}
				\mathfrak{S}_{\left[g, q, \tau, \beta\right]}\left(x_{0}, r \right) &\leq  \sum\limits_{j=0}^{\infty} \log^{\beta} \left(\frac{1}{r_{j}}\right)\omega_{g}\left(r_{j}\right).
			\end{align*}
			Now for the case $\beta =0$, we have 
			\begin{align*}
				\int_{0}^{2r} \omega_{g}\left(\rho\right)\frac{\mathrm{d} \rho}{\rho}
				&= \sum\limits_{j=0}^{\infty}\int_{r_{j+1}}^{r_{j}} \omega_{g}\left(\rho\right)\frac{\mathrm{d} \rho}{\rho} + \int_{r}^{2r} \omega_{g}\left(\rho\right)\frac{\mathrm{d} \rho}{\rho}\\
				&\ge \sum\limits_{j=0}^{\infty} \omega_{g}\left(r_{j+1}\right)\int_{r_{j+1}}^{r_{j}} \frac{\mathrm{d} \rho}{\rho} + \omega_{g}\left(r\right)\log 2 \\
				&=\log \left(\frac{1}{\tau}\right)\sum\limits_{j=0}^{\infty} \omega_{g}\left(r_{j+1}\right) + \omega_{g}\left(r\right)\log 2 \ge \log 2 \sum\limits_{j=0}^{\infty} \omega_{g}\left(r_{j}\right). 
			\end{align*}
			This implies the estimate
			\begin{align*}
				\mathfrak{S}_{\left[g, q, \tau, 0\right]}\left(x_{0}, r \right) \le \frac{1}{\log 2}\int_{0}^{2r} \omega_{g}\left(\rho\right)\frac{\mathrm{d} \rho}{\rho},
			\end{align*}
			from which the result follows. For $\beta=1,$ we estimate 
			\begin{align*}
				\int_{0}^{2r}& \omega_{g}\left(\rho\right)\log \left(\frac{1}{\rho}\right)\frac{\mathrm{d} \rho}{\rho} \\
				&= \sum\limits_{j=0}^{\infty}\int_{r_{j+1}}^{r_{j}} \omega_{g}\left(\rho\right)\log \left(\frac{1}{\rho}\right)\frac{\mathrm{d} \rho}{\rho} + \int_{r}^{2r} \omega_{g}\left(\rho\right)\log \left(\frac{1}{\rho}\right)\frac{\mathrm{d} \rho}{\rho}\\
				&\ge \sum\limits_{j=0}^{\infty} \omega_{g}\left(r_{j+1}\right) \log \left(\frac{1}{r_{j}}\right)\int_{r_{j+1}}^{r_{j}}\frac{\mathrm{d} \rho}{\rho} + \omega_{g}\left(r\right)\log \left(\frac{1}{2r}\right)\int_{r}^{2r} \frac{\mathrm{d} \rho}{\rho}\\
				&\ge \log \left(\frac{1}{\tau}\right)\sum\limits_{j=0}^{\infty} \omega_{g}\left(r_{j+1}\right)\log \left(\frac{1}{r_{j}}\right) + \omega_{g}\left(r\right)\log \left(\frac{1}{2r}\right)\log 2\\
				&\begin{aligned}[t]
					\ge \log \left(\frac{1}{\tau}\right)\sum\limits_{j=0}^{\infty} \omega_{g}\left(r_{j+1}\right)&\left[ \log \left(\frac{1}{r_{j+1}}\right) -\log \left(\frac{1}{\tau}\right) \right] \\ &\quad+ \log 2\cdot  \omega_{g}\left(r\right)\left[\log \left(\frac{1}{r}\right) - \log 2\right]
				\end{aligned}\\
				&\ge\log 2\sum\limits_{j=0}^{\infty} \omega_{g}\left(r_{j}\right) \log \left(\frac{1}{r_{j}}\right) - \log^{2} \left(\frac{1}{\tau}\right)\sum\limits_{j=0}^{\infty} \omega_{g}\left(r_{j}\right). 
			\end{align*}
			Using the estimate for $\beta=0$ for the last term, we obtain 
			\begin{align*}
				\mathfrak{S}_{\left[g, q, \tau, 1\right]}\left(x_{0}, r \right) &\le \frac{1}{\log 2}\int_{0}^{2r} \omega_{g}\left(\rho\right)\log \left(\frac{1}{\rho}\right)\frac{\mathrm{d} \rho}{\rho} + \frac{\log^{2} \left(\frac{1}{\tau}\right)}{\log^{2} 2}  \int_{0}^{2r} \omega_{g}\left(\rho\right)\frac{\mathrm{d} \rho}{\rho} \\
				&\le \left[ \frac{1}{\log 2} +\frac{\log^{2} \left(\frac{1}{\tau}\right)}{\log^{3} 2}\right]\int_{0}^{2r} \omega_{g}\left(\rho\right)\log \left(\frac{1}{\rho}\right)\frac{\mathrm{d} \rho}{\rho},
			\end{align*}
			where we have used $ \log 2 \le \log \left(1/\rho\right)$ for all $0 < \rho \le 2r < 1/2.$ The claimed result now easily follows from this. 
		\end{proof}
		On the other hand, some continuity properties are implied by the local uniform convergence of these sums. 
		\begin{lemma}\label{local vanishing log holder}
			Let $\Omega \subset \mathbb{E}^{n}$ be open and bounded and let $\beta \in \left\lbrace 0, 1 \right\rbrace$. Let $g \in L^{Q_{\mathbb{E}}}\left(\Omega\right)$ be such that for some $1 < q < \infty$ and some $\tau \in (0, 1/4),$ we have 
			\begin{align*}
				\mathfrak{S}_{\left[g, q, \tau, \beta\right]}\left(x_{0}, r \right) \rightarrow 0 \quad \text{ as } r \rightarrow 0 \qquad  \text{ locally uniformly in } x_{0} \in \Omega.	
			\end{align*}
			Then $g$ is continuous in $\Omega.$ Furthermore, for any compact set $K \subset \Omega,$ we have 
			\begin{align*}
				\lim\limits_{\bar{r} \rightarrow 0} \sup\limits_{0 < r \le \bar{r}}\ \omega_{g_{K}}\left(r\right)\log^{\beta} \left(\frac{1}{r}\right) = 0, 
			\end{align*}
			where $g_{K}$ is the uniformly continuous restriction of $g$ to $K$.  
		\end{lemma}
		\begin{proof} First, observe that $\mathfrak{S}_{\left[g, q, \tau, \beta\right]}(x_{0}, r)$ is not necessarily monotone in $r$ when $\beta=1.$ However, if we set 
			\begin{align*}
				\widetilde{\mathfrak{S}}_{\left[g, q, \tau, \beta \right]}(x_{0}, \bar{r}) : =\sup\limits_{0 < r \le \bar{r}} \mathfrak{S}_{\left[g, q, \tau, \beta  \right]}(x_{0}, r),
			\end{align*}
			then $\bar{r} \mapsto 	\widetilde{\mathfrak{S}}_{\left[g, q, \tau, \beta \right]}(x_{0}, \bar{r})$ is nondecreasing by construction. Also, it is easy to see that our hypothesis implies 
			\begin{align*}
				\widetilde{\mathfrak{S}}_{\left[g, q, \tau, \beta\right]}\left(x_{0}, \bar{r} \right) \rightarrow 0 \quad \text{ as } \bar{r} \rightarrow 0 \qquad  \text{ locally uniformly in } x_{0} \in \Omega.	
			\end{align*}
			Thus, replacing $\mathfrak{S}_{\left[g, q, \tau, \beta \right]}\left(x_{0}, r \right)$ by $\widetilde{\mathfrak{S}}_{\left[g, q, \tau, \beta\right]}\left(x_{0}, r \right),$ we can assume without loss of generality that $r \mapsto \mathfrak{S}_{\left[g, q, \tau, \beta\right]}\left(x_{0}, r \right)$ is nondecreasing. Similarly,  we can assume $r \mapsto \omega_{g_{K}}\left(r\right)\log^{\beta} \left(\frac{1}{r}\right)$ is nondecreasing and just prove 
			\begin{align*}
				\lim\limits_{\bar{r} \rightarrow 0} \omega_{g_{K}}\left(\bar{r}\right)\log^{\beta} \left(\frac{1}{\bar{r}}\right) = 0. 
			\end{align*}
            Now we prove the continuity of $g$. For $r < 1/4 < 1/e,$ we trivially have 
			\begin{align*}
				\mathfrak{S}_{\left[g, q, \tau, 0\right]}\left(x_{0}, r \right) \le \mathfrak{S}_{\left[g, q, \tau, 1\right]}\left(x_{0}, r \right). 
			\end{align*} Fix $x \in \Omega$ and choose a compact subset $K \subset \subset \Omega$ such that $x \in K. $ 
			Now we have 
			\begin{align*}
				\left\lvert \left(g\right)_{x, \tau^{j}r} - \left(g\right)_{x, \tau^{m}r}\right\rvert &\le \sum\limits_{i=j}^{m-1} \left\lvert \left(g\right)_{x, \tau^{i}r} - \left(g\right)_{x, \tau^{i+1}r}\right\rvert \\
				&\le \sum\limits_{i=j}^{m-1}\fint_{B_{\tau^{i+1}r}(x)} \left\lvert g - \left(g\right)_{x, \tau^{i}r} \right\rvert \\
				&\le \sum\limits_{i=j}^{m-1}\left(\fint_{B_{\tau^{i+1}r}(x)} \left\lvert g - \left(g\right)_{x, \tau^{i}r} \right\rvert^{q}\right)^{\frac{1}{q}} \\
				&\le \tau^{-\frac{Q_{\mathbb{E}}}{q}}\sum\limits_{i=j}^{m-1}\left(\fint_{B_{\tau^{i}r}(x)} \left\lvert g - \left(g\right)_{x, \tau^{i}r} \right\rvert^{q}\right)^{\frac{1}{q}} \\
				&\le \tau^{-\frac{Q_{\mathbb{E}}}{q}}\sum\limits_{i=j}^{\infty}\left(\fint_{B_{\tau^{i}r}(x)} \left\lvert g - \left(g\right)_{x, \tau^{i}r} \right\rvert^{q}\right)^{\frac{1}{q}} \hspace{-3.3pt}\le \tau^{-\frac{Q_{\mathbb{E}}}{q}}\mathfrak{S}_{\left[g, q, \tau, 0\right]}(x, r_{j}),  
			\end{align*}
			which goes to $0$, as $j,m \rightarrow \infty$ uniformly in $x \in K.$ By the continuity of translation in $L^{1}$, the maps $x \mapsto \left(g\right)_{x, \tau^{j}r}$ are continuous. Hence, by Lebesgue differentation theorem, $g$ agrees a.e. with the uniform limit of continuous maps and thus admits a continuous representative. Now fix $K \subset \Omega$ compact and choose $r \in (0, 1/4)$ such that $B_{4r}(x) \subset \Omega$ for all $x \in K.$  Letting $m \rightarrow \infty$ in the last display, we deduce  
			\begin{align*}
				\left\lvert \left(g\right)_{x, \tau^{j}r} - g(x)\right\rvert \le \tau^{-\frac{Q_{\mathbb{E}}}{q}}\mathfrak{S}_{\left[g, q, \tau, 0\right]}(x, r_{j}) \le \tau^{-\frac{Q_{\mathbb{E}}}{q}}\frac{1}{\log^{\beta} \left(\frac{1}{r_{j}}\right)}\mathfrak{S}_{\left[g, q, \tau, \beta\right]}(x, r_{j}),
			\end{align*}
			for all $x \in K.$ Let $x, y \in K$ with $d_{\mathbb{E}}(x,y) \le \bar{r}$, where $0 < \bar{r} < \tau^{4}r/4.$ Let $j \geq 4$ be an integer such that $\tau^{j+1}r \le \bar{r} < \tau^{j}r.$ Then we have 
			\begin{multline}
			    \label{add substract averages}
				\left\lvert g(x) - g(y)\right\rvert  \le  	\left\lvert g(x) - \left(g\right)_{x, \tau^{j+1}r}\right\rvert + 	\left\lvert \left(g\right)_{x, \tau^{j+1}r} - \left(g\right)_{y, \tau^{j+1}r}\right\rvert \\+ 	\left\lvert g(y) - \left(g\right)_{y, \tau^{j+1}r}\right\rvert.
			\end{multline} 
			Using the assumption that $r \mapsto \mathfrak{S}_{\left[g, q, \tau, \beta \right]}(x, r)$ is monotone nondecreasing, we have 
			\begin{align}\label{estimate for 1st term}
				\left\lvert g(x) - \left(g\right)_{x, \tau^{j+1}r}\right\rvert &\le \frac{\tau^{-\frac{Q_{\mathbb{E}}}{q}}}{\log^{\beta} \left(\frac{1}{r_{j+1}}\right)}\mathfrak{S}_{\left[g, q, \tau, \beta\right]}(x, r_{j+1}) \notag \\
				&\le \frac{\tau^{-\frac{Q_{\mathbb{E}}}{q}}}{\log^{\beta} \left(\frac{1}{\bar{r}}\right)}\mathfrak{S}_{\left[g, q, \tau, \beta \right]}(x, r_{j+1}) \le \frac{\tau^{-\frac{Q_{\mathbb{E}}}{q}}}{\log^{\beta} \left(\frac{1}{\bar{r}}\right)}\mathfrak{S}_{\left[g, q, \tau, \beta \right]}(x, \bar{r}).
			\end{align}
			Similarly, we have 
			\begin{align}\label{estimate for 3rd term}
				\left\lvert g(x) - \left(g\right)_{y, \tau^{j+1}r}\right\rvert \le \tau^{-\frac{Q_{\mathbb{E}}}{q}}\frac{1}{\log^{\beta} \left(\frac{1}{\bar{r}}\right)}\mathfrak{S}_{\left[g, q, \tau, \beta \right]}(y, \bar{r}). 
			\end{align} 
			For the middle term, we have 
			\begin{align}\label{middle term esti}
				&\left\lvert \left(g\right)_{x, \tau^{j+1}r} - \left(g\right)_{y, \tau^{j+1}r}\right\rvert \notag \\ 
				&\hspace{5pt}\le \fint_{B_{\tau^{j+1}r}(x)} \left\lvert g - \left(g\right)_{x, \tau^{j-1}r}\right\rvert + \fint_{B_{\tau^{j+1}r}(y)} \left\lvert g - \left(g\right)_{x, \tau^{j-1}r}\right\rvert \notag\\
				&\hspace{5pt}\le \frac{\left\lvert B_{\tau^{j-1}r}(x)\right\rvert}{\left\lvert B_{\tau^{j+1}r}(x)\right\rvert}\fint_{B_{\tau^{j-1}r}(x)} \left\lvert g - \left(g\right)_{x, \tau^{j-1}r}\right\rvert + \frac{\left\lvert B_{\tau^{j-1}r}(x)\right\rvert}{\left\lvert B_{\tau^{j+1}r}(y)\right\rvert}\fint_{B_{\tau^{j-1}r}(x)} \left\lvert g - \left(g\right)_{x, \tau^{j-1}r}\right\rvert
				 \notag\\
				&\hspace{5pt}\le 2\tau^{-2Q_{\mathbb{E}}} \left( \fint_{B_{\tau^{j-1}r}(x)} \left\lvert g - \left(g\right)_{x, \tau^{j-1}r}\right\rvert^{q} \right)^{\frac{1}{q}}  \notag \\
				&\hspace{5pt}\le 2\tau^{-2Q_{\mathbb{E}} -\frac{Q_{\mathbb{E}}}{q}}\frac{1}{\log^{\beta} \left(\frac{1}{r_{j-1}}\right)}\mathfrak{S}_{\left[g, q, \tau, \beta\right]}(x, r_{j-1}). 
			\end{align}
			Since $j \geq 4$ and $r<1,$ we have $\tau^{j+1}r \le \tau^{4}$ and thus $\log \left(\frac{1}{r_{j+1}}\right) \ge 4 \log \left(\frac{1}{\tau}\right).$ Hence 
			\begin{align*}
				\log \left(\frac{1}{r_{j-1}}\right) = 	\log\left(\frac{1}{r_{j+1}}\right) - 2	\log \left(\frac{1}{\tau}\right) \geq \frac{1}{2} \log \left(\frac{1}{r_{j+1}}\right).
			\end{align*}
			Combining this with \eqref{middle term esti}, we arrive at 
			\begin{align}\label{estimate for 2nd term}
				\left\lvert \left(g\right)_{x, \tau^{j+1}r} - \left(g\right)_{y, \tau^{j+1}r}\right\rvert &\le 4\tau^{-2Q_{\mathbb{E}} -\frac{Q_{\mathbb{E}}}{q}}\frac{1}{\log^{\beta} \left(\frac{1}{r_{j+1}}\right)}\mathfrak{S}_{\left[g, q, \tau, \beta \right]}(x, r_{j-1}) \notag \\
				&\le 4\tau^{-2Q_{\mathbb{E}} -\frac{Q_{\mathbb{E}}}{q}}\frac{1}{\log^{\beta} \left(\frac{1}{\bar{r}}\right)}\mathfrak{S}_{\left[g, q, \tau, \beta \right]}(x, r_{j-1}).  
			\end{align}
			Combining \eqref{estimate for 1st term}, \eqref{estimate for 3rd term} and \eqref{estimate for 2nd term} with \eqref{add substract averages} and taking supremum over all $x, y \in K$ with $d_{\mathbb{E}}(x,y) \le \bar{r}$, we deduce 
			\begin{align*}
				\omega_{g_{K}}\left(\bar{r}\right)\log^{\beta} \left(\frac{1}{\bar{r}}\right) \le C\left(\tau, q, Q_{\mathbb{E}}\right) \sup\limits_{x \in K} \left[ \mathfrak{S}_{\left[g, q, \tau, \beta\right]}(x, \bar{r}) + \mathfrak{S}_{\left[g, q, \tau, \beta \right]}(x, r_{j-1})\right]. 
			\end{align*} Noting that $j \rightarrow \infty$ as $\bar{r} \rightarrow 0,$ we have our result. 
		\end{proof}
		\subsection{Weak solutions and known estimates}
		\subsubsection{Weak solutions and minimization}
		First we define local weak solutions.
		\begin{definition}
			Let $\mathfrak{p}:\Omega \rightarrow [\gamma_{1}, \gamma_{2}]$ be locally log-H\"older continuous and $\mathfrak{a}:\Omega \rightarrow [\nu, L]$ be a measurable function with $1 < \gamma_{1} \leq \gamma_{2} < \infty$ and $ 0 < \nu < L < \infty.$ Assume $f \in L^{Q_{\mathbb{E}}}\left( \Omega\right).$ 
			We say $u \in \mathbb{E}W_{\text{loc}}^{1,p\left(\cdot\right)} \left(\Omega\right)$ is a \emph{local weak solution} to the equation 	\begin{align*}
				\operatorname{div}_{\mathbb{E}}\left[ \mathfrak{a}(x) \lvert \nabla_{\mathbb{E}} u \rvert^{\mathfrak{p}\left(x\right)-2} \nabla_{\mathbb{E}} u\right]    = f   &&\text{ in } \Omega, 
			\end{align*}
			if for any open set $\tilde{\Omega} \subset \subset \Omega$  and any $\phi \in \mathbb{E}W_{0}^{1,p\left(\cdot\right)} \left(\tilde{\Omega} \right)$,  we have 
			\begin{align}\label{local weak formulation}
				\int_{\tilde{\Omega}} \left\langle \mathfrak{a}(x) \lvert \nabla_{\mathbb{E}} u \rvert^{\mathfrak{p}\left(x\right)-2} \nabla_{\mathbb{E}} u, \nabla_{\mathbb{E}} \phi\right\rangle +  \int_{\tilde{\Omega}} f\phi = 0.
			\end{align}
			Furthermore, if $\mathfrak{p} \in \mathcal{P}^{\log}_{[\gamma_{1}, \gamma_{2}]}\left(\Omega \right)$ and $u \in \mathbb{E}W^{1,p\left(\cdot\right)} \left(\Omega\right),$ then \eqref{local weak formulation} is equivalent to 
			\begin{align}\label{energy weak formulation}
				\int_{\Omega} \left\langle \mathfrak{a}(x) \lvert \nabla_{\mathbb{E}} u \rvert^{\mathfrak{p}\left(x\right)-2} \nabla_{\mathbb{E}} u, \nabla_{\mathbb{E}} \phi\right\rangle +  \int_{\Omega} f\phi = 0  
			\end{align}
			for all $\phi \in \mathbb{E}W_{0}^{1,p\left(\cdot\right)}\left(\Omega\right)$.
		\end{definition}
		\begin{proposition}\label{existence of minimizers for variable exponent}
			Let $\mathfrak{p}:\Omega \rightarrow [\gamma_{1}, \gamma_{2}]$ and $\mathfrak{a}:\Omega \rightarrow [\nu, L]$ be measurable functions  with $1 < \gamma_{1} \leq \gamma_{2} < \infty$ and $ 0 < \nu < L < \infty.$ Let $\mathfrak{p}$ be log-H\"{o}lder continuous in $\Omega$. Assume $f \in L^{Q_{\mathbb{E}}}\left( \Omega\right)$ and $u_{0} \in \mathbb{E}W^{1,p\left(\cdot\right)} \left(\Omega\right).$ Then the minimization problem 
			\begin{align}
				m:= \inf \left\lbrace \int_{\Omega} \left[ \frac{\mathfrak{a}(x)}{\mathfrak{p}(x)}\left\lvert \nabla_{\mathbb{E}} u \right\rvert^{\mathfrak{p}(x)} +  f u \right]\ \mathrm{d}x: u \in u_{0} + \mathbb{E}W^{1,p(\cdot)}_{0}\left( \Omega\right)\right\rbrace 
			\end{align}
			admits a unique minimizer $\bar{u} \in \mathbb{E}W^{1,p(\cdot)}\left( \Omega\right)$, which uniquely solves  
			\begin{align*}
				\left\lbrace \begin{aligned}
					\operatorname{div}_{\mathbb{E}}\left[ \mathfrak{a}(x) \lvert \nabla_{\mathbb{E}} \bar{u} \rvert^{\mathfrak{p}\left(x\right)-2} \nabla_{\mathbb{E}} \bar{u}\right] &= f &&\text{ in } \Omega, \\
					\bar{u} &=u_{0} &&\text{ on } \partial\Omega. 
				\end{aligned}\right. 
			\end{align*} 
		\end{proposition}
		This can be proved using standard variational methods. Note that the function $$\phi(x, t) = \frac{\mathfrak{a}(x)}{\mathfrak{p}(x)}t^{\mathfrak{p}(x)} \qquad x \in \Omega, t\le 0,$$ is a  generalized uniformly convex N-function (see \cite[Chapter 2]{Diening_et_al_variable_exponent}). 
		\subsubsection{\texorpdfstring{$p$}{p}-harmonic functions}
		For constant coefficient homogeneous case, i.e. when $a\left(x\right) \equiv a$ for some $a>0$ and $f \equiv 0,$ we have the following fundamental estimate. 
		\begin{theorem}\label{Uhlenbeck estimate}
			Let $1<p < \infty$ and $v \in \mathbb{E}W^{1,p}_{\text{loc}}\left(\Omega\right)$ be a local weak solution to 
			\begin{align*}
				\operatorname{div}_{\mathbb{E}}\left( \left\lvert \nabla_{\mathbb{E}} v\right\rvert^{p-2}\nabla_{\mathbb{E}} v\right) &=0 &&\text{ in } \Omega.
			\end{align*}
			Then $\nabla_{\mathbb{E}} v$ is locally H\"{o}lder continuous in $\Omega$ and there exist constants $C_{1}, C_{2} \ge 1$ and an exponent $0 < \beta < 1$, all depending only on $Q_{\mathbb{E}}$ and $p$, such that  we have 
			\begin{align}\label{sup estimate constant homogeneous}
				\sup\limits_{B_{R/2}(x_{0})} \left\lvert \nabla_{\mathbb{E}} v \right\rvert &\le C_{1}  \fint_{B_{R}(x_{0})} \left\lvert \nabla_{\mathbb{E}} v \right\rvert   \\ \intertext{ and }
				\label{oscillation estimate constant homogeneous} 
				\sup\limits_{x, y \in B_{\tau R}(x_{0})} \left\lvert \nabla_{\mathbb{E}} v\left(x\right) - \nabla_{\mathbb{E}} v\left(y\right) \right\rvert &\le C_{2}\tau^{\beta} \left( \fint_{B_{R}(x_{0})} \left\lvert \nabla_{\mathbb{E}} v \right\rvert^{p}\right)^{\frac{1}{p}},  
			\end{align}
			whenever $x_{0} \in \Omega$, $\tau \in (0, 1/2)$, $R >0$ and $B_{R}(x_{0}) \subset \Omega$. 
			 Furthermore, for any $1 \leq q \leq 2, $ there exists a constant $C_{q}= C_{q}\left( n, p , q\right)\ge 1$, such that  
			\begin{align}\label{excess decay}
				\left( \fint_{B_{\tau R}(x_{0})} \left\lvert \nabla_{\mathbb{E}} v  - \left( \nabla_{\mathbb{E}} v\right)_{B_{\tau R}(x_{0})}\right\rvert^{q}\right)^{\frac{1}{q}}  \leq C_{q}\tau^{\beta} \left( \fint_{B_{R}(x_{0})} \left\lvert \nabla_{\mathbb{E}} v - \left( \nabla_{\mathbb{E}} v\right)_{B_{R}(x_{0})}\right\rvert^{q}\right)^{\frac{1}{q}},    
			\end{align}
			whenever $x_{0} \in \Omega$, $\tau \in (0, 1/2)$, $R >0$ and $B_{R}(x_{0}) \subset \Omega$. 
		\end{theorem}
			\begin{remark}\label{p uniform constant decay estimates}
			Note that while the constants in the estimates \eqref{sup estimate constant homogeneous}, \eqref{oscillation estimate constant homogeneous} and \eqref{excess decay} depend on $p$, but tracking the constants in the proofs reveal that the constants are explicitly quantifiable. Thus, if $1 < \gamma_{1} \leq p \leq \gamma_{2}< \infty,$ the constants can all be chosen (by replacing them with larger or smaller constants if necessary) to depend only $\gamma_{1}, \gamma_{2}$ such that the estimates are uniform in $p \in [\gamma_{1}, \gamma_{2}].$ See also Remark \ref{p uniform constants Sobolev inequalities}. Observe that though the proof of \eqref{excess decay}, when $\mathbb{E}= \mathbb{H}_n$, in \cite{Mallick_Sil_ExcessDecay} uses a contradiction argument, it does not use a `contradiction-compactness argument' and thus the `contradiction argument' is, in effect, merely a case by case analysis, where the constants are quantifiable for each case. 
		\end{remark}
		The estimates \eqref{sup estimate constant homogeneous} and \eqref{oscillation estimate constant homogeneous} in the Euclidean case is due to Uraltceva \cite{Uralceva},  Uhlenbeck \cite{uhlenbecknonlinearelliptic} Evans \cite{Evans_pLaplacian}, DiBenedetto \cite{DiBenedetto-C1-alpha-83}, Lewis \cite{Lewis83}, Tolksdorff \cite{Tolksdorf_regularity} (see also Giaquinta-Modica \cite{Giaquinta_Modica_Uhlenbeck_result}, Hamburger \cite{hamburgerregularity}).  due to implies a `excess-to-energy decay estimate'. When $\mathbb{E}^{n}= \mathbb{H}_{n},$ the estimates \eqref{sup estimate constant homogeneous} and \eqref{oscillation estimate constant homogeneous} are established in Zhong \cite{zhong2018regularityvariationalproblemsheisenberg} and Mukherjee-Zhong \cite{Mukherjee_Zhong}.

		However, for our purposes, \emph{the most crucial estimate is the fundamental `excess-to-excess decay estimate'} \eqref{excess decay}.  When $\mathbb{E}^{n}= \mathbb{R}^{n},$ a different version of this estimate involving the nonlinear quantity $V\left(\nabla u\right)$ was obtained in \cite{Giaquinta_Modica_Uhlenbeck_result} and \cite{hamburgerregularity}. In the stated form, this was obtained for $q=2$ by Dibenedetto-Manfredi\cite{DiBenedettoManfredi}. However, the additional flexibility to choose any $q \in [1,2]$ will be pivotal for us here. In this form, the estimate was obtained by the present authors in \cite{Mallick_Sil_ExcessDecay}.  When $\mathbb{E}^{n}= \mathbb{H}_{n}$, no estimate of this kind was known before \cite{Mallick_Sil_ExcessDecay}. \emph{This is the key ingredient which allows us to treat the subelliptic case by Euclidean perturbation techniques. Once this estimate is available, the effect of noncommutativity completely disappears. }.   

		\subsubsection{\texorpdfstring{$p\left(x\right)$}{p(x)}-harmonic functions}
		Let $\Omega \subset \mathbb{E}^{n}$ be open and bounded. If $\mathfrak{p}:\Omega \rightarrow [\gamma_{1}, \gamma_{2}]$ be a measurable function  with $1 < \gamma_{1} \leq \gamma_{2} < \infty, $ such that for every $\tau \in (0, 1/4)$ and every $1< q< \infty,$ we have 
		\begin{align*}
			\lim\limits_{r \rightarrow 0}	\sup\limits_{x \in K} &\sum\limits_{j=0}^{\infty}\log \left(\frac{1}{\tau^{j}r}\right)\left(\fint_{B_{\tau^{j}r}(x)} \left\lvert \mathfrak{p} - \left(\mathfrak{p}\right)_{x, \tau^{j}r} \right\rvert^{q}\right)^{\frac{1}{q}} = 0, 
		\end{align*}
		whenever $K \subset \subset \Omega$ is a compact subset, then by Lemma \ref{local vanishing log holder}, $\mathfrak{p}$ is log-H\"{o}lder continuous on every compact subset $K \subset \subset \Omega$. Let $\Lambda^{K}_{\log}\left(\mathfrak{p}\right) >0$ be any number greater than or equal to the the log-H\"{o}lder constant of $\mathfrak{p}$ restricted to $K$, i.e.  
		\begin{align}
			\Lambda^{K}_{\log}\left(\mathfrak{p}\right) \ge 	\lim\limits_{R \rightarrow 0} \omega_{\mathfrak{p}_{K}}\left(R\right)\log \left(\frac{1}{R}\right),
		\end{align}
		where $\mathfrak{p}_{K}$ denotes the restriction of $\mathfrak{p}$ to $K$. Furthermore, by our assumption, for any compact subset $K \subset \subset \Omega,$ there exists a function $\Theta_{\mathfrak{p}, K}:(0,1)\times (0, 1/4)\times (1, \infty) \rightarrow (0, 1/16)$ such that 
		\begin{align*}
		\sup\limits_{x \in K}	\sum\limits_{j=0}^{\infty}\log \left(\frac{1}{\tau^{j}r}\right)\left(\fint_{B_{\tau^{j}r}(x)} \left\lvert \mathfrak{p} - \left(\mathfrak{p}\right)_{x, \tau^{j}r} \right\rvert^{q}\right)^{\frac{1}{q}} < \varepsilon
		\end{align*}
		whenever $0 < r < \Theta_{\mathfrak{p}, K}\left(\varepsilon, \tau, q\right).$ With these notations, we can now state the following result. 
		
		\begin{theorem}\label{homogenous eqn estimate}
			Let $\Omega \subset \mathbb{E}^{n}$ be open and bounded. Let $\mathfrak{p}:\Omega \rightarrow [\gamma_{1}, \gamma_{2}]$ be a measurable function  with $1 < \gamma_{1} \leq \gamma_{2} < \infty, $ such that for every $\tau \in (0, 1/4)$ and every $1< q< \infty,$ we have 
			\begin{align*}
				\lim\limits_{r \rightarrow 0}	\sup\limits_{x \in K_0} &\sum\limits_{j=0}^{\infty}\log \left(\frac{1}{\tau^{j}r}\right)\left(\fint_{B_{\tau^{j}r}(x)} \left\lvert \mathfrak{p} - \left(\mathfrak{p}\right)_{x, \tau^{j}r} \right\rvert^{q}\right)^{\frac{1}{q}} = 0, 
			\end{align*}
			whenever $K_0 \subset \subset \Omega$ is a compact subset. Let $\Omega' \subset \subset \bar{K} \subset \subset \Omega,$ where $\Omega'$ is open and $\bar{K}$ is compact. If $w \in \mathbb{E}W^{1,p\left(\cdot\right)} \left(\Omega'\right)$ is a local weak solution to the equation 
			\begin{align*}
				\operatorname{div}_{\mathbb{E}}\left[\lvert \nabla_{\mathbb{E}} w \rvert^{\mathfrak{p}\left(x\right)-2} \nabla_{\mathbb{E}} w\right]    = 0   &&\text{ in } \Omega',  
			\end{align*}	
			then $\nabla_{\mathbb{E}} w$ is continuous in $\Omega'.$ Moreover, if  $\mathcal{K}_{0} >1$ is a real number such that  
			\begin{align}\label{energy bound w by u}
				\mathcal{K}_{0} \ge 1 + \int_{\Omega'} \left( 1 + \left\lvert  \nabla_{\mathbb{E}} w \right\rvert \right)^{\mathfrak{p}\left(x\right)} \ \mathrm{d}x,  
			\end{align} then  
			\begin{enumerate}[(i)]
				\item There exists a constant $C_{3} = C_{3}\left(Q_{\mathbb{E}}, \gamma_{1}, \gamma_{2}, \Lambda^{\bar{K}}_{\log}\left(\mathfrak{p}\right)\right) \geq 1$ and a positive radius $$R_{3} = R_{3}\left( Q_{\mathbb{E}}, \gamma_{1}, \gamma_{2}, \Lambda^{\bar{K}}_{\log}\left(\mathfrak{p}\right) , \Theta_{\mathfrak{p}, \bar{K}},  \mathcal{K}_{0}, \operatorname*{dist}\left(\bar{K}, \partial\Omega\right) \right) >0$$ such that if $0 < R \leq R_{3},$ then we have 
				\begin{align}\label{sup bound homogeneous}
					\sup\limits_{B\left(x, R/2\right)} \left\lvert \nabla_{\mathbb{E}} w \right\rvert \leq C_{3}\fint_{B\left(x, R\right)}\left( 1 +  \left\lvert \nabla_{\mathbb{E}} w \right\rvert \right)  
				\end{align}
				holds whenever $B\left(x, R\right) \subset  \Omega'$.  
				\item For any $\varepsilon \in (0,1)$ and any $A \geq 1,$ there exists a positive constant $\tau_{1}  \in (0, 1/4),$ depending only on 
				\begin{align*}
					Q_{\mathbb{E}}, \gamma_{1}, \gamma_{2}, \Lambda^{\bar{K}}_{\log}\left(\mathfrak{p}\right) , \Theta_{\mathfrak{p}, \bar{K}},  \mathcal{K}_{0}, \operatorname{dist}\left(\bar{K}, \partial\Omega\right), A \text{ and } \varepsilon, 
				\end{align*} such that  
				\begin{align}
					\sup\limits_{B\left(x_{0}, R/2\right)} \left\lvert \nabla_{\mathbb{E}} w \right\rvert \leq A\lambda  \implies \sup\limits_{x, y \in B\left(x_{0}, \tau R\right)} \left\lvert \nabla_{\mathbb{E}} w \left( x\right) - \nabla_{\mathbb{E}} w \left( y\right) \right\rvert \leq \varepsilon \lambda, 
				\end{align}
				for all $0 < \tau \leq \tau_{1},$ whenever $0 < R < R_{3}$, $B\left(x_0, R\right) \subset \Omega'$ and $\lambda >1$ is a positive constant.  
			\end{enumerate}
		\end{theorem}	
		This result was established by the present authors in \cite[Theorem 2]{Mallick_Sil_continuityvariablegrowth} for $\mathbb{H}_{n}$ under the assumption that $\X \mathfrak{p} \in L^{\left(Q_{\mathbb{E}},1\right)}\log L \left(\Omega; \mathbb{R}^{2n}\right)$. However, it is not difficult to see that the proof really used our assumption on $\mathfrak{p}$ here. The condition on the horizontal gradient of $\mathfrak{p}$ was used only to obtain our condition via the estimates in Proposition \ref{derivative to summability}.

		\section{Higher integrability results}\label{Higher integrability}
		In this section, we prove the higher integrability results that we need. For this section, we only assume that the exponent function $\mathfrak{p}:\Omega \subset \mathbb{E}^{n} \rightarrow \left[\gamma_{1}, \gamma_{2}\right]$ satisfies the log-H\"{o}lder condition 
		\begin{align}\label{log holder bound Global HI}
			\sup_{R \in (0,1]} \omega_{\mathfrak{p}}\left(R\right)\log \left(\frac{1}{R}\right) \le\Lambda_{\log} < \infty 
		\end{align} and the coefficient function $\mathfrak{a}: \Omega \subset \mathbb{E}^{n} \rightarrow \left[\nu, L\right]$ is measurable.  The notation \texttt{data} will stand for the set of constants $$ \texttt{data} := \left\lbrace Q_{\mathbb{E}}, \gamma_{1}, \gamma_{2}, \nu, L, \Lambda_{\log} \right\rbrace. $$   
		We assume $\Lambda_{\log} = L/2,$ as this can always be achieved by increasing $L$.
		\subsection{Local higher integrability} 
		A version of the following higher integrability result goes back to Zhikov \cite{Zhikov_higherintegrability} for $\mathbb{E}^n= \mathbb{R}^n$. However, the additional flexibility offered by the exponent $\bar{\mathfrak{p}}$ is crucial for us. For $f=0$, a similar result is proved in Baroni \cite[Theorem 2.4]{Baroni_pxStein}. \cite{Mallick_Sil_continuityvariablegrowth} has a version with $f$ in divergence form for $\mathbb{E}^n= \mathbb{H}_n$.  
		\begin{theorem}[Higher integrability]\label{higher integrability f}
			Let $f \in L^{Q_{\mathbb{E}}} \left(\Omega\right)$ and $u \in \mathbb{E}W^{1,\mathfrak{p}\left(\cdot\right)} \left(\Omega\right)$ be a weak solution to the equation 
			\begin{align}\label{p(x) Laplace homogeneous diveregnce form f}
				\operatorname{div}_{\mathbb{E}}( \mathfrak{a}(x) \lvert \nabla_{\mathbb{E}} u \rvert^{\mathfrak{p}\left(\cdot\right)-2} \nabla_{\mathbb{E}}u) )  = f    &&\text{ in } \Omega.
			\end{align} Let $\gamma$ be as in \eqref{def gamma} and $\mathcal{K}_{1} >1$ be a real number such that  
			\begin{align}\label{energy bound f}
				\mathcal{K}_{1} \ge 1 + \int_{\Omega} \left( 1 + \left\lvert  \nabla_{\mathbb{E}} u \right\rvert \right)^{\mathfrak{p}\left(x\right)} \ \mathrm{d}x + \int_{\Omega} \left( 1 + \left\lvert f \right\rvert  \right)^{Q_{\mathbb{E}}} \ \mathrm{d}x.   
			\end{align}	
			 There exist $C_{4} \equiv C_{4} \left( \textnormal{\texttt{data}} \right)\ge 1$, $\sigma_{1} \equiv \sigma_{1}\left(\textnormal{\texttt{data}}\right) \in (0,1)$, $\theta = \theta \left(Q_{\mathbb{E}}, \gamma, \sigma_{1}\right)$ with
			\begin{align}\label{qdef}
			 1 < \theta \le \frac{\gamma Q_{\mathbb{E}}}{\gamma Q_{\mathbb{E}} -Q_{\mathbb{E}} + \gamma}\left(1 + \sigma_{1}\right) < Q_{\mathbb{E}} 
			\end{align} and a radius $R_{1} = R_{1}\left(\textnormal{\texttt{data}}, \mathcal{K}_{1}\right)>0$ with $R_1 \mathcal{K}_1 <1$, such that for any ball $B_{4R}\left( x_{0}\right) \subset \Omega$ with $0 < R \leq R_{1},$  any $ 0 < \sigma \leq \sigma_{1}$ and any $\bar{x} \in \overline{B_{R}\left( x_{0}\right)},$ we have 
			\begin{align}\label{higher integrability estimate f}
				&\left( \fint_{B_{R}}\left\lvert \nabla_{\mathbb{E}} u \right\rvert^{\mathfrak{p}\left(x\right)\left(1 + \sigma\right)}\ \mathrm{d}x\right)^{\frac{1}{ 1+\sigma }}\notag \\
				 &\hspace{40pt}\leq C_4 \left[ \left( \fint_{B_{2R}}\left\lvert \nabla_{\mathbb{E}} u \right\rvert^{\gamma}\ \mathrm{d}x \right)^{\frac{1}{\gamma}}  +  \left( R^{\theta} \fint_{B_{2R}}\left\lvert f \right\rvert^{\theta}\ \mathrm{d}x \right)^{\frac{1}{\theta\left(\mathfrak{p}_{0}-1\right)}} + 1 \right]^{\bar{\mathfrak{p}}},
			\end{align} 
			where $\mathfrak{p}_{0}:= \mathfrak{p}(x_{0}),$ $\bar{\mathfrak{p}} := \mathfrak{p}\left(\bar{x}\right)$, $B_{R} = B_{R}(x_{0})$ and $B_{2R} = B_{2R}(x_{0})$. 
		\end{theorem}
		\begin{proof}
			Fix $x_0 \in \Omega$ and choose $R_1$ sufficiently small such that $B_{4R_1}(x_0) \subset \Omega$ and 
			\begin{align}
				R_{1} &\leq \min \left\lbrace 1/\mathcal{K}_1, 1/4 \right\rbrace, \label{radius dependence on K0} \\
              2Q_{\mathbb{E}}\, \omega_{\mathfrak{p}}(4R_1) &\le s:= \min \left\lbrace \gamma_{1}, \frac{2Q_{\mathbb{E}}}{2Q_{\mathbb{E}}-1}\right\rbrace, \label{radius dependence on gamma_1}
                \end{align} 
                \begin{align}\omega_{\mathfrak{p}}\left( R\right)\log \left(\frac{1}{R}\right) &\leq L  \qquad \qquad \text{ for all } 0 < R \leq R_{1}/16. \label{choice of radius-nondiv}
			\end{align} 
			Note that these choices imply that for every $0 < R \leq R_{1}/16$, we have the bounds  
			\begin{align*}
				R^{-\omega_{\mathfrak{p}}(R)}, \mathcal{K}_{1}^{\omega_{\mathfrak{p}}(R)} \le C \left( \texttt{data}\right) \qquad \text{ and } \qquad \frac{\mathfrak{p}_{1}}{\mathfrak{p}_{2}} \geq 1-\frac{s}{2Q_{\mathbb{E}}}, \qquad \text{ where }
			\end{align*}
			\begin{align*}
				\mathfrak{p}_{2}:= \max\limits_{x \in \overline{B_{2R}(x_{0})}} \mathfrak{p} (x) \qquad \text{ and } \qquad \mathfrak{p}_{1}:= \min\limits_{x \in \overline{B_{2R}(x_{0})}} \mathfrak{p} (x).
			\end{align*}
			Observe that by our definition of $s$, $\mathbb{E}W^{1,\mathfrak{p}_{1}/s }$ embeds into $L^{\mathfrak{p}_{2}}$	by Sobolev embedding for $\mathbb{E}^{n}$ (see for e.g. \cite[equation (48)]{Mallick_Sil_continuityvariablegrowth})
			From now on, all balls would be centered at $x_{0}$ and we omit writing the center of the balls and $0 < R \le R_{0}.$ Now we choose a cut-off function $\eta \in C_{c}^{\infty}\left(B_{2R}\right)$ with 
			\begin{align*}
				0 \leq \eta \leq 1 \text{ in  } B_{2R}, \qquad \eta \equiv 1 \text{ in } B_{R} \qquad \text{ and } \left\lvert \nabla_{\mathbb{E}} \eta \right\rvert \leq C/R.   
			\end{align*}
			Then using Young's inequality with $\varepsilon>0,$ and Sobolev inequality we derive 
			\begin{align*}
				\int_{B_{2R}}& \left \lvert f \eta^{\mathfrak{p}_{2}}  \left( u - \left(u\right)_{2R}\right) \right\rvert \\ 
				& \leq \left( \int_{B_{2R}} \left\lvert \eta^{\mathfrak{p}_{2}}  \left( u - \left(u\right)_{2R}\right) \right\rvert^{\gamma_{1}^{\ast}} \right)^{\frac{1}{\gamma_{1}^{\ast}}}\left( \int_{B_{2R}} \left\lvert f \right\rvert^{\left(\gamma_{1}^{\ast}\right)^{'}} \right)^{\frac{1}{\left(\gamma_{1}^{\ast}\right)^{'}}} \\
				&\leq c \left( \int_{B_{2R}} \left\lvert \nabla_{\mathbb{E}} \left[ \eta^{\mathfrak{p}_{2}}  \left( u - \left(u\right)_{2R}\right) \right] \right\rvert^{\gamma_{1}} \right)^{\frac{1}{\gamma_{1}}}\left( \int_{B_{2R}} \left\lvert f \right\rvert^{\left(\gamma_{1}^{\ast}\right)^{'}} \right)^{\frac{1}{\left(\gamma_{1}^{\ast}\right)^{'}}} \\
				&\leq c R^{Q_{\mathbb{E}}\left( \frac{1}{\gamma_{1}} - \frac{1}{\mathfrak{p}_{1}}\right)}\left( \int_{B_{2R}} \left\lvert \nabla_{\mathbb{E}} \left[ \eta^{\mathfrak{p}_{2}}  \left( u - \left(u\right)_{2R}\right) \right] \right\rvert^{\mathfrak{p}_{1}} \right)^{\frac{1}{\mathfrak{p}_{1}}}\left( \int_{B_{2R}} \left\lvert f \right\rvert^{\left(\gamma_{1}^{\ast}\right)^{'}} \right)^{\frac{1}{\left(\gamma_{1}^{\ast}\right)^{'}}} \\
				&\leq \varepsilon \int_{B_{2R}} \left\lvert \nabla_{\mathbb{E}} \left[ \eta^{\mathfrak{p}_{2}}  \left( u - \left(u\right)_{2R}\right) \right] \right\rvert^{\mathfrak{p}_{1}} + C_{\varepsilon}R^{Q_{\mathbb{E}}\mathfrak{p}_{1}^{'}\left( \frac{1}{\gamma_{1}} - \frac{1}{\mathfrak{p}_{1}}\right)}\left( \int_{B_{2R}} \left\lvert f \right\rvert^{\left(\gamma_{1}^{\ast}\right)^{'}} \right)^{\frac{\mathfrak{p}_{1}^{'}}{\left(\gamma_{1}^{\ast}\right)^{'}}} \\
				&\begin{aligned}
				\leq \varepsilon \int_{B_{2R}} \eta^{\mathfrak{p}_{2}} \left\lvert \nabla_{\mathbb{E}} u \right\rvert^{\mathfrak{p}(x)} &+ c \int_{B_{2R}} \left\lvert \frac{\left( u - \left(u\right)_{2R}\right)}{2R}\right\rvert^{\mathfrak{p}_{2}} \\
				&+ c \left\lvert B_{2R} \right\rvert   + C_{\varepsilon}R^{Q_{\mathbb{E}}\mathfrak{p}_{1}^{'}\left( \frac{1}{\gamma_{1}} - \frac{1}{\mathfrak{p}_{1}}\right)}\left( \int_{B_{2R}} \left\lvert f \right\rvert^{\left(\gamma_{1}^{\ast}\right)^{'}} \right)^{\frac{\mathfrak{p}_{1}^{'}}{\left(\gamma_{1}^{\ast}\right)^{'}}}. 
				\end{aligned}
			\end{align*}
			Plugging $\phi = \eta^{\mathfrak{p}_{2}} \left( u - \left(u\right)_{2R}\right)  \in \mathbb{E}W_{0}^{1, \mathfrak{p}\left(\cdot\right)}\left( B_{2R}\right)$ in the weak formulation of \eqref{p(x) Laplace homogeneous diveregnce form f}, and using the above estimate we deduce the Caccioppoli inequality 
			\begin{align}\label{higher int f rev holder}
				\fint_{B_{R}}\left\lvert \nabla_{\mathbb{E}} u  \right\rvert^{\mathfrak{p}\left(x\right)} \leq C \fint_{B_{2R}}\left\lvert \frac{\left( u - \left(u\right)_{2R}\right)}{2R} \right\rvert^{\mathfrak{p}_{2}} + c  + C R^{\mathfrak{p}_{1}^{'}}\left( \fint_{B_{2R}} \left\lvert f \right\rvert^{\left(\gamma_{1}^{\ast}\right)^{'}} \right)^{\frac{\mathfrak{p}_{1}^{'}}{\left(\gamma_{1}^{\ast}\right)^{'}}}. 
			\end{align}
			Here again we refer the readers to the derivation of \cite[equation (53)]{Mallick_Sil_continuityvariablegrowth} for details. Now, by Poincar\'e-Sobolev \eqref{poincaresobolevineqwithmeans} and H\"older inequality, we have 
			\begin{align*}
				\fint_{B_{2R}}\left\lvert \frac{\left( u - \left(u\right)_{2R}\right)}{2R} \right\rvert^{\mathfrak{p}_{2}} &\leq c \left( 1 + \fint_{B_{2R}} \left\lvert \nabla_{\mathbb{E}} u \right\rvert^{\frac{\mathfrak{p}\left(x\right)}{s}}\right)^{\frac{\mathfrak{p}_{2}s}{\mathfrak{p}_{1}}} \notag \\&\leq c \left( 1 + \fint_{B_{2R}} \left\lvert \nabla_{\mathbb{E}} u \right\rvert^{\frac{\mathfrak{p}\left(x\right)}{s}}\right)^{ \frac{s(\mathfrak{p}_{2} -\mathfrak{p}_1)}{\mathfrak{p}_{1}} } \left( 1 + \fint_{B_{2R}} \left\lvert \nabla_{\mathbb{E}} u \right\rvert^{\frac{\mathfrak{p}\left(x\right)}{s}}\right)^{s} . 
			\end{align*}
			Now since $0 < \left( \frac{\mathfrak{p}_{2}}{\mathfrak{p}_{1}} -1 \right)s < 1$ and $s >1,$ which follows from \eqref{radius dependence on gamma_1}, we have 
			\begin{align*}
				\left( 1 + \fint_{B_{2R}} \left\lvert\nabla_{\mathbb{E}} u \right\rvert^{\frac{\mathfrak{p}\left(x\right)}{s}}\right)^{\left( \frac{\mathfrak{p}_{2}}{\mathfrak{p}_{1}} -1 \right)s} &\leq   1 + \left( \fint_{B_{2R}} \left\lvert \nabla_{\mathbb{E}} u \right\rvert^{\mathfrak{p}\left(x\right)} \right)^{\left( \frac{\mathfrak{p}_{2}}{\mathfrak{p}_{1}} -1 \right)}\\
				&\leq 1 + cR^{-Q_{\mathbb{E}}\left( \frac{\mathfrak{p}_{2}}{\mathfrak{p}_{1}} -1 \right)} \mathcal{K}_{1}^{\left( \frac{\mathfrak{p}_{2}}{\mathfrak{p}_{1}} -1 \right)} \le 1 + C \left( \texttt{data}\right).  
			\end{align*}
			Combining the last two displays, we have 
			\begin{align}\label{PS for higher int f final}
				\fint_{B_{2R}}\left\lvert \frac{\left( u - \left(u\right)_{2R}\right)}{2R} \right\rvert^{\mathfrak{p}_{2}} \leq c \left( 1 + \fint_{B_{2R}} \left\lvert \nabla_{\mathbb{E}} u \right\rvert^{\frac{\mathfrak{p}\left(x\right)}{s}}\right)^{s} . 
			\end{align}
			Plugging  \eqref{PS for higher int f final} in \eqref{higher int f rev holder} we arrive at the reverse H\"{o}lder inequality 
			\begin{align*}
				\fint_{B_{R}}\left\lvert \nabla_{\mathbb{E}} u  \right\rvert^{\mathfrak{p}\left(x\right)} \leq c \left( \fint_{B_{2R}} \left\lvert \nabla_{\mathbb{E}} u \right\rvert^{\frac{\mathfrak{p}\left(x\right)}{s}}\right)^{s} + c\fint_{B_{2R}}\left\lvert \tilde{f} \right\rvert^{\left(\gamma_{1}^{\ast}\right)^{'}} + c, 
			\end{align*}
			where 	
			\begin{align}\label{tilde f def}
				\tilde{f}:= \left[ R^{\mathfrak{p}_{1}^{'}}\left( \fint_{B_{2R}} \left\lvert f \right\rvert^{\left(\gamma_{1}^{\ast}\right)^{'}} \right)^{\frac{\mathfrak{p}_{1}^{'}}{\left(\gamma_{1}^{\ast}\right)^{'}}-1}\right]^{\frac{1}{\left(\gamma_{1}^{\ast}\right)^{'}}} f. 
			\end{align}
			By a standard Gehring lemma type argument (see \cite[Theorem 3.3]{GheringZatorska-Goldstein}), this implies the existence of a constant $\sigma_{1}>0$ such that for any $t \in (0,1],$ there exists a constant $c_{t} >0,$ and  which depends on $t$ and $\texttt{data}$  we have the estimate 
			\begin{align}\label{higher integrability tilde f}
				\fint_{B_{R}}\left\lvert \nabla_{\mathbb{E}} u  \right\rvert^{\mathfrak{p}\left(x\right)\left(1 + \sigma \right)}  \leq c_{t}\left[  \left( \fint_{B_{2R}} \left\lvert \nabla_{\mathbb{E}} u \right\rvert^{\mathfrak{p}\left(x\right)t}\right)^{\frac{1+\sigma}{t}} + \fint_{B_{2R}}\left\lvert \tilde{f} \right\rvert^{\left(\gamma_{1}^{\ast}\right)^{'}\left(1+\sigma\right)} + 1 \right] ,  
			\end{align}
			for all $ 0 < \sigma\leq \sigma_1.$	
			By the definition of $\theta$ in \eqref{qdef}, we have $\left(\gamma_1^{\ast}\right)' \leq \left(\gamma_1^{\ast}\right)'(1+\sigma)  \leq \theta$. This, coupled with the definition of  $\tilde{f}$ in \eqref{tilde f def} and H\"older's inequality yields 
			\begin{align*}
				\fint_{B_{2R}}&\left\lvert \tilde{f} \right\rvert^{\left(\gamma_{1}^{\ast}\right)^{'}\left(1+\sigma\right)} \leq \left[ R^{\mathfrak{p}_{1}^{'}}\left( \fint_{B_{2R}}\left\lvert f \right\rvert^{\theta} \right)^{\frac{\mathfrak{p}_{1}^{'}}{\theta}}\right]^{(1+\sigma)}. 
			\end{align*}   
			Plugging this in \eqref{higher integrability tilde f}, we deduce 
			\begin{align*}
				\left( \fint_{B_{R}}\left\lvert \nabla_{\mathbb{E}} u  \right\rvert^{\mathfrak{p}\left(x\right)\left(1 + \sigma \right)} \right)^{\frac{1}{1+ \sigma}} 
				\hspace{-5pt}\leq c_{t} \left(  \left( \fint_{B_{2R}} \left\lvert \nabla_{\mathbb{E}} u \right\rvert^{\mathfrak{p}\left(x\right)t}\right)^{\frac{1}{t}} + R^{\mathfrak{p}_{1}^{'}}\left( \fint_{B_{R}}\left\lvert f \right\rvert^{\theta} \right)^{\frac{\mathfrak{p}_{1}^{'}}{\theta}} + 1 \right) . 
			\end{align*}
			Now, we take $t= \gamma/\gamma_{2}$ and estimate 
			\begin{align*}
				\left( \fint_{B_{2R}} \left\lvert \nabla_{\mathbb{E}} u \right\rvert^{\mathfrak{p}\left(x\right)\frac{\gamma}{\gamma_{2}}}\right)^{\frac{\gamma_{2}}{\gamma}} &\leq \left( \fint_{B_{2R}} \left\lvert \nabla_{\mathbb{E}} u \right\rvert^{\mathfrak{p}\left(x\right)\frac{\gamma}{\mathfrak{p}_{2}}}\right)^{\frac{\mathfrak{p}_{2}}{\gamma}}\notag \\
				&= \left( \fint_{B_{2R}} \left\lvert \nabla_{\mathbb{E}} u \right\rvert^{\mathfrak{p}\left(x\right)\frac{\gamma}{\mathfrak{p}_{2}}}\right)^{\frac{\mathfrak{p}_{0}}{\gamma}}\left( \fint_{B_{2R}} \left\lvert \nabla_{\mathbb{E}} u \right\rvert^{\mathfrak{p}\left(x\right)\frac{\gamma}{\mathfrak{p}_{2}}}\right)^{\frac{\mathfrak{p}_{2} - \mathfrak{p}_{0}}{\gamma}}.
			\end{align*}
			But the rightmost term is bounded by $C \left( \texttt{data}\right)$ and for the other we have 
			\begin{align*}
				\left( \fint_{B_{2R}} \left\lvert \nabla_{\mathbb{E}} u \right\rvert^{\mathfrak{p}\left(x\right)\frac{\gamma}{\mathfrak{p}_{2}}}\right)^{\frac{\mathfrak{p}_{0}}{\gamma}} &\leq c\left( \fint_{B_{2R}} \left( 1 + \left\lvert \nabla_{\mathbb{E}} u \right\rvert\right)^{\gamma}\right)^{\frac{\mathfrak{p}_{0}}{\gamma}} \leq c \left[ 1 + \left( \fint_{B_{2R}} \left\lvert \nabla_{\mathbb{E}} u \right\rvert^{\gamma}\right)^{\frac{\mathfrak{p}_{0}}{\gamma}} \right]. 
			\end{align*}
			Similarly, noting that $\mathfrak{p}_{1}^{'} \geq \mathfrak{p}_{0}^{'}$ and $Q_{\mathbb{E}}>\theta$, we can estimate 
			\begin{align*}
				\left( R^{\theta} \fint_{B_{2R}}\left\lvert f \right\rvert^{\theta} \right)^{\frac{\mathfrak{p}_{1}^{'}}{\theta}} &\leq \left( R^{\theta} \fint_{B_{2R}}\left\lvert f \right\rvert^{\theta} \right)^{\frac{\mathfrak{p}_{0}^{'}}{\theta}}\left( R^{\theta} \fint_{B_{2R}}\left\lvert f \right\rvert^{\theta} \right)^{\frac{\mathfrak{p}_{1}^{'}-\mathfrak{p}_{0}^{'}}{\theta}} \\
				&\leq \left( R^{\theta} \fint_{B_{2R}}\left\lvert f \right\rvert^{\theta} \right)^{\frac{\mathfrak{p}_{0}^{'}}{\theta}}\left( R^{Q_{\mathbb{E}}} \fint_{B_{2R}}\left\lvert f \right\rvert^{Q_{\mathbb{E}}} \right)^{\frac{\mathfrak{p}_{1}^{'}-\mathfrak{p}_{0}^{'}}{Q_{\mathbb{E}}}} \\
				&\leq  C \left( \texttt{data}\right)\left( R^{\theta} \fint_{B_{2R}}\left\lvert f \right\rvert^{\theta} \right)^{\frac{\mathfrak{p}_{0}^{'}}{\theta}}. 
			\end{align*}
			%
			In view of these estimates, we have 
			\begin{align*}
				\left( \fint_{B_{R}}\left\lvert \nabla_{\mathbb{E}} u  \right\rvert^{\mathfrak{p}\left(x\right)\left(1 + \sigma \right)} \right)^{\frac{1}{1+ \sigma}} \leq c \left[  \left( \fint_{B_{2R}} \left\lvert \nabla_{\mathbb{E}} u \right\rvert^{\gamma}\right)^{\frac{\mathfrak{p}_{0}}{\gamma}} + \left( R^{\theta} \fint_{B_{2R}}\left\lvert f \right\rvert^{\theta} \right)^{\frac{\mathfrak{p}_{0}^{'}}{\theta}} + 1 \right] .
			\end{align*}
			This implies 
			\begin{align*}
				\left( \fint_{B_{R}}\left\lvert \nabla_{\mathbb{E}} u  \right\rvert^{\mathfrak{p}\left(x\right)\left(1 + \sigma \right)} \right)^{\frac{1}{1+ \sigma}} \notag &\leq c \left[  \left( \fint_{B_{2R}} \left\lvert \nabla_{\mathbb{E}} u \right\rvert^{\gamma}\right)^{\frac{1}{\gamma}} + \left( R^{q} \fint_{B_{2R}}\left\lvert f \right\rvert^{\theta} \right)^{\frac{1}{\theta\left(\mathfrak{p}_{0}-1\right)}} + 1 \right]^{\mathfrak{p}_{0}} \notag \\
				& \leq c \left[  \left( \fint_{B_{2R}} \left\lvert \nabla_{\mathbb{E}} u \right\rvert^{\gamma}\right)^{\frac{1}{\gamma}} + \left( R^{q} \fint_{B_{2R}}\left\lvert f \right\rvert^{\theta} \right)^{\frac{1}{\theta\left(\mathfrak{p}_{0}-1\right)}} + 1 \right]^{\mathfrak{p}_{2}}. 
			\end{align*}
			By fairly straightforward estimates using the fact that $R < 1$, alongside the bounds $R^{-\omega_{\mathfrak{p}}(R)}, \mathcal{K}_{1}^{\omega_{\mathfrak{p}}(R)} \le C \left( \texttt{data}\right)$, we have 
			\begin{align*}
				\left[  \left( \fint_{B_{2R}} \left\lvert \nabla_{\mathbb{E}} u \right\rvert^{\gamma}\right)^{\frac{1}{\gamma}} + \left( R^{q} \fint_{B_{2R}}\left\lvert f \right\rvert^{\theta} \right)^{\frac{1}{\theta\left(\mathfrak{p}_{0}-1\right)}} + 1 \right]^{\mathfrak{p}_{2} -\bar{\mathfrak{p}}} \leq C \left( \texttt{data}\right). 
			\end{align*}
			Combining the last two displays, we have the desired result. 
		\end{proof}
		\subsection{Global higher integrability}
		\begin{theorem}\label{global_higher_integrability}
			Let $u\in \mathbb{E}W^{1, \mathfrak{p}(\cdot)}(\Omega)$ and let $\mathcal{K}_{2} >1$ be a real number such that  
			\begin{align}\label{energy bound f Global HI}
			\mathcal{K}_{2} \ge	1 + \int_{\Omega} \left(  1+ \left\lvert \nabla_{\mathbb{E}} u \right\rvert \right)^{\mathfrak{p}\left(x\right)} \ \mathrm{d}x.
			\end{align}
			Let $R>0$ be such that $B_{4R} \subset \Omega$ and  assume $u\in \mathbb{E}W^{1,\mathfrak{p}(\cdot)(1+\tilde{\sigma})}(B_{2R})$ for some $\tilde{\sigma}>0$. Let $w \in u+\mathbb{E}W_0^{1,\mathfrak{p}(\cdot)}\left(B_{R}\right)$ be the unique minimizer of 
			\begin{align*}
				\inf \left\lbrace \int_{B_{R}} \frac{1}{\mathfrak{p}(x)}\left\lvert \nabla_{\mathbb{E}} w\right\rvert^{\mathfrak{p}\left(x\right)}\ \mathrm{d}x: w \in u + \mathbb{E}W_{0}^{1,\mathfrak{p}(\cdot)}\left(B_{R}\right) \right\rbrace.  
			\end{align*}
			Then there exist constants $R_{2} \equiv R_{2}\left(\textnormal{\texttt{data}}, \mathcal{K}_{2} \right)$, $C_{5} \equiv C_{5} \left(\textnormal{\texttt{data}} \right) \ge 1$ and $\sigma_{2} \equiv \sigma_{2}\left(\textnormal{\texttt{data}}, \tilde{\sigma} \right) \in [0,\tilde{\sigma})$ such that for all $0 < \sigma \leq \sigma_{2},$ we have 
			\begin{align}\label{Global HI}
				\fint_{B_{R}} \left\lvert \nabla_{\mathbb{E}} w \right\rvert^{\mathfrak{p}(x)\left(1 + \sigma\right)} \leq C_{5} \left( 1+  \fint_{B_{2R}} \left\lvert \nabla_{\mathbb{E}} u \right\rvert^{\mathfrak{p}(x)\left(1 + \sigma\right)}\right), 
			\end{align}
			for any $0 < R < R_{2}$, whenever $B_{4R} \subset \Omega$. 
		\end{theorem}
		\begin{proof}
			When $\mathbb{E}^{n} = \mathbb{R}^{n}$ and $\mathfrak{p}(x) \equiv p,$ this is just \cite[Theorem 6.8]{Giusti_DCV}. However, since the argument with cubes cannot work in $\mathbb{H}^{n},$ we need to adapt the proof using the Gehring lemma proved in \cite{GheringZatorska-Goldstein}. First we set $0<R_{2}<1/\mathcal{K}_2$ so small that  
			\begin{align}\label{smallness of wp Global HI}
				\omega_{\mathfrak{p}}(4R)< \min \left \lbrace \gamma_1/ Q_{\mathbb{E}}, 2/\left(2Q_{\mathbb{E}}-1\right)\right \rbrace \qquad \text{ for all } 0 < R < R_{2}.
			\end{align}
			Observe that $w \in u+\mathbb{E}W_0^{1,\mathfrak{p}(\cdot)}\left(B_{R}\right)$ satisfies 
			\begin{align}\label{equation frozen Global HI}
				\int_{B_{R}} \left\langle \left\lvert \nabla_{\mathbb{E}} w \right\rvert^{\mathfrak{p}\left(x\right)-2} \nabla_{\mathbb{E}} w , \nabla_{\mathbb{E}} \phi \right\rangle\ \mathrm{d}x = 0 \qquad \text{ for all } \phi \in \mathbb{E}W_{0}^{1,\mathfrak{p}(\cdot)}\left(B_{R}\right).  
			\end{align} Also, since $w$ achieves the minimum, we have 
			\begin{align}\label{minimality of w Global HI}
				\int_{B_{R}}\left\lvert \nabla_{\mathbb{E}} w \right\rvert^{\mathfrak{p}\left(x\right)}\ \mathrm{d}x \leq c(\gamma_1,\gamma_2) \int_{B_{R}}\left\lvert \nabla_{\mathbb{E}} u \right\rvert^{\mathfrak{p}\left(x\right)}\ \mathrm{d}x. 
			\end{align} Now we set 
			\begin{align*}
				\tilde{w}(x): = \left\lbrace \begin{aligned}
					&w(x) &&\text{ in } B_{R},\\
					&u(x) &&\text{ in } B_{2R}\setminus B_{R}.
				\end{aligned}\right. 
			\end{align*}
			We claim the following reverse H\"{o}lder inequality: 
			\begin{align}\label{rev holder}
				\fint_{B_{\rho}\left(x_{1}\right)} \left\lvert \nabla_{\mathbb{E}} \tilde{w} \right\rvert^{\mathfrak{p}(x)} \leq c\left( 1 + \fint_{B_{2\rho}\left(x_{1}\right)} \left\lvert \nabla_{\mathbb{E}} u \right\rvert^{\mathfrak{p}(x)} +\left(\fint_{B_{2\rho}\left(x_{1}\right)}  \left\lvert \nabla_{\mathbb{E}} \tilde{w} \right\rvert^{\frac{\mathfrak{p}\left(x\right)}{s}}\right)^{s}\right). 
			\end{align}
			whenever $x_{1} \in \overline{B_{R}}$ and $B_{2\rho}(x_{1}) \subset B_{2R},$ and for $s = \min \left\lbrace \gamma_{1}, 2Q_{\mathbb{E}}/(2Q_{\mathbb{E}} -1)\right\rbrace.$ Once this is established, using \cite[Theorem 3.3]{GheringZatorska-Goldstein} with `$q=s$'  we conclude that there exist a positive constant $c =c \left(\textnormal{\texttt{data}} \right),$ and $\sigma_{2} = \sigma_{2}\left(\textnormal{\texttt{data}}, \tilde{\sigma} \right) \in [0,\tilde{\sigma})$ such that for all $0 \leq \sigma \leq \sigma_{2},$ we have  
			\begin{align*}
				\left( \fint_{B_{R}} \left\lvert \nabla_{\mathbb{E}} \tilde{w} \right\rvert^{\mathfrak{p}(x)(1+\sigma)} \right)^{\frac{1}{(1+\sigma)}}  \hspace{-3pt}\leq c \fint_{B_{2R}}\left\lvert  \nabla_{\mathbb{E}} \tilde{w} \right\rvert^{\mathfrak{p}(x)} + c \left(1+ \fint_{B_{2R}} \left\lvert \nabla_{\mathbb{E}} u \right\rvert^{\mathfrak{p}(x)(1+\sigma)}\right)^{\frac{1}{(1+\sigma)}}.   
			\end{align*}
			By definition of $\tilde{w}$ and using \eqref{minimality of w Global HI}, we arrive at 
			\begin{align*}
				&\left( \fint_{B_{R}} \left\lvert \nabla_{\mathbb{E}} w \right\rvert^{\mathfrak{p}(x)(1+\sigma)} \right)^{\frac{1}{(1+\sigma)}} \\
				&\hspace{25pt}\leq c\fint_{B_{R}}\left\lvert  \nabla_{\mathbb{E}} w  \right\rvert^{\mathfrak{p}(x)}+ c\fint_{B_{2R}}\left\lvert  \nabla_{\mathbb{E}}  u \right\rvert^{\mathfrak{p}(x)} + c \left( 1+\fint_{B_{2R}} \left\lvert \X u \right\rvert^{\mathfrak{p}(x)(1+\sigma)}\right)^{\frac{1}{(1+\sigma)}} \\
				&\hspace{25pt}\leq c\fint_{B_{R}}\left\lvert  \nabla_{\mathbb{E}}  u \right\rvert^{\mathfrak{p}(x)} + c \left( 1+\fint_{B_{2R}} \left\lvert \nabla_{\mathbb{E}}  u \right\rvert^{\mathfrak{p}(x)(1+\sigma)}\right)^{\frac{1}{(1+\sigma)}}.
			\end{align*}
			This implies our result by H\"{o}lder inequality. Thus it remains to establish \eqref{rev holder}. By standard covering arguments, it is enough to show that estimates for the two cases: (i) when $x_{1} \in B_{R}$ with $\rho < \operatorname{dist}\left(x_{1}, \partial B_{R}\right)$ and (ii) when $x_{1} \in \partial B_{R}.$ We only show the last case, as the first one is analogous and simpler. 
			
			Let $x_1\in \partial B_{R}$ and $\rho >0$ be such that $B_{2\rho}(x_1) \subset B_{2R}$. We define 
			\begin{align*}
				\mathfrak{p}_{2}:= \max\limits_{x \in \overline{B_{2\rho}(x_{1})}} \mathfrak{p} (x) \qquad \text{ and } \qquad \mathfrak{p}_{1}:= \min\limits_{x \in \overline{B_{2\rho}(x_{1})}} \mathfrak{p} (x).
			\end{align*} Note that \eqref{smallness of wp Global HI} implies 
			\begin{align}\label{bound on p1 by p2 Global HI}
				\frac{\mathfrak{p}_{1}}{\mathfrak{p}_{2}} \geq 1-\frac{s}{2Q_{\mathbb{E}}}.
			\end{align}
			We set
			\begin{align}\label{p*1/s-nondiv Global HI}
				\left(\frac{\mathfrak{p}_1}{s}\right)^*:= \begin{cases}
					\frac{Q_{\mathbb{E}}\mathfrak{p}_1}{Q_{\mathbb{E}}s-\mathfrak{p}_1} \mbox{ if } \mathfrak{p}_1<Q_{\mathbb{E}}s, \\ 
					\mathfrak{p}_2+1 \mbox{ else .}
				\end{cases}
			\end{align}
			Then clearly, $\left(\frac{\mathfrak{p}_1}{s}\right)^*\geq \mathfrak{p}_2$. Indeed, in the case $\mathfrak{p}_1<Q_{\mathbb{E}}s$ this follows from the computation below. 
			\begin{align*}
				\left(\frac{\mathfrak{p}_1}{s}\right)^* \stackrel{\eqref{bound on p1 by p2 Global HI}}{\geq}  \mathfrak{p}_2\frac{Q_{\mathbb{E}}\left(1-\frac{s}{2Q_{\mathbb{E}}}\right)}{Q_{\mathbb{E}}s-s} \geq \mathfrak{p}_2,
			\end{align*}
			where we have used the bound $s \leq \min \left\lbrace \gamma_1, 2Q_{\mathbb{E}}/ (2Q_{\mathbb{E}}-1)\right\rbrace $.
			
			\par Now fix $0< \rho \leq R/2.$ Choose a cut off function $\eta \in C_{c}^{\infty}\left(B_{2\rho}\left(x_{1}\right)\right)$ such that $0 \leq \eta \leq 1,$ $\eta \equiv 1$ in $B_{\rho}\left(x_{1}\right)$ and $\left\lvert \nabla_{\mathbb{E}} \eta \right\rvert \leq c/\rho$ for some constant $c = c(n)>1.$ Now observe that as $2\rho \le R$ and $u=w$ on $\partial B_{R}$, we have $\phi = \eta^{\mathfrak{p}_2}\left( w - u \right) \in \mathbb{E}W_{0}^{1,\mathfrak{p}(\cdot)}\left(B_{R}\right)$. Thus, plugging $\phi$ in \eqref{equation frozen Global HI}, we deduce 
			\begin{align*}
				\int_{B_{2\rho}\left(x_{1}\right)\cap B_{R}} &\eta^{\mathfrak{p}_2}\left \langle \left \lvert \nabla_{\mathbb{E}} w \right\rvert^{\mathfrak{p}(x)-2} \nabla_{\mathbb{E}} w, \nabla_{\mathbb{E}} w -\nabla_{\mathbb{E}} u  \right\rangle\ \mathrm{d}x  \\
				&= - \int_{B_{2\rho}\left(x_{1}\right)\cap B_{R}} \left\langle \left\lvert \nabla_{\mathbb{E}} w \right\rvert^{\mathfrak{p}(x)-2} \nabla_{\mathbb{E}} w , \mathfrak{p}_2 \left( w - u \right) \eta^{\mathfrak{p}_2-1} \nabla_{\mathbb{E}} \eta \right\rangle\ \mathrm{d}x.  
			\end{align*}
			A fairly straightforward estimate using Young's inequality with $0<\varepsilon<1$ and using the estimates $\varepsilon^{\mathfrak{p}'(x)}, \varepsilon^{\mathfrak{p}_2} \leq \varepsilon^{\gamma_2'}$ and $\varepsilon^{-\mathfrak{p}(x)}, \varepsilon^{-\mathfrak{p}_2} \leq \varepsilon^{-\gamma_2}$, for any $x\in \Omega$,  we derive, 
			\begin{align*}
				&\int_{B_{2\rho}\left(x_{1}\right)\cap B_{R}} \eta^{\mathfrak{p}_2}\left\lvert \nabla_{\mathbb{E}} w \right\rvert^{\mathfrak{p}(x)}\ \mathrm{d}x\\
				&\quad\leq \frac{1}{\varepsilon^{\gamma_2}\gamma_1} \int_{B_{2 \rho}\left(x_{1}\right)\cap B_{R}} \eta^{\mathfrak{p}_2}\left\lvert \nabla_{\mathbb{E}} u \right\rvert^{\mathfrak{p}(x)}\ \mathrm{d}x + c\varepsilon^{\gamma'_2} \int_{B_{2\rho}\left(x_{1}\right)\cap B_{R}} \eta^{\mathfrak{p}_2}\left\lvert \nabla_{\mathbb{E}} w \right\rvert^{\mathfrak{p}(x)}\ \mathrm{d}x \\
				&\hspace{3cm}+ c \left \lvert B_{2\rho}(x_1)\cap B_R \right\rvert\ \mathrm{d}x  + \frac{\gamma_2}{\varepsilon^{\gamma_2}\gamma_1} \int_{B_{2\rho}\left(x_{1}\right)\cap B_{R}} \frac{\left\lvert  w - u \right\rvert^{\mathfrak{p}_2}}{\rho^{\mathfrak{p}_2}}\ \mathrm{d}x.	  
			\end{align*}
			Choosing $\varepsilon>0$ small enough we deduce 
			\begin{align*}
				\int_{B_{\rho}\left(x_{1}\right)\cap B_{R}} \left\lvert \nabla_{\mathbb{E}} w \right\rvert^{\mathfrak{p}(x)}
				\leq c \int_{B_{2\rho}\left(x_{1}\right)\cap B_{R}} \left\lvert \nabla_{\mathbb{E}} u \right\rvert^{\mathfrak{p}(x)}  &+ c \int_{B_{2\rho}\left(x_{1}\right)\cap B_{R}} \frac{\left\lvert  w - u \right\rvert^{\mathfrak{p}_2}}{\rho^{\mathfrak{p}_2}} \\ 
				&+  c \left \lvert B_{2\rho}(x_1)\cap B_R \right\rvert. 
			\end{align*}
			This implies 
			\begin{align}\label{boundary rev holder eq1}
				\fint_{B_{\rho}\left(x_{1}\right)\cap B_{R}} \left\lvert \nabla_{\mathbb{E}} w \right\rvert^{\mathfrak{p}(x)}\ \mathrm{d}x 
				&\leq c\left(1+ \fint_{B_{2\rho}\left(x_{1}\right)\cap B_{R}} \left\lvert \nabla_{\mathbb{E}} u \right\rvert^{\mathfrak{p}(x)}\ \mathrm{d}x\right)\notag \\ & \hspace{32pt}+  c \fint_{B_{2\rho}\left(x_{1}\right)\cap B_{R}}\frac{\left\lvert  w - u \right\rvert^{\mathfrak{p}_2}}{\rho^{\mathfrak{p}_2}}\ \mathrm{d}x. 
			\end{align} 
			For the last term on the right, we set $\psi= \tilde{w}-u$ and estimate 
			\begin{align}
				\fint_{B_{2\rho}\left(x_{1}\right)\cap B_{R}}\frac{\left\lvert  w - u \right\rvert^{\mathfrak{p}_2}}{\rho^{\mathfrak{p}_2}} &= \frac{1}{\left\lvert B_{2\rho}\left(x_{1}\right)\cap B_{R}\right\rvert}\int_{B_{2\rho}\left(x_{1}\right)\cap B_{R}}\left\lvert\frac{ w - u }{\rho}\right\rvert^{\mathfrak{p}_2}\ \mathrm{d}x \notag\\
				&=  \frac{\left\lvert B_{2\rho}\left(x_{1}\right)\right\rvert}{\left\lvert B_{2\rho}\left(x_{1}\right)\cap B_{R}\right\rvert}\fint_{B_{2\rho}\left(x_{1}\right)}\left\lvert\frac{ \psi}{\rho}\right\rvert^{\mathfrak{p}_2}\leq  c \fint_{B_{2\rho}\left(x_{1}\right)}\left\lvert\frac{ \psi}{\rho}\right\rvert^{\mathfrak{p}_2}, \label{boundary rev holder eq2}
			\end{align}
			where we have used the fact that $\tilde{w}-u=0$ outside $B_{R}$ and the obvious geometric fact that the ratio of the measures are bounded below by a constant which depends only on $n$, since there is a ball of radius $\rho$ which is contained inside $B_{2\rho}\left(x_{1}\right)\cap B_{R}.$ 
			Now we have the obvious estimate 
			\begin{align*}
				\left( \fint_{B_{2\rho}\left(x_{1}\right)}\left\lvert\frac{  \psi }{\rho}\right\rvert^{\mathfrak{p}_2}\right)^{\frac{1}{\mathfrak{p}_2}} \leq 	\left( \fint_{B_{2\rho}\left(x_{1}\right)}\left\lvert\frac{  \psi - \left(\psi\right)_{B_{2\rho}\left(x_{1}\right)} }{\rho}\right\rvert^{\mathfrak{p}_2}\right)^{\frac{1}{\mathfrak{p}_2}} + \left\lvert\frac{ \left(\psi\right)_{B_{2\rho}\left(x_{1}\right)} }{\rho}\right\rvert. 
			\end{align*}
			Since $\psi=0$ outside $B_{R},$ by H\"{o}lder inequality, we deduce 
			\begin{align*}
				\left\lvert\frac{ \left(\psi\right)_{B_{2\rho}\left(x_{1}\right)} }{\rho}\right\rvert &\leq \fint_{B_{2\rho}\left(x_{1}\right)}\left\lvert\frac{ \psi}{\rho}\right\rvert \\&=\frac{1}{\left\lvert B_{2\rho}\left(x_{1}\right)\right\rvert}\int_{B_{2\rho}\left(x_{1}\right)\cap B_{R}}\left\lvert\frac{ \psi}{\rho}\right\rvert \\
				&\leq  \frac{1}{\left\lvert B_{2\rho}\left(x_{1}\right)\right\rvert}\left(\int_{B_{2\rho}\left(x_{1}\right)\cap B_{R}}\left\lvert\frac{ \psi }{\rho}\right\rvert^{\mathfrak{p}_2} \right)^{\frac{1}{\mathfrak{p}_2}} \left\lvert B_{2\rho}\left(x_{1}\right)\cap B_{R}\right\rvert^{1-\frac{1}{\mathfrak{p}_2}} \\
				&=\left( \frac{\left\lvert B_{2\rho}\left(x_{1}\right)\cap B_{R}\right\rvert}{\left\lvert B_{2\rho}\left(x_{1}\right)\right\rvert}\right)^{1-\frac{1}{\mathfrak{p}_2}}\left(\fint_{B_{2\rho}\left(x_{1}\right)}\left\lvert\frac{ \psi }{\rho}\right\rvert^{\mathfrak{p}_2} \right)^{\frac{1}{\mathfrak{p}_2}}  \\
				&\leq \left( \frac{1}{2}\right)^{1-\frac{1}{\mathfrak{p}_2}}\left(\fint_{B_{2\rho}\left(x_{1}\right)}\left\lvert\frac{ \psi }{\rho}\right\rvert^{\mathfrak{p}_2} \right)^{\frac{1}{\mathfrak{p}_2}},
			\end{align*}
			where we have used the elementary estimate $\left\lvert B_{2\rho}\left(x_{1}\right)\cap B_{R}\right\rvert \leq \frac{1}{2}\left\lvert B_{2\rho}\left(x_{1}\right)\right\rvert$ for metric balls which is valid since $x_{1} \in \partial B_{R}$. Combining the last two displays  we obtain 
			\begin{align*}
				\fint_{B_{2\rho}\left(x_{1}\right)}\left\lvert\frac{  \psi }{\rho}\right\rvert^{\mathfrak{p}_2} \leq c  \fint_{B_{2\rho}\left(x_{1}\right)}\left\lvert\frac{  \psi - \left(\psi\right)_{B_{2\rho}\left(x_{1}\right)} }{\rho}\right\rvert^{\mathfrak{p}_2}.
			\end{align*}
			From this, using the Poincar\'e-Sobolev inequality, the equation \eqref{boundary rev holder eq1}, and the assumption $0<R<1/\mathcal{K}_1​$ combined with \eqref{log holder bound Global HI}, we can deduce that
			\begin{align*}
				\fint_{B_{2\rho}\left(x_{1}\right)}\left\lvert\frac{  \psi }{\rho}\right\rvert^{\mathfrak{p}_2} 
				&\leq c\left( 1 + \fint_{B_{2\rho}} \left\lvert \nabla_{\mathbb{E}} \psi \right\rvert^{\frac{\mathfrak{p}\left(x\right)}{s}}\right)^{s} \\
				&\leq c\left( 1 + \fint_{B_{2\rho}} \left\lvert \nabla_{\mathbb{E}} u \right\rvert^{\mathfrak{p}(x)} +\left(\fint_{B_{2\rho}}  \left\lvert \nabla_{\mathbb{E}} \tilde{w} \right\rvert^{\frac{\mathfrak{p}\left(x\right)}{s}}\right)^{s}\right). 						
			\end{align*}
			Also, note that 
			\begin{align*}
				\fint_{B_{\rho}(x_{1})} \lvert \nabla_{\mathbb{E}} \tilde{w} \rvert^{\mathfrak{p}(x)}
				&\leq\fint_{B_{\rho}(x_1)\cap B_R} \left\lvert \nabla_{\mathbb{E}} w \right\rvert^{\mathfrak{p}(x)} +   \frac{1}{\left \lvert B_\rho(x_1) \right \rvert} \int_{B_{\rho}(x_1)\cap[B_{2R}\setminus B_R]} \left\lvert  \nabla_{\mathbb{E}} u \right\rvert^{\mathfrak{p}(x)} \\
				&\leq\fint_{B_{\rho}(x_1)\cap B_R} \left\lvert \nabla_{\mathbb{E}} w \right\rvert^{\mathfrak{p}(x)} +  c \fint_{B_{2\rho}(x_1)} \left\lvert  \nabla_{\mathbb{E}} u \right\rvert^{\mathfrak{p}(x)} .
			\end{align*}
			Combining these two estimates with \eqref{boundary rev holder eq1} and \eqref{boundary rev holder eq2}, and the fact that the ratio  $\left \lvert B_{2\rho}(x_1) \right \rvert$ and $\left \lvert B_{2\rho}(x_1) \cap B_R\right \rvert$  is bounded by a constant depending only on $Q_{\mathbb{E}}$, we deduce  \eqref{rev holder}. This completes the proof. \end{proof}
		
		\section{Comparison estimates} 
		\subsection{Basic setup}\label{basic setup for comparison}
		For this entire section, we shall always assume that $\Omega \subset \mathbb{E}^{n}$ is open and bounded and $\mathfrak{p}:\Omega \rightarrow [\gamma_{1}, \gamma_{2}]$, $\mathfrak{a}:\Omega \rightarrow [\nu, L]$ are measurable functions  with $1 < \gamma_{1} \leq \gamma_{2} < \infty$ and $ 0 < \nu < L < \infty,$ such that for all $\sigma \in (0, 1/4)$ and for all $1 < q< \infty,$ we have 
		\begin{align*}
			\mathfrak{S}_{\left[\mathfrak{a}, q, \sigma, 0\right]}\left(x_{0}, r \right), \mathfrak{S}_{\left[\mathfrak{p}, q, \sigma, 1\right]}\left(x_{0}, r \right) &\rightarrow 0	 \quad \text{ as } \bar{r} \rightarrow 0 \quad  \text{ locally uniformly in } x_{0} \in \Omega.	
		\end{align*}
		Note that, this implies $p$ satisfies \eqref{log holder bound Global HI} for any subdomain $\Omega' \subset \subset \Omega.$ 
		\subsubsection{Setup for equation with data: } Let $u \in 	\mathbb{E}W_{\text{loc}}^{1, \mathfrak{p}(\cdot)}\left(\Omega\right)$ be a weak solution to the following: 
		\begin{align}\label{main equation stein}
			\operatorname{div}_{\mathbb{E}}\left[ \mathfrak{a}\left(x\right)\left\lvert \nabla_{\mathbb{E}} u \right\rvert^{\mathfrak{p}\left(x\right)-2} \nabla_{\mathbb{E}} u\right]  &= f &&\text{ in } \Omega,
		\end{align}
		where $f \in L_{\text{loc}}^{\left(Q_{\mathbb{E}}, 1\right)}\left(\Omega\right).$ Since all the estimate we would prove are local, we can assume that $u \in \mathbb{E}W^{1, \mathfrak{p}(\cdot)}\left(\Omega\right)$ is a weak solution of \eqref{main equation stein}, $f \in L^{\left(Q_{\mathbb{E}}, 1\right)}\left(\Omega\right)$ and $\mathfrak{p}$ is log-H\"{o}lder continuous in $\Omega.$ Let $\gamma$ and $\mathcal{K}_{1}$ be as in \eqref{def gamma} and  \eqref{energy bound f} respectively. Let $\sigma_{1}, R_{1}$ be given by Theorem \ref{higher integrability f}. Let $\sigma_2$, $ R_2$ be given by Theorem \ref{global_higher_integrability} with the choice $\mathcal{K}_{2} = \mathcal{K}_{1}$ and $\tilde{\sigma} = \sigma_{1}$. We set 
		\begin{align}
			\bar{\sigma}:= \min \left\lbrace \gamma_1-1,  \sigma_{1}, \sigma_{2} \right\rbrace \label{sigma bar and R bar} \qquad \text{ and } \qquad \bar{R} := \min \left\lbrace R_{1}, R_{2} \right\rbrace. 
		\end{align} Now fix 
		\begin{align}\label{theta def}
		1 < \theta \le \frac{Q_{\mathbb{E}}\gamma}{Q_{\mathbb{E}}\gamma -Q_{\mathbb{E}} + \gamma}\left(1 + \bar{\sigma}\right) = \left(\gamma^{\ast}\right)'\left(1+\bar{\sigma}\right) <Q_{\mathbb{E}}. 
		\end{align}
		Now fix a point $x_{0} \in \Omega$ such that $B_{4\bar{R}}(x_{0}) \subset  \Omega.$ From now on, all balls would be centered at $x_{0}$ if the center is not specified and we omit writing the center of the balls. 
		\par Note that, by Theorem \ref{higher integrability f}, $u\in \mathbb{E}W^{1,\mathfrak{p}(\cdot)(1+\sigma_{1})}(B_{\bar{R}})$ and hence satisfies all the assumptions of Theorem \ref{global_higher_integrability} with $\tilde{\sigma} = \bar{\sigma}$, $\mathcal{K}_{2}= \mathcal{K}_{1}$ and $R=\bar{R}/64$.  
		For any $R >0,$ we would use the following shorthands. 
		\begin{align}
			\mathfrak{A}_{R, \sigma}&:=  \left(\fint_{B_{R}} \left\lvert \mathfrak{a} - \left(\mathfrak{a}\right)_{R} \right\rvert^{\frac{2(1+\sigma)}{\sigma}}\right)^{\frac{\sigma}{2(1+\sigma)}}, \label{mathfrak a}\\   \mathfrak{P}_{R, \sigma}&:= \log \left(\frac{1}{R}\right)\left(\fint_{B_{R}} \left\lvert \mathfrak{p} - \left(\mathfrak{p}\right)_{R} \right\rvert^{\frac{2(1+\sigma)}{\sigma}}\right)^{\frac{\sigma}{2(1+\sigma)}},  \label{mathfrak p} \\ \mathfrak{F}_{R, \theta}&:= R \left( \fint_{B_{R}} \left\lvert f \right\rvert^{\theta} \right)^{\frac{1}{\theta}}. \label{mathfrak f} 
		\end{align}
		Now, choose a radius $ 0 < R_{4}\le  \bar{R}/2^{16}$ so small such that 
		\begin{align}\label{smallness of wp}
			\omega_{\mathfrak{p}}\left(R_{4}\right) \leq \min \left\lbrace \frac{ \bar{\sigma}}{4}, \frac{\gamma \bar{\sigma}}{4\gamma^{'}\left(2 + \bar{\sigma}\right)}, \frac{\gamma_1}{2Q_{\mathbb{E}}}\right\rbrace 
		\end{align}
		and  
		\begin{align}\label{smallness of mathfrak a+mathfrak p}
			\mathfrak{A}_{R, \bar{\sigma}} +   \mathfrak{P}_{R, \bar{\sigma}} \leq 1 \qquad \text{ for all } 0 < R \leq \frac{R_{4}}{4}.
		\end{align}
		Clearly $R_4$ depends on $Q_{\mathbb{E}}, \gamma_{1},$ $\gamma_{2}, \nu, L, \Lambda_{\log}$ and $\mathcal{K}_{1}$, where $ \Lambda_{\log} >0$ is any number satisfying \eqref{log holder bound Global HI}. 
		\subsubsection{Setup for equation with frozen coefficient: }Set $\mathfrak{p}_{0}:= \mathfrak{p}(x_{0}).$ Pick $0 < R < R_4/32$ for now and let $w \in u+\mathbb{E}W_{0}^{1, \mathfrak{p}(\cdot)}\left(B_{R}\right)$ be the unique minimizer of 
		\begin{align*}
			\inf \left\lbrace \mathfrak{a}_{R}\int_{B_{R}} \frac{1}{\mathfrak{p}(x)}\left\lvert \nabla_{\mathbb{E}} w\right\rvert^{\mathfrak{p}\left(x\right)}\ \mathrm{d}x: w \in u + \mathbb{E}W_{0}^{1, \mathfrak{p}(\cdot)}\left(B_{R}\right) \right\rbrace,  
		\end{align*}
		where $$\mathfrak{a}_{R}:= \fint_{B_{R}} \mathfrak{a}. $$ Note that $w$ is a weak solution to 
			\begin{align}\label{equation frozen}
				\operatorname{div}_{\mathbb{E}}\left[\left\lvert \nabla_{\mathbb{E}} w \right\rvert^{\mathfrak{p}\left(x\right)-2} \nabla_{\mathbb{E}} w \right]  &= 0 &&\text{ in } B_{R}. 
		\end{align} 
		Clearly by minimality of $w$, we have from some constant $c(\gamma_1, \gamma_2)>1$ 
		\begin{align}\label{minimality of w}
			\int_{B_{R}}\left\lvert \nabla_{\mathbb{E}} w \right\rvert^{\mathfrak{p}\left(x\right)}\ \mathrm{d}x \leq c(\gamma_1, \gamma_2) 	\int_{B_{R}}\left\lvert \nabla_{\mathbb{E}} u \right\rvert^{\mathfrak{p}\left(x\right)}\ \mathrm{d}x. 
		\end{align} 
		Also, by Theorem \ref{global_higher_integrability} we have 
		\begin{align}\label{higher integrability estimate f frozen and without gamma Global HI}
			\left( \fint_{B_{R}}\left\lvert \nabla_{\mathbb{E}} w \right\rvert^{\mathfrak{p}\left(x\right)\left(1 + \sigma\right)}\ \mathrm{d}x\right)^{\frac{1}{ 1+\sigma }}
			\leq c \left[ \left(\fint_{B_{2R}}\left\lvert \nabla_{\mathbb{E}} u \right\rvert^{\mathfrak{p}(x)(1+\sigma)}\ \mathrm{d}x\right)^{\frac{1}{1+\sigma}}  + 1 \right],
		\end{align}
		for any $0\leq \sigma \leq \bar{\sigma}$. 
		\subsubsection{Setup for equation with frozen exponent and coefficient: }Clearly, the smallness in \eqref{smallness of wp} implies that 
		$$\mathfrak{p}_{R/2} - \inf_{B_{R/2}} \mathfrak{p} (\cdot) \leq \omega_{\mathfrak{p}}(\bar{R}) \leq  \bar{\sigma} \leq \gamma_1 \bar{\sigma} \implies \mathfrak{p}_{R/2} \leq \mathfrak{p}(x) (1+\bar{\sigma}), \mbox{ for all } x\in \overline{B_{R/2}}.$$  
		Here $\mathfrak{p}_{R/2}:= \fint_{B_{R/2}} \mathfrak{p}. $
		Thus by \eqref{higher integrability estimate f frozen and without gamma Global HI}, $w\in \mathbb{E}W^{1,\mathfrak{p}_{R/2}}\left(B_{R/2}\right)$.  Hence, let $v \in w+\mathbb{E}W^{1,\mathfrak{p}_{R/2}}\left(B_{R/2}\right)$ be the unique minimizer of 
		\begin{align*}
			\inf \left\lbrace \int_{B_{R/2}} \frac{1}{\mathfrak{p}_{R/2}}\left\lvert \nabla_{\mathbb{E}} w\right\rvert^{\mathfrak{p}_{R/2}}\ \mathrm{d}x: v \in w + \mathbb{E}W_{0}^{1,\mathfrak{p}_{R/2}}\left(B_{R/2}\right) \right\rbrace.  
		\end{align*}
		Note that $v$ is a weak solution to 
		\begin{align}\label{equation double frozen}
			\operatorname{div}_{\mathbb{H}}\left(\left\lvert \nabla_{\mathbb{E}} v \right\rvert^{\mathfrak{p}_{R/2}-2} \nabla_{\mathbb{E}} v \right) &= 0 &&\text{ in } B_{R/2}
		\end{align}	
		and $v-w$ is an admissible test function.

		\subsection{Preliminary comparison estimates}
		
		We begin with the following comparison estimate.	
		\begin{lemma}[Comparison Lemma I]\label{comparison lemma 1}
			Let $u$ and $w$ be as defined in \eqref{main equation stein} and \eqref{equation frozen} respectively. Also, let  $R_3$, $\theta$, $\gamma$ and $\bar{\sigma}$ be as in \eqref{smallness of wp} and \eqref{smallness of mathfrak a+mathfrak p}, \eqref{theta def}, \eqref{def gamma} and  \eqref{sigma bar and R bar}, respectively. For any $\bar{x} \in \overline{B_{R/2}},$ let us set $\bar{\mathfrak{p}}:= \mathfrak{p}\left( \bar{x}\right)$ and let $\lambda_{1} $ be defined by  
			\begin{align}\label{def lambda1}
			\lambda_1\coloneqq 1 + \left(  \fint_{B_{4R}}\left\lvert  \nabla_{\mathbb{E}} u \right\rvert^{\gamma} \right)^{\frac{1}{\gamma}} +  \mathfrak{F}_{4R, \theta}^{\frac{1}{\mathfrak{p}_{0}-1}}. 
			\end{align} 
			Then there exists a constant $ C_{6} \equiv C_{6} \left(Q_{\mathbb{E}}, \gamma_{1}, \gamma_{2}, \nu, L, \Lambda_{\log} \right) \ge 1$ for which  following estimates are true for any $0<R\leq R_3/64$. 
			\begin{enumerate}[(i)]
				\item For any $\bar{\mathfrak{p}}$, we have \begin{align}\label{comparison lemma 1 main estimate}
					\fint_{B_R} &\left\lvert V_{\bar{\mathfrak{p}}}\left( \nabla_{\mathbb{E}} u\right) - V_{\bar{\mathfrak{p}}}\left( \nabla_{\mathbb{E}} w\right)\right\rvert^{2} \notag\\
					&\leq 	C_{6} \mathfrak{F}_{R, \theta}\left(\fint_{B_{R}}\left\lvert \nabla_{\mathbb{E}}u -\nabla_{\mathbb{E}} w\right\rvert^{\gamma} \right)^{\frac{1}{\gamma}}  + C_{6}\left( \mathfrak{A}_{R, \bar{\sigma}}^{2} + \mathfrak{P}_{R, \bar{\sigma}}^{2} \right) \lambda_{1}^{\bar{\mathfrak{p}}}.
				\end{align}
				\item If $\bar{\mathfrak{p}} >2,$ we have 
				\begin{align}\label{comparison lemma 1 bar p bigger than 2}
					\fint_{B_{R}} \left\lvert  \nabla_{\mathbb{E}} u - \nabla_{\mathbb{E}} w\right\rvert^{\bar{\mathfrak{p}}} \leq 	C_{6} \left( \mathfrak{F}_{R, \theta} \right)^{\frac{\bar{\mathfrak{p}}}{\left( \bar{\mathfrak{p}}-1\right)}}  + C_{6}\left( \mathfrak{A}_{R, \bar{\sigma}} + \mathfrak{P}_{R, \bar{\sigma}} \right)^{2} \lambda_{1}^{\bar{\mathfrak{p}}}.
				\end{align}
				\item If $\bar{\mathfrak{p}} \leq 2,$ we have 
				\begin{align}\label{comparison lemma 1 bar p smaller than 2}
					\fint_{B_{R}} \left\lvert  \nabla_{\mathbb{E}} u - \nabla_{\mathbb{E}} w\right\rvert^{\bar{\mathfrak{p}}} \leq 	C_{6} \lambda_{1}^{\bar{\mathfrak{p}}\left( 2-\bar{\mathfrak{p}}\right)}\left( \mathfrak{F}_{R, \theta} \right)^{\bar{\mathfrak{p}}}  + C_{6}\left( \mathfrak{A}_{R, \bar{\sigma}} + \mathfrak{P}_{R, \bar{\sigma}} \right)^{\bar{\mathfrak{p}}} \lambda_{1}^{\bar{\mathfrak{p}}}.
				\end{align}
			\end{enumerate}
		\end{lemma}
		\begin{proof} 
			We begin by estimating $\lambda_1$. Let $\mathcal{K}_{1}>1$ be as in \eqref{energy bound f}. We use $R<1$, $q<Q_{\mathbb{E}}$, $1<\gamma < \mathfrak{p}(x)$ for any $x\in \Omega$ and H\"older inequality to derive  
			\begin{align}\label{estimate of lambda1}
				\lambda_1 \leq c\left( 1 + \mathcal{K}_1^\frac{1}{\gamma}R^{-\frac{Q_{\mathbb{E}}}{\gamma}} + \mathcal{K}_1^{\frac{1}{Q_{\mathbb{E}}(\mathfrak{p}_0-1)}}\right) \leq c\mathcal{K}_1 R^{-Q_{\mathbb{E}}}.   
			\end{align}
			Using this estimate for $\lambda_1$ and \eqref{higher integrability estimate f frozen and without gamma Global HI} we have
			\begin{align}
				\left( \fint_{B_{R}}\left\lvert \nabla_{\mathbb{E}} w \right\rvert^{\mathfrak{p}\left(x\right)\left(1 + \sigma\right)}\ \mathrm{d}x\right)^{\frac{1}{ 1+\sigma }}
				&\leq c \left[ 1 + \left( \fint_{B_{2R}}\left\lvert \nabla_{\mathbb{E}} u \right\rvert^{\mathfrak{p}(x)(1+\sigma)}\ \mathrm{d}x \right)^{\frac{1}{1+\sigma}}  \right]\notag	\\
				&\stackrel{\eqref{higher integrability estimate f}}{\leq} c \lambda_1^{\bar{\mathfrak{p}}}\leq \left( cR^{-Q_{\mathbb{E}}}\mathcal{K}_{1} \right)^{\bar{\mathfrak{p}}} \leq c R^{-\bar{\mathfrak{p}}\left(Q_{\mathbb{E}}+1\right)}, \label{estimate for w in terms of u}
			\end{align} 
			where to get the last estimate we used the bound $\mathcal{K}_{1}< 1/R$. Note that, $w-u \in \mathbb{E}W_0^{1, \mathfrak{p}(\cdot)}(B_R)$. So, plugging in $w-u$ in the weak formulations of \eqref{equation frozen} and \eqref{main equation stein} and subtracting, we deduce 
			\begin{align*}
				\int_{B_{R}} \left\langle \mathfrak{a}\left(x\right)\left\lvert \nabla_{\mathbb{E}} u \right\rvert^{\mathfrak{p}\left(x\right)-2} \nabla_{\mathbb{E}} u - \mathfrak{a}_{R}\left\lvert \nabla_{\mathbb{E}} w \right\rvert^{\mathfrak{p}\left(x\right)-2} \nabla_{\mathbb{E}} w, \nabla_{\mathbb{E}} u -\nabla_{\mathbb{E}} w \right\rangle   = \int_{B_{R}} f \left( u -w\right). 
			\end{align*}
			Hence using the definition of $V$ and uniform bounds on $\mathfrak{a}$ we have 
			\begin{align}\label{I1 and I2 in first comparison}
				\fint_{B_{R}} &\left\lvert V_{\mathfrak{p}(\cdot)}\left( \nabla_{\mathbb{E}} u\right) - V_{\mathfrak{p}(\cdot)}\left( \nabla_{\mathbb{E}} w\right)\right\rvert^{2}\notag  \\&\leq c \fint_{B_{R}} \left\langle \mathfrak{a}_{R} \left\lvert \nabla_{\mathbb{E}} u \right\rvert^{\mathfrak{p}\left(x\right)-2} \nabla_{\mathbb{E}} u - \mathfrak{a}_{R}\left\lvert \nabla_{\mathbb{E}} w \right\rvert^{\mathfrak{p}\left(x\right)-2} \nabla_{\mathbb{E}} w, \nabla_{\mathbb{E}} u -\nabla_{\mathbb{E}} w \right\rangle \notag
				\\&\leq  c\fint_{B_{R}}  \left\lvert  \mathfrak{a}(x) - \mathfrak{a}_{R}\right\rvert  \left\lvert \nabla_{\mathbb{E}} u \right\rvert^{\mathfrak{p}\left(x\right)-1} \left\lvert \nabla_{\mathbb{E}} u -\nabla_{\mathbb{E}} w \right\rvert +c \fint_{B_{R}} \left\lvert f \right\rvert \left\lvert u -w\right\rvert :=I_{1} + I_{2}. 
			\end{align}
			Let $\varepsilon>0$ and fix $x\in B_{R}$. We first estimate $I_{1}.$  
			If $\mathfrak{p}(x) \geq 2,$ then we have 
			\begin{align*}
				\left\lvert  \mathfrak{a}(x) - \mathfrak{a}_{R}\right\rvert & \left\lvert \nabla_{\mathbb{E}} u \right\rvert^{\mathfrak{p}\left(x\right)-1} \left\lvert \nabla_{\mathbb{E}} u -\nabla_{\mathbb{E}} w \right\rvert\\
				&\leq \left\lvert  \mathfrak{a}(x) - \mathfrak{a}_{R}\right\rvert \left\lvert \nabla_{\mathbb{E}} u \right\rvert^{\frac{\mathfrak{p}\left(x\right)}{2}} \left(  \left\lvert \nabla_{\mathbb{E}} w \right\rvert +   \left\lvert \nabla_{\mathbb{E}} u \right\rvert  \right)^{\frac{\mathfrak{p}\left(x\right)-2}{2}}\left\lvert \nabla_{\mathbb{E}} u -\nabla_{\mathbb{E}} w \right\rvert 
				\\&\leq \varepsilon \left\lvert V_{\mathfrak{p}(x)}\left( \nabla_{\mathbb{E}} u\right) - V_{\mathfrak{p}(x)}\left( \nabla_{\mathbb{E}} w\right)\right\rvert^{2} + c 	\left\lvert  \mathfrak{a}(x) - \mathfrak{a}_{R}\right\rvert^{2} \left\lvert \nabla_{\mathbb{E}} u \right\rvert^{\mathfrak{p}\left(x\right)}. 
			\end{align*}	
			If $\mathfrak{p}(x) < 2,$ then we use \eqref{v estimate p less 2} and the uniform bounds on $\mathfrak{p}$ to deduce 
			\begin{align*}
				\left\lvert  \mathfrak{a}(x) - \mathfrak{a}_{R}\right\rvert  &\left\lvert \nabla_{\mathbb{E}} u \right\rvert^{\mathfrak{p}\left(x\right)-1} \left\lvert \nabla_{\mathbb{E}} u -\nabla_{\mathbb{E}} w \right\rvert \\
				&\begin{aligned}[t]
					\leq c\left\lvert  \mathfrak{a}(x) - \mathfrak{a}_{R}\right\rvert&\left\lvert \nabla_{\mathbb{E}} u \right\rvert^{\mathfrak{p}\left(x\right)-1}\left( \left\lvert V_{\mathfrak{p}(x)}\nabla_{\mathbb{E}} u\right) - V_{\mathfrak{p}(x)}\left( \nabla_{\mathbb{E}} w\right)\right\rvert^{\frac{2}{\mathfrak{p}(x)}} \\ &+ c\left\lvert  \mathfrak{a}(x) - \mathfrak{a}_{R}\right\rvert\left\lvert \nabla_{\mathbb{E}} u \right\rvert^{\frac{\mathfrak{p}\left(x\right)}{2}}\left\lvert V_{\mathfrak{p}(x)}\left( \nabla_{\mathbb{E}} u\right) - V_{\mathfrak{p}(x)}\left( \nabla_{\mathbb{E}} w\right)\right\rvert  
				\end{aligned}
				\\&\begin{multlined}
					\leq \varepsilon \left\lvert V_{\mathfrak{p}(x)}\left( \nabla_{\mathbb{E}} u\right) - V_{\mathfrak{p}(x)}\left( \nabla_{\mathbb{E}} w\right)\right\rvert^{2} + c 	\left\lvert  \mathfrak{a}(x) - \mathfrak{a}_{R}\right\rvert^{2} \left\lvert \nabla_{\mathbb{E}} u \right\rvert^{\mathfrak{p}\left(x\right)} \\+ c \left\lvert  \mathfrak{a}(x) - \mathfrak{a}_{R}\right\rvert^{\frac{\mathfrak{p}(x)}{\mathfrak{p}(x)-1}} \left\lvert \nabla_{\mathbb{E}} u \right\rvert^{\mathfrak{p}\left(x\right)}.
				\end{multlined} 
			\end{align*}
			Now since $\mathfrak{p}(x) < 2,$ we have $\frac{\mathfrak{p}(x)}{\mathfrak{p}(x)-1} >2 $	and thus we have 
			\begin{align*}
				\left\lvert  \mathfrak{a}(x) - \mathfrak{a}_{R}\right\rvert^{\frac{\mathfrak{p}(x)}{\mathfrak{p}(x)-1}} \left\lvert \nabla_{\mathbb{E}} u \right\rvert^{\mathfrak{p}\left(x\right)} 
				&\leq \left\lVert \mathfrak{a} \right\rVert_{L^{\infty}}^{\left(\frac{\mathfrak{p}(x)}{\mathfrak{p}(x)-1} -2\right)}\left\lvert  \mathfrak{a}(x) - \mathfrak{a}_{R}\right\rvert^{2} \left\lvert \nabla_{\mathbb{E}} u \right\rvert^{\mathfrak{p}\left(x\right)}\\ 
				&\leq \left\lVert a \right\rVert_{L^{\infty}}^{\left(\frac{\gamma_{1}}{\gamma_{1}-1} -2\right)}\left\lvert  a(x) - a_{R}\right\rvert^{2} \left\lvert \nabla_{\mathbb{E}} u \right\rvert^{\mathfrak{p}\left(x\right)}. 
			\end{align*}
			Thus, in either case, for any $x\in B_{R}$, we have the pointwise inequality 	
			\begin{align*}
				\left\lvert  \mathfrak{a}(x) - \mathfrak{a}_{R}\right\rvert & \left\lvert \nabla_{\mathbb{E}} u \right\rvert^{\mathfrak{p}\left(x\right)-1} \left\lvert \nabla_{\mathbb{E}} u -\nabla_{\mathbb{E}} w \right\rvert\\
				&\leq \varepsilon \left\lvert V_{\mathfrak{p}(x)}\left( \nabla_{\mathbb{E}} u\right) - V_{\mathfrak{p}(x)}\left( \nabla_{\mathbb{E}} w\right)\right\rvert^{2} + c 	\left\lvert  \mathfrak{a}(x) - \mathfrak{a}_{R}\right\rvert^{2} \left\lvert \nabla_{\mathbb{E}} u \right\rvert^{\mathfrak{p}\left(x\right)}, 
			\end{align*}
			where $c$ depends only on $\varepsilon>0$ and $\gamma_{1}, \gamma_{2},$ but not on $x.$ Integrating, we deduce 
			\begin{align*}
				I_{1} \leq \varepsilon	\fint_{B_{R}} \left\lvert V_{\mathfrak{p}(\cdot)}\left( \nabla_{\mathbb{E}} u\right) - V_{\mathfrak{p}(\cdot)}\left( \nabla_{\mathbb{E}} w\right)\right\rvert^{2}  + c 	\fint_{B_{R}}\left\lvert  \mathfrak{a}(x) - \mathfrak{a}_{R}\right\rvert^{2} \left\lvert \nabla_{\mathbb{E}} u \right\rvert^{\mathfrak{p}\left(x\right)}. 
			\end{align*}
			Next, we estimate $I_{2}.$ We use H\"older's inequality, \eqref{theta def} and Sobolev's inequality to derive
			\begin{align*}
				I_{2}= \fint_{B_{R}} \left\lvert f \right\rvert \left\lvert u -w\right\rvert &\leq \left(\fint_{B_{R}} \left\lvert f \right\rvert^{\left(\gamma^{\ast}\right)^{'}} \right)^{\frac{1}{\left(\gamma^{\ast}\right)^{'}}} \left(\fint_{B_{R}}\left\lvert u -w\right\rvert^{\gamma^{\ast}} \right)^{\frac{1}{\gamma^{\ast}}} \\
				&\leq  cR\left(\fint_{B_{R}} \left\lvert f \right\rvert^{\left(\gamma^{\ast}\right)^{'}} \right)^{\frac{1}{\left(\gamma^{\ast}\right)^{'}}} \left(\fint_{B_{R}}\left\lvert \nabla_{\mathbb{E}}u -\nabla_{\mathbb{E}} w\right\rvert^{\gamma} \right)^{\frac{1}{\gamma}} \\
				&\leq c \mathfrak{F}_{R, \theta} \left(\fint_{B_{R}}\left\lvert \nabla_{\mathbb{E}}u -\nabla_{\mathbb{E}} w\right\rvert^{\gamma} \right)^{\frac{1}{\gamma}} . 
			\end{align*}	
			Combing the estimates of $I_1$ and $I_2$ in \eqref{I1 and I2 in first comparison} and choosing $\varepsilon>0$ small enough we deduce  
			\begin{align} \label{before smallness of a est}
				\fint_{B_{R}}& \left\lvert V_{\mathfrak{p}(\cdot)}\left( \nabla_{\mathbb{E}} u\right) - V_{\mathfrak{p}(\cdot)}\left( \nabla_{\mathbb{E}} w\right)\right\rvert^{2} \notag \\ 
				&\leq c\mathfrak{F}_{R, \theta} \left(\fint_{B_{R}}\left\lvert \nabla_{\mathbb{E}}u -\nabla_{\mathbb{E}} w\right\rvert^{\gamma} \right)^{\frac{1}{\gamma}} + c \fint_{B_{R}}\left\lvert  \mathfrak{a}(x) - \mathfrak{a}_{R}\right\rvert^{2} \left\lvert \nabla_{\mathbb{E}} u \right\rvert^{\mathfrak{p}\left(x\right)}. 
			\end{align}
			Now, using H\"older's inequality, we have 
			\begin{align*}
				\fint_{B_{R}}&\left\lvert  \mathfrak{a}(x) - \mathfrak{a}_{R}\right\rvert^{2} \left\lvert \nabla_{\mathbb{E}} u \right\rvert^{\mathfrak{p}\left(x\right)} \\
				&\leq \left(  \fint_{B_{R}}\left\lvert  \mathfrak{a}(x) - \mathfrak{a}_{R}\right\rvert^{\frac{2 \left(1+\bar{\sigma}\right)}{\bar{\sigma}}} \right)^{\frac{\bar{\sigma}}{1+\bar{\sigma}}} \left(  \fint_{B_{R}}\left\lvert \nabla_{\mathbb{E}} u \right\rvert^{\mathfrak{p}\left(x\right)\left(1+\bar{\sigma}\right)} \right)^{\frac{1}{1+\bar{\sigma}}}. 	\end{align*}
                We use \eqref{higher integrability estimate f} to estimate  
                \begin{align*}
                  \fint_{B_{R}}&\left\lvert  \mathfrak{a}(x) - \mathfrak{a}_{R}\right\rvert^{2} \left\lvert \nabla_{\mathbb{E}} u \right\rvert^{\mathfrak{p}\left(x\right)} \\
				&\leq \mathfrak{A}_{R, \bar{\sigma}}^{2}\left(  \fint_{B_{R}}\left\lvert \nabla_{\mathbb{E}} u \right\rvert^{\mathfrak{p}\left(x\right)\left(1+\bar{\sigma}\right)} \right)^{\frac{1}{1+\bar{\sigma}}} \\  
				&\leq c \mathfrak{A}_{R, \bar{\sigma}}^{2}\left[ \left(\fint_{B_{2R}}\left\lvert \nabla_{\mathbb{E}}u \right\rvert^{\gamma} \right)^{\frac{1}{\gamma}} + \left( R^{\theta}\fint_{B_{2R}} \left\lvert f \right\rvert^{\theta} \right)^{\frac{1}{\theta\left(\mathfrak{p}(x_{0}) -1\right)}}+1\right]^{\bar{\mathfrak{p}}} \leq c \mathfrak{A}_{R, \bar{\sigma}}^{2}\lambda_{1}^{\bar{\mathfrak{p}}}.
			\end{align*}
			Thus, from \eqref{before smallness of a est}, we arrive at 
			\begin{align}\label{V px estimate}
				\fint_{B_{R}} &\left\lvert V_{\mathfrak{p}(\cdot)}\left( \nabla_{\mathbb{E}} u\right) - V_{\mathfrak{p}(\cdot)}\left( \nabla_{\mathbb{E}} w\right)\right\rvert^{2}  \notag \\&\qquad \qquad \qquad \leq c \mathfrak{F}_{R, \theta}\left(\fint_{B_{R}}\left\lvert \nabla_{\mathbb{E}}u -\nabla_{\mathbb{E}} w\right\rvert^{\gamma} \right)^{\frac{1}{\gamma}}  + c\mathfrak{A}_{R, \bar{\sigma}}^{2}\lambda_{1}^{\bar{\mathfrak{p}}}. 
			\end{align}
			Now, fix $x\in \overline{B_R}$ and use the temporarily shorthand $\Phi:= \left\lvert \nabla_{\mathbb{E}}u\left(x\right) \right\rvert + \left\lvert \nabla_{\mathbb{E}} w\left(x\right) \right\rvert. $ Then using the triangle inequality and Young's inequality, we have 
			\begin{align}
	&\Phi^{\bar{\mathfrak{p}}-2}\left\lvert \nabla_{\mathbb{E}}u\left(x\right) - \nabla_{\mathbb{E}} w\left(x\right)\right\rvert^{2} \notag \\
	&= \left[  \Phi^{\mathfrak{p}(x)-2} + \left( \Phi^{\bar{\mathfrak{p}}-2} - \Phi^{\mathfrak{p}(x)-2}\right) \right]\left\lvert \nabla_{\mathbb{E}}u\left(x\right) - \nabla_{\mathbb{E}} w\left(x\right)\right\rvert^{2} \notag\\ 
	&= \left[  \Phi^{\mathfrak{p}(x)-2} + \Phi^{\mathfrak{p}(x)-2}\left( \Phi^{\bar{\mathfrak{p}}-\mathfrak{p}(x)} - 1\right) \right]\left\lvert \nabla_{\mathbb{E}}u\left(x\right) - \nabla_{\mathbb{E}} w\left(x\right)\right\rvert^{2} \notag\\
				& \leq \Phi^{\mathfrak{p}(x)-2}\left\lvert \nabla_{\mathbb{E}}u\left(x\right) - \nabla_{\mathbb{E}} w\left(x\right)\right\rvert^{2} + \Phi^{\mathfrak{p}(x)-2}\left\lvert \Phi^{\bar{\mathfrak{p}}-\mathfrak{p}(x)} - 1\right\rvert \left\lvert \nabla_{\mathbb{E}}u\left(x\right) - \nabla_{\mathbb{E}} w\left(x\right)\right\rvert^{2}
				\notag \\
				&\leq \Phi^{\mathfrak{p}(x)-2}\left\lvert \nabla_{\mathbb{E}}u\left(x\right) - \nabla_{\mathbb{E}} w\left(x\right)\right\rvert^{2} + \Phi^{\frac{\mathfrak{p}(x)}{2}}\left\lvert \Phi^{\bar{\mathfrak{p}}-\mathfrak{p}(x)} - 1\right\rvert  \Phi^{\frac{\mathfrak{p}(x)}{2}-1} \left\lvert \nabla_{\mathbb{E}}u\left(x\right) - \nabla_{\mathbb{E}} w\left(x\right)\right\rvert
				\notag \\
				&\leq \frac{3}{2}\Phi^{\mathfrak{p}(x)-2}\left\lvert \nabla_{\mathbb{E}}u\left(x\right) - \nabla_{\mathbb{E}} w\left(x\right)\right\rvert^{2} + \frac{1}{2}\Phi^{\mathfrak{p}(x)}\left\lvert \Phi^{\bar{\mathfrak{p}}-\mathfrak{p}(x)} - 1\right\rvert^{2}. \label{estimate of phi}
			\end{align}
			
			Now note that $0 < R < \bar{R}/4$. Set 
			\begin{align*}
				\mathfrak{p}_{1}:= \min_{\overline{B_{R}}}\mathfrak{p}\left(x\right) \quad \text{ and } \quad \mathfrak{p}_{2}:= \max\limits_{\overline{B_{R}}}\mathfrak{p}\left(x\right). 
			\end{align*}
			Note that by our choice of parameters, and the fact that $\bar{\sigma} \leq \mathfrak{p_1}-1$ we have 
			\begin{align}\label{comparison p2 by px}
				\bar{\mathfrak{p}}\left( 1 + \frac{\bar{\sigma}}{4}\right) \leq \mathfrak{p}_{2} \left( 1 + \frac{\bar{\sigma}}{4}\right) 
				&= \left( \mathfrak{p}_{1} + \omega_{\mathfrak{p}}\left(R\right) \right)\left( 1 + \frac{\bar{\sigma}}{4}\right)  \notag \\
				&= \mathfrak{p}_{1}\left( 1 + \frac{\bar{\sigma}}{4}\right) + \omega_{\mathfrak{p}}\left(R\right) \left( 1 + \frac{\bar{\sigma}}{4}\right)\notag  \\
				&\leq \mathfrak{p}_{1}\left( 1 + \frac{\bar{\sigma}}{4}\right) + \omega_{\mathfrak{p}}\left(R\right) \left( 1 + \bar{\sigma}\right) \notag \\		
				&\leq \mathfrak{p}_{1} \left(  1+ \frac{\bar{\sigma}}{4} + \omega_{\mathfrak{p}}\left(R\right) \right) \notag \\ 
				&\stackrel{\eqref{smallness of wp}}{\leq} \mathfrak{p}\left(x\right) \left(  1+ \frac{{\bar{\sigma}}}{4} + \omega_{\mathfrak{p}}\left(R\right) \right) \leq \mathfrak{p}\left(x\right) \left(  1+ \bar{\sigma} \right)
			\end{align}
			for all $x \in \overline{B_{R}}.$ Thus, by Theorem \ref{higher integrability f} and \ref{global_higher_integrability} we have $\nabla_{\mathbb{E}} u, \nabla_{\mathbb{E}} w \in L^{\bar{\mathfrak{p}}}\left(B_{R}\right).$ 
            
            Using Lemma \ref{prop of V}, integrating \eqref{estimate of phi} we deduce
			\begin{align}\label{V p bar estimate}
				\fint_{B_{R}} \left\lvert V_{\bar{\mathfrak{p}}}\left( \nabla_{\mathbb{E}} u\right) - V_{\bar{\mathfrak{p}}}\left( \nabla_{\mathbb{E}} w\right)\right\rvert^{2} &\leq c\fint_{B_{R}} 	\Phi^{\bar{\mathfrak{p}}-2}\left\lvert \nabla_{\mathbb{E}}u\left(x\right) - \nabla_{\mathbb{E}} w\left(x\right)\right\rvert^{2} \notag \\ 
				&\leq c \fint_{B_{R}} \left\lvert V_{\mathfrak{p}(\cdot)}\left( \nabla_{\mathbb{E}} u\right) - V_{\mathfrak{p}(\cdot)}\left( \nabla_{\mathbb{E}} w\right)\right\rvert^{2} \notag \\
				 & \hspace{60pt}+ c \fint_{B_{R}}\Phi^{\mathfrak{p}(x)}\left\lvert \Phi^{\bar{\mathfrak{p}}-\mathfrak{p}(x)} - 1\right\rvert^{2}.  
			\end{align}
			Now we plan to estimate the last term on the right. By mean value theorem, we have 
			\begin{align*}
				\left\lvert \Phi^{\bar{\mathfrak{p}}-\mathfrak{p}(x)} - 1\right\rvert &\leq \left\lvert\bar{\mathfrak{p}}-\mathfrak{p}(x)\right\rvert \Phi^{\lambda_{x}\left(\bar{\mathfrak{p}}-\mathfrak{p}(x)\right)}\left\lvert \log \Phi\right\rvert,
			\end{align*}
			for some $\lambda_{x} \in (0,1).$ Now we estimate $\Phi^{\mathfrak{p}(x) + 2\lambda_{x}\left(\bar{\mathfrak{p}}-\mathfrak{p}(x)\right)}\left\lvert \log \Phi\right\rvert^{2}.$ If $\Phi \in (0,e),$ we have 
			\begin{align*}
				\Phi^{\mathfrak{p}(x) + 2\lambda_{x}\left(\bar{\mathfrak{p}}-\mathfrak{p}(x)\right)}\left\lvert \log \Phi\right\rvert^{2} \leq c(\gamma_{1}, \gamma_{2}),
			\end{align*}
			due to the elementary estimate 
			\begin{align*}
				\sup\limits_{t \in (0, e)} t^{\sigma}\left\lvert \log t \right\rvert \leq \frac{1}{ae} + e^b, \qquad \text{ for all } \sigma \in [a,b] 
			\end{align*} with $0 < a < b$ and the estimate $\left\lvert \lambda_{x}\left(\bar{\mathfrak{p}}-\mathfrak{p}(x)\right)\right\rvert \leq \omega_{\mathfrak{p}}\left(R\right) \leq \gamma_{1}/4.$	If $\Phi \geq e,$ we have 
			\begin{align*}
				\Phi^{\mathfrak{p}(x) + 2\lambda_{x}\left(\bar{\mathfrak{p}}-\mathfrak{p}(x)\right)}\left\lvert \log \Phi\right\rvert^{2} \leq \Phi^{\mathfrak{p}(x) + 2\omega_{\mathfrak{p}}\left(R\right)} \log^{2} \left( e + \Phi^{\mathfrak{p}(x) + 2\omega_{\mathfrak{p}}\left(R\right)} \right). 
			\end{align*}
			Thus, in either case, we have 
			\begin{align*}
				\Phi^{\mathfrak{p}(x)}\left\lvert \Phi^{\bar{\mathfrak{p}}-\mathfrak{p}(x)} - 1\right\rvert^{2} \leq c\left(\gamma_{1}, \gamma_{2}\right)\left\lvert\bar{\mathfrak{p}}-\mathfrak{p}(x)\right\rvert^{2}
				&\left( 1 + \Phi^{\mathfrak{p}(x) + 2\omega_{\mathfrak{p}}\left(R\right)}\right) \notag \\ 
				&\times\log^{2} \left( e + 1 + \Phi^{\mathfrak{p}(x) + 2\omega_{\mathfrak{p}}\left(R\right)}\right). 
			\end{align*}
			Using H\"{o}lder inequality, this implies for $\sigma \in (0, \bar{\sigma}],$ we have 
			\begin{align}\label{esti in term of J}
				\fint_{B_{R}}&\Phi^{\mathfrak{p}(x)}\left\lvert \Phi^{\bar{\mathfrak{p}}-\mathfrak{p}(x)} - 1\right\rvert^{2} \leq c\left(\gamma_{1}, \gamma_{2}\right) \left( \fint_{B_{R}} \left\lvert\bar{\mathfrak{p}}-\mathfrak{p}(x)\right\rvert^{\frac{2(1+\sigma)}{\sigma}} \right)^{\frac{\sigma}{1+\sigma}} J^{\frac{1}{1+\sigma}}, 
			\end{align}
			where 
			\begin{align}\label{def J}
				J:=  \fint_{B_{R}}\left( 1 + \Phi^{\mathfrak{p}(x) + 2\omega_{\mathfrak{p}}\left(R\right)}\right)^{1+\sigma} \log^{2(1+\sigma)} \left( e + 1 + \Phi^{\mathfrak{p}(x) + 2\omega_{\mathfrak{p}}\left(R\right)}\right)\ \mathrm{d}x. 
			\end{align}
			We use \eqref{smallness of wp} and choose $\sigma \in (0,\bar{\sigma}]$ small, depending only on $\bar{\sigma}$ such that  
			\begin{align}\label{smallness of sigma}
				\left( 1 + \frac{2\omega_{\mathfrak{p}}\left(R\right)}{\gamma} \right)\left( 1+ \sigma \right)
				&\leq \left( 1 + \frac{2\omega_{\mathfrak{p}}\left(R\right)}{\gamma} \right)\left( 1+ \sigma \right)^{2} \leq 	\left( 1 + \frac{\bar{\sigma}}{2} \right)\left( 1+ \sigma \right) \leq 1 + \bar{\sigma}.  
			\end{align}
			Thus we estimate 
			\begin{align*}
				\fint_{B_{R}}\left( 1 + \Phi^{\mathfrak{p}(x) + 2\omega_{\mathfrak{p}}\left(R\right)}\right)^{1+\sigma}\ \mathrm{d}x 
				&\leq c \fint_{B_{R}}\left( 1 + \Phi^{\mathfrak{p}(x)\left( 1 + \frac{2\omega_{\mathfrak{p}}\left(R\right)}{\gamma}\right)}\right)^{1+\sigma}\ \mathrm{d}x \\				&\leq c \fint_{B_{R}} \left( 1 + \Phi^{\mathfrak{p}(x)\left( 1 + \bar{\sigma}\right)}\right)\ \mathrm{d}x.
			\end{align*}
			Observe that, by Theorem \ref{higher integrability f} and the estimate \eqref{estimate of lambda1}, we have 
			\begin{align*}
				\fint_{B_{R}} \left\lvert \nabla_{\mathbb{E}} u \right\rvert^{\mathfrak{p}(x)\left( 1 + \bar{\sigma}\right)} \ \mathrm{d}x \leq cR^{-{\bar{\mathfrak{p}}\left(Q_{\mathbb{E}}+1\right)\left( 1 + \bar{\sigma}\right)}} .
			\end{align*}
			This, combined with \eqref{estimate for w in terms of u} gives us the estimate  
			\begin{align*}
				\fint_{B_{R}} \left( 1 + \Phi^{\mathfrak{p}(x)\left( 1 + \bar{\sigma}\right)}\right)\ \mathrm{d}x &\leq \fint_{B_{R}} \left( 1 + \left\lvert \nabla_{\mathbb{E}} u \right\rvert^{\mathfrak{p}(x)\left( 1 + \bar{\sigma}\right)} + \left\lvert \nabla_{\mathbb{E}} w \right\rvert^{\mathfrak{p}(x)\left( 1 + \bar{\sigma}\right)}\right)\ \mathrm{d}x \\
				&\leq c R^{-\bar{\mathfrak{p}}\left(Q_{\mathbb{E}}+1\right)\left( 1 + \bar{\sigma}\right)} . 
			\end{align*}
			Thus, we arrive at 
			\begin{align}\label{final esti for phi in terms of R}
				\fint_{B_{R}}\left( 1 + \Phi^{\mathfrak{p}(x) + 2\omega_{\mathfrak{p}}\left(R\right)}\right)^{1+\sigma}\ \mathrm{d}x \leq c R^{-\gamma_2\left(Q_{\mathbb{E}}+1\right)\left( 1 + \bar{\sigma}\right)} . 
			\end{align}
			Now, recalling \eqref{def J} and applying Lemma \ref{Llog beta L lemma} with the choices 
			\begin{gather*}
				f= \left( 1 + \Phi^{\mathfrak{p}(x) + 2\omega_{\mathfrak{p}}\left(R\right)}\right)^{1+\sigma}, 
				\zeta = 1 + \sigma,\  \beta = 2\left( 1+\sigma\right),\ \delta = 1/\left( 1 + \sigma\right)
			\end{gather*} and  $$\tau = \gamma_2\left(Q_{\mathbb{E}}+1\right)\left( 1 + \bar{\sigma}\right) -Q_{\mathbb{E}},$$ we estimate 
			\begin{align*}
				J \leq c &\left[ 1 + R^{\gamma_2\left(Q_{\mathbb{E}}+1\right)\left( 1 + \bar{\sigma}\right) -Q_{\mathbb{E}}}  \int_{B_{R}}\left( 1 + \Phi^{\mathfrak{p}(x) + 2\omega_{\mathfrak{p}}\left(R\right)}\right)^{1+\sigma}\ \mathrm{d}x \right]^{2\left( 1+ \sigma\right)} \\
				&\qquad\times \log^{2\left( 1+ \sigma\right)} \left( \frac{2}{R}\right)
				\left[ \fint_{B_{R}}\left( 1 + \Phi^{\mathfrak{p}(x) + 2\omega_{\mathfrak{p}}\left(R\right)}\right)^{\left(1+\sigma\right)^{2}}\ \mathrm{d}x\right]^{\frac{1}{1+\sigma}}. 
			\end{align*}
			In view of \eqref{final esti for phi in terms of R}, this implies 
			\begin{align*}
				J \leq c \log^{2\left( 1+ \sigma\right)} \left( \frac{1}{R}\right)
				\left[ \fint_{B_{R}}\left( 1 + \Phi^{\mathfrak{p}(x) + 2\omega_{\mathfrak{p}}\left(R\right)}\right)^{\left(1+\sigma\right)^{2}}\ \mathrm{d}x\right]^{\frac{1}{1+\sigma}}. 
			\end{align*}
			Thus, plugging this estimate in \eqref{esti in term of J}, we arrive at 
			\begin{align}\label{estimate of the Phi term}
				\fint_{B_{R}}\Phi^{\mathfrak{p}(x)}\left\lvert \Phi^{\bar{\mathfrak{p}}-\mathfrak{p}(x)} - 1\right\rvert^{2} &\leq c \left( \fint_{B_{R}} \left\lvert\bar{\mathfrak{p}}-\mathfrak{p}(x)\right\rvert^{\frac{2(1+\sigma)}{\sigma}} \right)^{\frac{\sigma}{1+\sigma}} \notag\\ & \hspace{6pt}\times \log^{2} \left( \frac{1}{R}\right) \times \left[ \fint_{B_{R}}\left( 1 + \Phi^{\mathfrak{p}(x) + 2\omega_{\mathfrak{p}}\left(R\right)}\right)^{\left(1+\sigma\right)^{2}}\ \mathrm{d}x\right]^{\frac{1}{\left(1+\sigma\right)^{2}}}
			\end{align}
			
			Hence using \eqref{smallness of sigma} we can estimate 
			\begin{align}\label{estimate of 1+ sigma whole square term}
				&\fint_{B_{R}}\left( 1 + \Phi^{\mathfrak{p}(x) + 2\omega_{\mathfrak{p}}\left(R\right)}\right)^{\left(1+\sigma\right)^{2}}\ \mathrm{d}x \notag \\
				&\qquad \quad \leq c\fint_{B_{R}}\left( 1 + \Phi^{\mathfrak{p}(x)\left( 1 + \frac{2\omega_{\mathfrak{p}}\left(R\right)}{\gamma} \right)}\right)^{\left(1+\sigma\right)^{2}}\ \mathrm{d}x \notag \\
				&\qquad \quad \leq c\left( \fint_{B_{R}}\left( 1 + \Phi^{\mathfrak{p}(x)\left( 1 + \bar{\sigma}\right)}\right) \ \mathrm{d}x \right)^{\left( 1 + \frac{2\omega_{\mathfrak{p}}\left(R\right)}{\gamma} \right)\frac{\left(1+\sigma\right)^{2}}{\left(1+ \bar{\sigma}\right)}}.
			\end{align}
			Now using \eqref{higher integrability estimate f frozen and without gamma Global HI}  we have 
			\begin{align}\label{estimate of w in terms of u minimality and higher integrability}
				\fint_{B_{R}} \left\lvert \nabla_{\mathbb{E}} w \right\rvert^{\mathfrak{p}\left(x\right)\left( 1 + \bar{\sigma}\right)}\ \mathrm{d}x
				&\leq c \left( 1 +  \fint_{B_{2R}}\left\lvert \nabla_{\mathbb{E}} u \right\rvert^{\mathfrak{p}\left(x\right)\left( 1 + \bar{\sigma}\right)}\ \mathrm{d}x\right). 
			\end{align}
			Thus, using the definition of $\lambda_{1},$ we estimate,
			\begin{align}
				\fint_{B_{R}}\left( 1 + \Phi^{\mathfrak{p}(x)\left( 1 + \bar{\sigma}\right)}\right) \ \mathrm{d}x &\leq c \fint_{B_{R}}\left( 1 + \left\lvert \nabla_{\mathbb{E}} u \right\rvert^{\mathfrak{p}(x)\left( 1 + \bar{\sigma}\right)} + \left\lvert \nabla_{\mathbb{E}} w \right\rvert^{\mathfrak{p}(x)\left( 1 + \bar{\sigma}\right)}\right) \ \mathrm{d}x \notag \\
				&\stackrel{\eqref{estimate of w in terms of u minimality and higher integrability}}{\leq} c \fint_{B_{2R}}\left( 1 + \left\lvert \nabla_{\mathbb{E}} u \right\rvert^{\mathfrak{p}(x)\left( 1 + \bar{\sigma}\right)} \right) \ \mathrm{d}x \stackrel{\eqref{higher integrability estimate f}}{\leq} c\lambda_{1}^{\bar{\mathfrak{p}}(1+\bar{\sigma})}. 
			\end{align}
			Combining this with \eqref{estimate of the Phi term} and \eqref{estimate of 1+ sigma whole square term} we deduce 
			\begin{align}
				\fint_{B_{R}}\Phi^{\mathfrak{p}(x)}\left\lvert \Phi^{\bar{\mathfrak{p}}-\mathfrak{p}(x)} - 1\right\rvert^{2} &\leq c \log^{2} \left( \frac{1}{R}\right) \left( \fint_{B_{R}} \left\lvert \bar{\mathfrak{p}}-\mathfrak{p}(x)\right\rvert^{\frac{2(1+\sigma)}{\sigma}} \right)^{\frac{\sigma}{1+\sigma}}  \times \lambda_{1}^{\bar{\mathfrak{p}}\left( 1 + \frac{2\omega_{\mathfrak{p}}\left(R\right)}{\gamma} \right)} \notag \\
				&\leq c \mathfrak{P}_{R, \bar{\sigma}}^{2}	\lambda_{1}^{\bar{\mathfrak{p}}\left( 1 + \frac{2\omega_{\mathfrak{p}}\left(R\right)}{\gamma} \right)}  
				\leq c \mathfrak{P}_{R, \bar{\sigma}}^{2}	\lambda_{1}^{\bar{\mathfrak{p}}}, 
			\end{align}
			where we have used \eqref{estimate of lambda1}, the bound $\mathcal{K}_1<1/R$ and  \eqref{log holder bound Global HI} to deduce 
			$$ \lambda_{1}^{\frac{2\omega_{\mathfrak{p}}\left(R\right)}{\gamma}} \leq c \left( Q_{\mathbb{E}}, \gamma_{1}, \gamma_{2},\nu, L, L_{1} \right). $$ Now, combining this with \eqref{V p bar estimate} and \eqref{V px estimate}, we arrive at \eqref{comparison lemma 1 main estimate}. \smallskip 
			
			Now if $\bar{\mathfrak{p}}>2,$ we deduce 
			\begin{align*}
				&\fint_{B_{R}} \left\lvert  \nabla_{\mathbb{E}} u - \nabla_{\mathbb{E}} w\right\rvert^{\bar{\mathfrak{p}}} \\
				& \qquad \quad \leq \fint_{B_{R}} \left\lvert V_{\bar{\mathfrak{p}}}\left( \nabla_{\mathbb{E}} u\right) - V_{\bar{\mathfrak{p}}}\left( \nabla_{\mathbb{E}} w\right)\right\rvert^{2} \\
				&\qquad \quad \leq C \left( R^{\theta}\fint_{B_{R}} \left\lvert f \right\rvert^{\theta} \right)^{\frac{1}{\theta}}\left(\fint_{B_{R}}\left\lvert \nabla_{\mathbb{E}}u -\nabla_{\mathbb{E}} w\right\rvert^{\gamma} \right)^{\frac{1}{\gamma}}  + C\left( \mathfrak{A}_{R, \bar{\sigma}}^{2} + \mathfrak{P}_{R, \bar{\sigma}}^{2} \right) \lambda_{1}^{\bar{\mathfrak{p}}} \\
				&\qquad \quad \leq C \left( R^{\theta}\fint_{B_{R}} \left\lvert f \right\rvert^{\theta} \right)^{\frac{1}{\theta}}\left(\fint_{B_{R}}\left\lvert \nabla_{\mathbb{E}}u -\nabla_{\mathbb{E}} w\right\rvert^{\bar{\mathfrak{p}}} \right)^{\frac{1}{\bar{\mathfrak{p}}}}  + C\left( \mathfrak{A}_{R, \bar{\sigma}} + \mathfrak{P}_{R, \bar{\sigma}} \right)^{2} \lambda_{1}^{\bar{\mathfrak{p}}}. 
			\end{align*}
			Thus \eqref{comparison lemma 1 bar p bigger than 2} follows from Young's inequality with $\varepsilon>0$ small enough with exponents $\left( \bar{\mathfrak{p}}, \bar{\mathfrak{p}}/\left( \bar{\mathfrak{p}}-1\right) \right).$ On the other hand, if $\bar{\mathfrak{p}}\leq 2,$ we have 
			\begin{align}\label{penultimate est}
				&\fint_{B_{R}} \left\lvert  \nabla_{\mathbb{E}} u - \nabla_{\mathbb{E}} w\right\rvert^{\bar{\mathfrak{p}}} \notag\\
				&\qquad \quad \leq c\fint_{B_{R}}\left( \left\lvert  \nabla_{\mathbb{E}} u\right\rvert^{2} + \left\lvert  \nabla_{\mathbb{E}} w\right\rvert^{2} \right)^{\frac{\bar{\mathfrak{p}}\left(\bar{\mathfrak{p}}-2\right)}{2}} \left\lvert  \nabla_{\mathbb{E}} u - \nabla_{\mathbb{E}} w\right\rvert^{\bar{\mathfrak{p}}}\left( \left\lvert  \nabla_{\mathbb{E}} u\right\rvert^{2} + \left\lvert  \nabla_{\mathbb{E}} w\right\rvert^{2}\right)^{\frac{\bar{\mathfrak{p}}\left(2- \bar{\mathfrak{p}}\right)}{2}} \notag\\
				&\qquad \quad \leq c\fint_{B_{R}}\left\lvert V_{\bar{\mathfrak{p}}}\left( \nabla_{\mathbb{E}} u\right) - V_{\bar{\mathfrak{p}}}\left( \nabla_{\mathbb{E}} w\right)\right\rvert^{\bar{\mathfrak{p}}}\left( \left\lvert  \nabla_{\mathbb{E}} u\right\rvert^{2} + \left\lvert  \nabla_{\mathbb{E}} w\right\rvert^{2}\right)^{\frac{\bar{\mathfrak{p}}\left(2- \bar{\mathfrak{p}}\right)}{2}} \notag\\
				&\qquad \quad \leq c\left( \fint_{B_{R}}\left\lvert V_{\bar{\mathfrak{p}}}\left( \nabla_{\mathbb{E}} u\right) - V_{\bar{\mathfrak{p}}}\left( \nabla_{\mathbb{E}} w\right)\right\rvert^{2}\right)^{\frac{\bar{\mathfrak{p}}}{2}}\left( \fint_{B_{R}}\left( \left\lvert  \nabla_{\mathbb{E}} u\right\rvert^{2} + \left\lvert  \nabla_{\mathbb{E}} w\right\rvert^{2}\right)^{\frac{\bar{\mathfrak{p}}}{2}}\right)^{\frac{2-\bar{\mathfrak{p}}}{2}} \notag\\
				&\qquad \quad \leq c\left( \fint_{B_{R}}\left\lvert V_{\bar{\mathfrak{p}}}\left( \nabla_{\mathbb{E}} u\right) - V_{\bar{\mathfrak{p}}}\left( \nabla_{\mathbb{E}} w\right)\right\rvert^{2}\right)^{\frac{\bar{\mathfrak{p}}}{2}}\left( \fint_{B_{R}}\left( \left\lvert  \nabla_{\mathbb{E}} u\right\rvert^{\bar{\mathfrak{p}}} + \left\lvert  \nabla_{\mathbb{E}} w\right\rvert^{\bar{\mathfrak{p}}}\right)\right)^{\frac{2-\bar{\mathfrak{p}}}{2}}.
			\end{align}
			Now observe that by using Theorem \ref{global_higher_integrability} we estimate 
			\begin{align}\label{estimate of p bar energy of w}
				\fint_{B_{R}} \left\lvert \nabla_{\mathbb{E}} w  \right\rvert^{\bar{\mathfrak{p}}}
				 &\leq \left(\fint_{B_{R}} \left\lvert \nabla_{\mathbb{E}} w  \right\rvert^{\bar{\mathfrak{p}}\left(1+ \bar{\sigma}\right)}\right)^\frac{1}{1+\bar{\sigma}} \notag\\ 			& \leq c \left(1+ \fint_{B_{2R}}\left\lvert \nabla_{\mathbb{E}} u  \right\rvert^{\mathfrak{p}\left(x\right)(1+\bar{\sigma})}\ \mathrm{d}x \right)^{\frac{1}{1+\bar{\sigma}}} \stackrel{\eqref{higher integrability estimate f}}{\leq} c \lambda_1^{\bar{\mathfrak{p}}}. 
			\end{align}
			
			Similarly, using Theorem \ref{higher integrability f} we can estimate $\fint_{B_{R}} \left\lvert \nabla_{\mathbb{E}} u  \right\rvert^{\bar{\mathfrak{p}}}$ to conclude 
			\begin{align*}
				\fint_{B_{R}}\left( \left\lvert  \nabla_{\mathbb{E}} u\right\rvert^{\bar{\mathfrak{p}}} + \left\lvert  \nabla_{\mathbb{E}} w\right\rvert^{\bar{\mathfrak{p}}}\right) \leq c \lambda_{1}^{\bar{\mathfrak{p}}}.
			\end{align*}
			Plugging this in \eqref{penultimate est}, we arrive at 
			\begin{align*}
				&\fint_{B_{R}} \left\lvert  \nabla_{\mathbb{E}} u - \nabla_{\mathbb{E}} w\right\rvert^{\bar{\mathfrak{p}}} \\
				&\qquad\leq c\lambda_{1}^{\frac{\bar{\mathfrak{p}} \left( 2-\bar{\mathfrak{p}}\right)}{2}}\left( \fint_{B_{R}}\left\lvert V_{\bar{\mathfrak{p}}}\left( \nabla_{\mathbb{E}} u\right) - V_{\bar{\mathfrak{p}}}\left( \nabla_{\mathbb{E}} w\right)\right\rvert^{2}\right)^{\frac{\bar{\mathfrak{p}}}{2}} \\		
			&\qquad\stackrel{\eqref{comparison lemma 1 main estimate}}{\leq} c\left( \mathfrak{A}_{R, \bar{\sigma}} + \mathfrak{P}_{R, \bar{\sigma}} \right)^{\bar{\mathfrak{p}}} \lambda_{1}^{\bar{\mathfrak{p}}} + c \lambda_{1}^{\frac{\bar{\mathfrak{p}} \left( 2-\bar{\mathfrak{p}}\right)}{2}}\left( R^{\theta}\fint_{B_{R}} \left\lvert f \right\rvert^{\theta} \right)^{\frac{\bar{\mathfrak{p}}}{2\theta}}\left(\fint_{B_{R}}\left\lvert \nabla_{\mathbb{E}}u -\nabla_{\mathbb{E}} w\right\rvert^{\gamma} \right)^{\frac{\bar{\mathfrak{p}}}{2\gamma}} \\
				&\qquad\leq c\left( \mathfrak{A}_{R, \bar{\sigma}} + \mathfrak{P}_{R, \bar{\sigma}} \right)^{\bar{\mathfrak{p}}} \lambda_{1}^{\bar{\mathfrak{p}}} + c \lambda_{1}^{\frac{\bar{\mathfrak{p}} \left( 2-\bar{\mathfrak{p}}\right)}{2}}\left( R^{\theta}\fint_{B_{R}} \left\lvert f \right\rvert^{\theta} \right)^{\frac{\bar{\mathfrak{p}}}{2\theta}}\left(\fint_{B_{R}}\left\lvert \nabla_{\mathbb{E}}u -\nabla_{\mathbb{E}} w\right\rvert^{\bar{\mathfrak{p}}} \right)^{\frac{1}{2}}.
			\end{align*}
			Now \eqref{comparison lemma 1 bar p smaller than 2} follows by a simple application of Young's inequality.
		\end{proof}
		
		\begin{corollary}\label{homogeneous comparison at p0}
			Consider the same setup as in Lemma \ref{comparison lemma 1}. If $2<\bar{\mathfrak{p}}\leq \mathfrak{p}_0=\mathfrak{p}\left(x_0\right),$ then we have 
			\begin{align}\label{homogeneous comparison bar p bigger than 2}
				\fint_{B_{R}} \left\lvert  \nabla_{\mathbb{E}} u - \nabla_{\mathbb{E}} w\right\rvert^{\bar{\mathfrak{p}}} \leq 	C \left( R^{\theta}\fint_{B_{R}} \left\lvert f \right\rvert^{\theta} \right)^{\frac{\bar{\mathfrak{p}}}{\theta\left( \mathfrak{p}_0-1\right)}}  + C\left( \mathfrak{A}_{R, \bar{\sigma}} + \mathfrak{P}_{R, \bar{\sigma}} \right)^{2} \lambda_{1}^{\bar{\mathfrak{p}}}.
			\end{align}
			for some constant $ C \equiv C \left(Q_{\mathbb{E}}, \gamma_{1}, \gamma_{2}, \nu, L, \Lambda_{\log} \right)>0$.
		\end{corollary}
		\begin{proof}
			Note that, $q<Q_{\mathbb{E}}$. So, by H\"older's inequality and \eqref{energy bound f} we have $\mathfrak{F}_{R,\theta}^{Q_{\mathbb{E}}} <\mathcal{K}_1$. But by our choice of $R_0$ in Theorem \ref{higher integrability f} we have $\mathcal{K}_1< 1/R$, for any $R< R_0$.   Thus, using $2< \bar{\mathfrak{p}}\leq \mathfrak{p}_0$ and $\mathfrak{p}_0-\bar{\mathfrak{p}} \leq \omega_\mathfrak{p}(R)$ we estimate 
			\begin{align}\label{estimate of mathfrac}
				\mathfrak{F}_{R,\theta}^{\frac{1}{\bar{\mathfrak{p}}-1}} = \mathfrak{F}_{R,\theta}^{\frac{1}{\mathfrak{p}_0-1}}\mathfrak{F}_{R,\theta}^{\frac{\mathfrak{p}_0-\bar{\mathfrak{p}}}{\left(\bar{\mathfrak{p}}-1\right) \left(\mathfrak{p}_0-1\right)}}
				&\stackrel{\eqref{energy bound f}}{\leq} \mathfrak{F}_{R,\theta}^{\frac{1}{\mathfrak{p}_0-1}} \mathcal{K}_1^{{\frac{\mathfrak{p}_0-\bar{\mathfrak{p}}}{Q_{\mathbb{E}}\left(\bar{\mathfrak{p}}-1\right) \left(\mathfrak{p}_0-1\right)}}} \notag\\
				&\leq \mathfrak{F}_{R,\theta}^{\frac{1}{\mathfrak{p}_0-1}} \mathcal{K}_1^{\frac{\omega_\mathfrak{p}(R)}{Q_{\mathbb{E}}\left(\gamma-1\right)^2}}  
				 \leq \mathfrak{F}_{R,\theta}^{\frac{1}{\mathfrak{p}_0-1}} \left(\frac{1}{R}\right)^{\frac{\omega_\mathfrak{p}(R)}{Q_{\mathbb{E}}\left(\gamma-1\right)^2}} \stackrel{\eqref{log holder bound Global HI}}{\leq} c \mathfrak{F}_{R,\theta}^{\frac{1}{\mathfrak{p}_0-1}}. 
			\end{align}
			Therefore, combining \eqref{estimate of mathfrac} and \eqref{comparison lemma 1 bar p bigger than 2} we conclude \eqref{homogeneous comparison bar p bigger than 2}.
		\end{proof}

		\begin{lemma}[Comparison Lemma II]\label{comparison lemma 2}
			Let $w$ and $v$ solve \eqref{equation frozen} and \eqref{equation double frozen}, respectively. Consider the same setup as in Lemma \ref{comparison lemma 1}. Also, let $\mathfrak{p}_{R/2}$ be as defined in Section \ref{basic setup for comparison} and  $\lambda_{1}$ be as defined in Lemma \ref{comparison lemma 1}.  Then we have the following estimates.
			\begin{enumerate}[(i)]
				\item \begin{align}\label{comparison lemma 2 main estimate}
					\fint_{B_{R/2}} \left\lvert V_{\mathfrak{p}_{R/2}}\left( \nabla_{\mathbb{E}} w\right) - V_{\mathfrak{p}_{R/2}}\left( \nabla_{\mathbb{E}} v\right)\right\rvert^{2} \leq 	C_{7} \mathfrak{P}_{R, \bar{\sigma}}^{2} \lambda_{1}^{\mathfrak{p}_{R/2}}.
				\end{align}
				\item If $\mathfrak{p}_{R/2} >2,$ we have 
				\begin{align}\label{comparison lemma 2 bar p bigger than 2}
					\fint_{B_{R/2}} \left\lvert  \nabla_{\mathbb{E}} w - \nabla_{\mathbb{E}} v\right\rvert^{\mathfrak{p}_{R/2}} \leq  C_{7} \mathfrak{P}_{R, \bar{\sigma}} ^{2} \lambda_{1}^{\mathfrak{p}_{R/2}}.
				\end{align}
				\item If $\mathfrak{p}_{R/2} \leq 2,$ we have 
				\begin{align}\label{comparison lemma 2 bar p smaller than 2}
					\fint_{B_{R/2}} \left\lvert  \nabla_{\mathbb{E}} w - \nabla_{\mathbb{E}} v\right\rvert^{\mathfrak{p}_{R/2}} \leq C_{7} \mathfrak{P}_{R, \bar{\sigma}} ^{\mathfrak{p}_{R/2}} \lambda_{1}^{\mathfrak{p}_{R/2}}.
				\end{align}
			\end{enumerate}
			Here $ C_{7} = C_{7} \left(Q_{\mathbb{E}}, \gamma_{1}, \gamma_{2}, \nu, L, \Lambda_{\log} \right) \ge 1$ is a positive constant.
		\end{lemma}
		This is almost exactly  \cite[Lemma 32]{Mallick_Sil_continuityvariablegrowth} and thus we skip the proof.
		\subsection{Linearized comparison estimates}
		Let $\lambda \geq 1$. Consider,  the radius $R_{4}$ which is defined by \eqref{smallness of wp} and \eqref{smallness of mathfrak a+mathfrak p}.
		Let $\tau \in (0,1/4)$, $s>0$ and for $j=0,1, 2, \dots $, define the following quantities 
		\begin{align}
			r_j: &= \tau^j \frac{R_{4}}{2}, \quad B_j:= B_{r_j}\left(x_0\right), \quad sB_j:= B_{sr_j} (x_0) \label{seq of balls} \\
			\mathfrak{p}_0: &= \mathfrak{p}(x_0), \quad \mathfrak{p}_{r_j/2}: = \fint_{B_{r_j/2}} \mathfrak{p} , \quad \mathfrak{p}_{-}: =\inf_{B_0} \mathfrak{p}(\cdot), \quad \mathfrak{p}_{+} := \sup_{B_0} \mathfrak{p}(\cdot). \label{exponents}
		\end{align}
		Also, let $w_j \in u+\mathbb{E}W_0^{1,\mathfrak{p}(\cdot)}(B_j)$ and $v_j\in w_j+ \mathbb{E}W_0^{1, \mathfrak{p}_{r_j/2}}(\frac{1}{2}B_j)$ solve \eqref{equation frozen} and \eqref{equation double frozen} respectively with $R$ being replaced by $r_j$. We have the following linearized comparison.
		\begin{lemma}\label{linearized comparison}
			Let $j\ge 2$ and  $\theta$ and $\bar{\sigma}$ be as in 
			\eqref{theta def} and \eqref{sigma bar and R bar}, respectively. Suppose  we have  
			\begin{equation}\label{upper bound of f}
				\left( \fint_{4B_{j-1}} \left \lvert \nabla_{\mathbb{E}} u \right \rvert ^\gamma \right)^{\frac{1}{\gamma}} \leq \lambda \quad \text{and} \quad 4r_{j-1} \left( \fint_{4B_{j-1}}\lvert f \rvert^{q} \right)^{\frac{1}{q}} \leq \lambda^{\mathfrak{p}_0-1}
			\end{equation}
			and for some constant $A \geq 1,$ the estimates 
			\begin{equation}\label{both side bounds on w}
				\sup\limits_{\frac{1}{2}B_{j}} \left\lvert \nabla_{\mathbb{E}} w_{j} \right\rvert \leq A\lambda \qquad \text{ and } \qquad \frac{\lambda}{4A} \leq \left\lvert \nabla_{\mathbb{E}} w_{j-1} \right\rvert 
				\leq A\lambda \quad \text{ in } B_{j}
			\end{equation}
			hold. Then there exists a constant $C_8 \equiv C_{8}\left( Q_{\mathbb{E}}, \gamma_1, \gamma_2,\nu, L, \Lambda_{\log}, A, \tau \right) \ge 1$ such that 
			\begin{align}\label{grad u -grad v pbigger2}
				\left( \fint_{\frac{1}{2}B_{j}} \left\lvert \nabla_{\mathbb{E}}  u - \nabla_{\mathbb{E}}  v_{j} \right\rvert^{\gamma}\right)^{\frac{1}{\gamma}} \leq C_{8}\left( \mathfrak{A}_{r_{j-1}, \bar{\sigma}} + \mathfrak{P}_{r_{j-1}, \bar{\sigma}} \right)\lambda 
				+ C_{8}\lambda^{2-\mathfrak{p}_0}\mathfrak{F}_{r_{j-1}, \theta}.
			\end{align}
		\end{lemma}
		\begin{proof}
			We divide the proof in several cases depending on $\mathfrak{p}_0$ and $\mathfrak{p}_{r_j/2}$. Note that, by our choice $\gamma \leq \min\lbrace \mathfrak{p}_0,\mathfrak{p}_0', \mathfrak{p}_{r_j/2}, \mathfrak{p}'_{r_j/2} \rbrace$. Also, by the definitions of $\mathfrak{A}_{r_j, \bar{\sigma}}$, $\mathfrak{P}_{r_j, \bar{\sigma}}$ and $\mathfrak{F}_{r_j, \theta}$ in \eqref{mathfrak a}, \eqref{mathfrak p} and \eqref{mathfrak f} respectively,we have  
			\begin{align*}
				\mathfrak{A}_{r_j, \bar{\sigma}} \leq \tau^\frac{-Q_{\mathbb{E}}\bar{\sigma}}{2(1+\bar{\sigma})}\mathfrak{A}_{r_{j-1}, \bar{\sigma}},\quad \mathfrak{P}_{r_j, \bar{\sigma}} \leq   2\tau^\frac{-Q_{\mathbb{E}}\bar{\sigma}}{2(1+\bar{\sigma})}\mathfrak{P}_{r_{j-1}, \bar{\sigma}}, \   \mbox{  and  }\  \mathfrak{F}_{r_j, \theta} \leq \tau^\frac{\theta-Q_{\mathbb{E}}}{\theta}\mathfrak{F}_{r_{j-1},\theta}. 
			\end{align*}
			Now using the facts $0<\bar{\sigma}/2(1+\bar{\sigma})\le 1$, and  $0<\tau<1/4$, $1<\theta\leq Q_{\mathbb{E}}$, we deduce
            \begin{align}\label{estimates of j by j-1 for a, p, f}
				\mathfrak{A}_{r_j, \bar{\sigma}} \leq \tau^{-Q_{\mathbb{E}}}\mathfrak{A}_{r_{j-1}, \bar{\sigma}},\quad \mathfrak{P}_{r_j, \bar{\sigma}} \leq   2\tau^{-Q_{\mathbb{E}}}\mathfrak{P}_{r_{j-1}, \bar{\sigma}}, \   \mbox{  and  } \,\mathfrak{F}_{r_j, \theta} \leq \tau^{-Q_{\mathbb{E}}}\mathfrak{F}_{r_{j-1},\theta}. 
			\end{align}
			\begin{case}
				$\mathfrak{p}_0, \ \mathfrak{p}_{r_j/2} \leq 2.$
			\end{case}
			In this case, we use \eqref{comparison lemma 2 bar p smaller than 2}, \eqref{upper bound of f} and \eqref{comparison lemma 1 bar p smaller than 2} with $\bar{\mathfrak{p}}= \mathfrak{p}_0$ and $R= r_j$ along with H\"older inequality to arrive at \eqref{grad u -grad v pbigger2}.  
			
			\noindent Next, we prove a general estimate. 
			Let $\bar{x}\in \overline{\frac{1}{2}B_{j}}$ and $\bar{\mathfrak{p}}= \mathfrak{p}(\bar{x})>2$. Then we claim
			\begin{claim}\label{gen p bar}
            There exists $C \equiv C\left( Q_{\mathbb{E}}, \gamma_1, \gamma_2,\nu, L, L_1, A, B, \tau \right) >0$, such that
				\begin{align}\label{grad u -grad w pbigger2}
					\left( \fint_{\frac{1}{2}B_{j}} \left\lvert \nabla_{\mathbb{E}}  u - \nabla_{\mathbb{E}}  w_{j} \right\rvert^{\bar{\mathfrak{p}}'}\right)^{\frac{1}{\bar{\mathfrak{p}}'}} \leq C\left( \mathfrak{A}_{r_{j-1}, \bar{\sigma}} + \mathfrak{P}_{r_{j-1}, \bar{\sigma}} \right)\lambda 
					+ C\lambda^{2-\bar{\mathfrak{p}}}\mathfrak{F}_{r_{j-1}, \theta}.
				\end{align}
			\end{claim}
			\noindent \emph{Proof of Claim \ref{gen p bar}:} Applying \eqref{comparison lemma 1 bar p bigger than 2} on $B_{j-1}$ and $B_{j},$  and using \eqref{estimates of j by j-1 for a, p, f}, \eqref{upper bound of f} and \eqref{def lambda1} we deduce 
			\begin{align}\label{xwj-xwj-1 gen}
				\fint_{B_{j}} \lvert \nabla_{\mathbb{E}}  w_{j-1} - \nabla_{\mathbb{E}}  w_{j}\rvert^{\bar{\mathfrak{p}}} 
				&\leq c\tau^{-Q_{\mathbb{E}}}\fint_{B_{j-1}}\lvert \nabla_{\mathbb{E}} u  - \nabla_{\mathbb{E}} w_{j-1}\rvert^{\bar{\mathfrak{p}}} + c \fint_{B_{j}}\lvert \nabla_{\mathbb{E}} u  - \nabla_{\mathbb{E}} w_{j}\rvert^{\bar{\mathfrak{p}}} \notag \\
				&\leq c \tau^{-2Q_{\mathbb{E}} - \bar{\mathfrak{p}}'Q_{\mathbb{E}}} \left(\mathfrak{F}^{\bar{\mathfrak{p}}'}_{r_{j-1},\theta} + \left( \mathfrak{A}_{r_{j-1}, \bar{\sigma}} + \mathfrak{P}_{r_{j-1}, \bar{\sigma}} \right)^2\lambda^{\bar{\mathfrak{p}}}  \right).
			\end{align}
			
			Next, we estimate
			\begin{align*}
				&\left( \fint_{B_{j}} \left\lvert \nabla_{\mathbb{E}}  u - \nabla_{\mathbb{E}}  w_{j} \right\rvert^{\bar{\mathfrak{p}}'}\right)^{\frac{1}{\bar{\mathfrak{p}}'}}\\ 
				&\qquad\stackrel{\eqref{both side bounds on w}}{\leq} 
				\frac{c}{\lambda^{\bar{\mathfrak{p}}-2}}  \left( \fint_{B_{j}} \left\lvert \nabla_{\mathbb{E}} w_{j-1} \right\rvert^{\bar{\mathfrak{p}}'(\bar{\mathfrak{p}}-2)}\left\lvert \nabla_{\mathbb{E}} u - \nabla_{\mathbb{E}}  w_{j} \right\rvert^{\bar{\mathfrak{p}}'}\right)^{\frac{1}{\bar{\mathfrak{p}}'}} \\
				&\qquad\begin{aligned}
					\leq \frac{c}{\lambda^{\bar{\mathfrak{p}}-2}} &\left( \fint_{B_{j}} \left\lvert \nabla_{\mathbb{E}} w_{j} - \nabla_{\mathbb{E}} w_{j-1} \right\rvert^{\bar{\mathfrak{p}}'(\bar{\mathfrak{p}}-2)}\left\lvert \nabla_{\mathbb{E}} u - \nabla_{\mathbb{E}}  w_{j} \right\rvert^{\bar{\mathfrak{p}}'}\right)^{\frac{1}{\bar{\mathfrak{p}}'}}
					\\ &\qquad + \frac{c}{\lambda^{\bar{\mathfrak{p}}-2}}  \left( \fint_{B_{j}} \left\lvert \nabla_{\mathbb{E}} w_{j} \right\rvert^{\bar{\mathfrak{p}}'(\bar{\mathfrak{p}}-2)}\left\lvert \nabla_{\mathbb{E}} u - \nabla_{\mathbb{E}}  w_{j} \right\rvert^{\bar{\mathfrak{p}}'}\right)^{\frac{1}{\bar{\mathfrak{p}}'}}
					\eqqcolon I_{1} + I_{2}.
				\end{aligned}
			\end{align*}
			We handle the last two integrals separately. We use $\bar{\mathfrak{p}}' <2< \bar{\mathfrak{p}}$ to estimate
			\begin{align*}
				I_{1} &\leq \frac{c}{\lambda^{\bar{\mathfrak{p}}-2}}
				\left( \fint_{B_{j}} \left\lvert \nabla_{\mathbb{E}} w_{j} - \nabla_{\mathbb{E}} w_{j-1} \right\rvert^{\bar{\mathfrak{p}}}\right)^{\frac{\bar{\mathfrak{p}}-2}{\bar{\mathfrak{p}}}}
				\left( \fint_{B_{j}} \left\lvert \nabla_{\mathbb{E}}  u - \nabla_{\mathbb{E}}  w_{j} \right\rvert^{\bar{\mathfrak{p}}}\right)^{\frac{1}{\bar{\mathfrak{p}}}} \\
				&\stackrel{\eqref{xwj-xwj-1 gen},\eqref{comparison lemma 1 bar p bigger than 2}}{\leq} c\lambda^{2-\bar{\mathfrak{p}}} \left(\mathfrak{F}^{\bar{\mathfrak{p}}'}_{r_{j-1},\theta} + \left( \mathfrak{A}_{r_{j-1}, \bar{\sigma}} + \mathfrak{P}_{r_{j-1}, \bar{\sigma}} \right)^2\lambda^{\bar{\mathfrak{p}}}\right)^{\frac{1}{\bar{\mathfrak{p}}'}} \\
				&\stackrel{\eqref{smallness of mathfrak a+mathfrak p}}{\leq} C\left( \mathfrak{A}_{r_{j-1}, \bar{\sigma}} + \mathfrak{P}_{r_{j-1}, \bar{\sigma}} \right)\lambda + C\lambda^{2-\bar{\mathfrak{p}}}\mathfrak{F}_{r_{j-1}, \theta}.
			\end{align*}
			To estimate $I_2$, we note that 
			\begin{align*}
				&\left( \fint_{B_{j}} \left\lvert \nabla_{\mathbb{E}} w_{j} \right\rvert^{\bar{\mathfrak{p}}'(\bar{\mathfrak{p}}-2)}\left\lvert \nabla_{\mathbb{E}}  u - \nabla_{\mathbb{E}}  w_{j} \right\rvert^{\bar{\mathfrak{p}}'}\right)^{\frac{1}{\mathfrak{p}'}} \\
				&\qquad= \left( \fint_{B_{j}} \left\lvert \nabla_{\mathbb{E}} w_{j} \right\rvert^{\frac{\bar{\mathfrak{p}}'\left(\bar{\mathfrak{p}}-2\right)}{2}}\left\lvert \nabla_{\mathbb{E}}  u - \nabla_{\mathbb{E}}  w_{j} \right\rvert^{\bar{\mathfrak{p}}'} \left\lvert \nabla_{\mathbb{E}} w_{j} \right\rvert^{\frac{\bar{\mathfrak{p}}'\left(\bar{\mathfrak{p}}-2\right)}{2}}\right)^{\frac{1}{\bar{\mathfrak{p}}'}} \\
				&\qquad\stackrel{\text{H\"{o}lder}}{\leq} 
				\left( \fint_{B_{j}} \left\lvert \nabla_{\mathbb{E}} w_{j} \right\rvert^{\bar{\mathfrak{p}}-2}\left\lvert \nabla_{\mathbb{E}}  u - \nabla_{\mathbb{E}}  w_{j} \right\rvert^{2} \right)^{\frac{1}{2}}
				\left( \fint_{B_{j}} \left\lvert \nabla_{\mathbb{E}} w_{j} \right\rvert^{\bar{\mathfrak{p}}}\right)^{\frac{\bar{\mathfrak{p}}-2}{2\bar{\mathfrak{p}}}} .
			\end{align*}
			By \eqref{higher integrability estimate f frozen and without gamma Global HI} we also have 
			\begin{align*}
				\fint_{B_{j}} \left\lvert \nabla_{\mathbb{E}} w_{j} \right\rvert^{\bar{\mathfrak{p}}} 
				&\leq \left(\fint_{B_{j}} \left\lvert \nabla_{\mathbb{E}} w_{j} \right\rvert^{\bar{\mathfrak{p}}(1+\bar{\sigma})}\right)^{\frac{1}{1+\bar{\sigma}}}\\   
				&\leq c\left(1+	\left(\fint_{2B_{j}} \left\lvert \nabla_{\mathbb{E}} u \right\rvert^{\bar{\mathfrak{p}}(1+\bar{\sigma})}\right)^{\frac{1}{1+\bar{\sigma}}}\right)	\stackrel{\eqref{higher integrability estimate f},\eqref{upper bound of f}}{\leq} c\lambda^{\bar{\mathfrak{p}}}.				
			\end{align*}
			Thus combining the last two estimates and using $\gamma\leq \gamma_2' \leq \bar{\mathfrak{p}}'$ we deduce, 
			\begin{align*}
				I_{2}
				&\leq c\lambda^{\frac{2-\bar{\mathfrak{p}}}{2}} \left( \fint_{B_{j}} \left\lvert \nabla_{\mathbb{E}} w_{j} \right\rvert^{\bar{\mathfrak{p}}-2}\left\lvert \nabla_{\mathbb{E}}  u - \nabla_{\mathbb{E}}  w_{j} \right\rvert^{2} \right)^{\frac{1}{2}} \\
				&\leq c\lambda^{\frac{2-\bar{\mathfrak{p}}}{2}} 
				\left( \fint_{B_{j}} \left( \left\lvert \nabla_{\mathbb{E}} u \right\rvert + \left\lvert \nabla_{\mathbb{E}} w_{j} \right\rvert \right)^{\bar{\mathfrak{p}}-2}\left\lvert \nabla_{\mathbb{E}}  u - \nabla_{\mathbb{E}}  w_{j} \right\rvert^{2} \right)^{\frac{1}{2}} \\
				&\leq c\lambda^{\frac{2-\bar{\mathfrak{p}}}{2}} 
				\left(\fint_{B_j}\left\lvert V_{\bar{\mathfrak{p}}}\left( \nabla_{\mathbb{E}} u\right) - V_{\bar{\mathfrak{p}}}\left( \nabla_{\mathbb{E}} w_j\right)\right\rvert^{2}\right)^{\frac{1}{2}} 
				\end{align*}
Now using \eqref{comparison lemma 1 main estimate} and \eqref{estimates of j by j-1 for a, p, f} we have
\begin{align*}
			I_{2}	&\leq c\lambda^{\frac{2-\bar{\mathfrak{p}}}{2}}  \mathfrak{F}^{\frac{1}{2}}_{r_{j-1}, \theta}\left(\fint_{B_{j}}\left\lvert \nabla_{\mathbb{E}}u -\nabla_{\mathbb{E}} w_j\right\rvert^{\gamma} \right)^{\frac{1}{2\gamma}}  + c\left( \mathfrak{A}_{r_{j-1}, \bar{\sigma}} + \mathfrak{P}_{r_{j-1}, \bar{\sigma}} \right) \lambda \notag \\
				&\stackrel{\text{H\"{o}lder}}{\leq} c\lambda^{\frac{2-\bar{\mathfrak{p}}}{2}} \left( \fint_{B_{j}} \lvert \nabla_{\mathbb{E}} u - \nabla_{\mathbb{E}} w_{j}\rvert^{\bar{\mathfrak{p}}'}\right)^{\frac{1}{2\bar{\mathfrak{p}}'}} \mathfrak{F}^{\frac{1}{2}}_{r_{j-1}, \theta} + c\left( \mathfrak{A}_{r_{j-1}, \bar{\sigma}} + \mathfrak{P}_{r_{j-1}, \bar{\sigma}} \right) \lambda\\ 
				&= c \left[ \left( \fint_{B_{j}} \lvert \nabla_{\mathbb{E}} u - \nabla_{\mathbb{E}} w_{j}\rvert^{\bar{\mathfrak{p}}'}\right)^{\frac{1}{\bar{\mathfrak{p}}'}}\right]^{\frac{1}{2}}
				\left[ \lambda^{2-\bar{\mathfrak{p}}} \mathfrak{F}_{r_{j-1}, \theta} \right]^{\frac{1}{2}}+ c\left( \mathfrak{A}_{r_{j-1}, \bar{\sigma}} + \mathfrak{P}_{r_{j-1}, \bar{\sigma}} \right) \lambda.
			\end{align*}
			But this implies, by Young's inequality with $\varepsilon > 0, $
			\begin{align*}
				I_{2} 
				&\leq  \varepsilon \left( \fint_{B_{j}} \lvert \nabla_{\mathbb{E}} u - \nabla_{\mathbb{E}} w_{j}\rvert^{\bar{\mathfrak{p}}'}\right)^{\frac{1}{\bar{\mathfrak{p}}'}} + c\lambda^{2-\bar{\mathfrak{p}}}\mathfrak{F}_{r_{j-1}, \theta} + c\left( \mathfrak{A}_{r_{j-1}, \theta} + \mathfrak{P}_{r_{j-1}, \theta} \right) \lambda.
			\end{align*}
			Combining the estimates of $I_1$ and $I_2$ and choosing $\varepsilon > 0$ small enough, and  using 
			$$\left( \fint_{\frac{1}{2}B_{j}} \lvert \nabla_{\mathbb{E}} u - \nabla_{\mathbb{E}} w_{j}\rvert^{\bar{\mathfrak{p}}'}\right)^{\frac{1}{\bar{\mathfrak{p}}'}} \leq c \left( \fint_{B_{j}} \lvert \nabla_{\mathbb{E}} u - \nabla_{\mathbb{E}} w_{j}\rvert^{\bar{\mathfrak{p}}'}\right)^{\frac{1}{\bar{\mathfrak{p}}'}}$$
			we obtain \eqref{grad u -grad w pbigger2}, proving Claim \ref{gen p bar}. \bigskip
			
			\begin{case}
				$\mathfrak{p}_0\leq 2< \mathfrak{p}_{r_j/2}$  or $2< \mathfrak{p}_0 \leq \mathfrak{p}_{r_j/2}.$
			\end{case}
			In this case we linearize w.r.t.  $\mathfrak{p}_{r_j/2}$. First we claim the following. 
			\begin{claim}\label{linearize wrt pr/2 frozen-double frozen}
            There exists $C \equiv C\left(Q_{\mathbb{E}}, \gamma_1, \gamma_2,\nu, L, L_1, A, B, \tau \right) >0$, such that 
				\begin{align}\label{linearize estimate wrt pr/2 frozen-double frozen}
					\left( \fint_{\frac{1}{2}B_{j}} \left\lvert \nabla_{\mathbb{E}}  v_{j} - \nabla_{\mathbb{E}}  w_{j} \right\rvert^{\mathfrak{p}_{r_j/2}'}\right)^{\frac{1}{\mathfrak{p}_{r_j/2}'}} 
					\leq C\left(\mathfrak{A}_{r_{j-1}, \bar{\sigma}} + \mathfrak{P}_{r_{j-1}, \bar{\sigma}}\right)\lambda + C \lambda^{2-\mathfrak{p}_{r_j/2}}  \mathfrak{F}_{r_{j-1},\theta}. 
				\end{align}
				
			\end{claim}
			\emph{Proof of Claim \ref{linearize wrt pr/2 frozen-double frozen}}: We use the shorthand $\mathfrak{p}_{j}= \mathfrak{p}_{r_j/2}$ and estimate
			\begin{align*}
				&\left( \fint_{\frac{1}{2}B_{j}} \left\lvert \nabla_{\mathbb{E}}  v_{j} - \nabla_{\mathbb{E}}  w_{j} \right\rvert^{\mathfrak{p}_{j}'}\right)^{\frac{1}{\mathfrak{p}_{j}'}} \\
				&\qquad\stackrel{\eqref{both side bounds on w}}{\leq} 
				\frac{c}{\lambda^{\mathfrak{p}_{j}-2}}  \left( \fint_{\frac{1}{2}B_{j}} \left\lvert \nabla_{\mathbb{E}} w_{j-1} \right\rvert^{\mathfrak{p}_{j}'(\mathfrak{p}_{j}-2)}\left\lvert \nabla_{\mathbb{E}}  v_{j} - \nabla_{\mathbb{E}} w_{j} \right\rvert^{\mathfrak{p}_{j}'}\right)^{\frac{1}{\mathfrak{p}_{j}'}} \\
				&\qquad\begin{aligned}
					\leq \frac{c}{\lambda^{\mathfrak{p}_{j}-2}} &\left( \fint_{\frac{1}{2}B_{j}} \left\lvert \nabla_{\mathbb{E}} w_{j} - \nabla_{\mathbb{E}} w_{j-1} \right\rvert^{\mathfrak{p}_{j}'(\mathfrak{p}_{j}-2)}\left\lvert \nabla_{\mathbb{E}} v_{j} - \nabla_{\mathbb{E}}  w_{j} \right\rvert^{\mathfrak{p}_{j}'}\right)^{\frac{1}{\mathfrak{p}_{j}'}}
					\\ &\quad + \frac{c}{\lambda^{\mathfrak{p}_{j}-2}} \left( \fint_{\frac{1}{2}B_{j}} \left\lvert \nabla_{\mathbb{E}} w_{j} \right\rvert^{\mathfrak{p}_{j}'(\mathfrak{p}_{j}-2)}\left\lvert \nabla_{\mathbb{E}} v_{j} - \nabla_{\mathbb{E}} w_{j} \right\rvert^{\mathfrak{p}_{j}'}\right)^{\frac{1}{\mathfrak{p}_{j}'}} \eqqcolon I_3+I_4.
				\end{aligned}
			\end{align*}
			The first integral on the right i.e. $I_3$ can be estimated as 
			\begin{align*}
                I_3
				 &\leq \frac{c}{\lambda^{\mathfrak{p}_{j}-2}}  \left( \fint_{\frac{1}{2}B_{j}} \left\lvert \nabla_{\mathbb{E}} w_{j} - \nabla_{\mathbb{E}} w_{j-1} \right\rvert^{\mathfrak{p}_{j}}\right)^{\frac{\mathfrak{p}_{j}-2}{\mathfrak{p}_{j}}}
				\left( \fint_{\frac{1}{2}B_{j}} \left\lvert \nabla_{\mathbb{E}}  v_{j} - \nabla_{\mathbb{E}}  w_{j} \right\rvert^{\mathfrak{p}_{j}}\right)^{\frac{1}{\mathfrak{p}_{j}}} \\
				&\stackrel{\eqref{comparison lemma 2 bar p bigger than 2}, \eqref{upper bound of f},\eqref{estimates of j by j-1 for a, p, f}}{\leq}	\frac{c}{\lambda^{\mathfrak{p}_{j}-2}} \left( \fint_{\frac{1}{2}B_{j}} \left\lvert \nabla_{\mathbb{E}} w_{j} - \nabla_{\mathbb{E}} w_{j-1} \right\rvert^{\mathfrak{p}_{j}}\right)^{\frac{\mathfrak{p}_{j}-2}{\mathfrak{p}_{j}}}
				\left[\mathfrak{P}_{r_{j-1},\bar{\sigma}}\right]^{\frac{2}{\mathfrak{p}_{j}}}\lambda \\
				&\leq c\mathfrak{P}_{r_{j-1},\bar{\sigma}} \lambda + c \lambda^{1-\mathfrak{p}_{j}} \fint_{\frac{1}{2}B_{j}} \left\lvert \nabla_{\mathbb{E}} w_{j} - \nabla_{\mathbb{E}} w_{j-1} \right\rvert^{\mathfrak{p}_{j}}.
			\end{align*}
			Now using \eqref{xwj-xwj-1 gen} with $\bar{\mathfrak{p}}=\mathfrak{p}_j$, $\lambda\geq 1$ and $\mathfrak{p}_0-\mathfrak{p}_j\leq 0$, we deduce
			\begin{align*}
				I_3
				&\leq c\mathfrak{P}_{r_{j-1},\bar{\sigma}} \lambda + c \lambda^{1-\mathfrak{p}_{j}} \left(\mathfrak{F}^{\mathfrak{p}_j'}_{r_{j-1},\theta} + \left( \mathfrak{A}_{r_{j-1}, \bar{\sigma}} + \mathfrak{P}_{r_{j-1}, \bar{\sigma}} \right)^2\lambda^{\mathfrak{p}_j}\right)  \\ 
				&\stackrel{\eqref{smallness of mathfrak a+mathfrak p}}{\leq} c\left(\mathfrak{A}_{r_{j-1}, \bar{\sigma}} + \mathfrak{P}_{r_{j-1}, \bar{\sigma}}\right)\lambda + c \lambda^{1-\mathfrak{p}_j} \mathfrak{F}^{\frac{1}{\mathfrak{p}_j-1}}_{r_{j-1},\theta} \mathfrak{F}_{r_{j-1},\theta}\\
				&\stackrel{\eqref{upper bound of f}}{\leq} c\left(\mathfrak{A}_{r_{j-1}, \bar{\sigma}} + \mathfrak{P}_{r_{j-1}, \bar{\sigma}}\right)\lambda + c \lambda^{1-\mathfrak{p}_j} \lambda^{\frac{\mathfrak{p}_0-1}{\mathfrak{p}_j-1}} \mathfrak{F}_{r_{j-1},\theta}\\
				&\leq c\left(\mathfrak{A}_{r_{j-1}, \bar{\sigma}} + \mathfrak{P}_{r_{j-1}, \bar{\sigma}}\right)\lambda + c \lambda^{2-\mathfrak{p}_j}  \mathfrak{F}_{r_{j-1},\theta},
			\end{align*}
			We estimate $I_4$ as follows. 
			\begin{align*}
				I_4=&\frac{c}{\lambda^{\mathfrak{p}_{j}-2}} \left( \fint_{\frac{1}{2}B_{j}} \left\lvert \nabla_{\mathbb{E}} w_{j} \right\rvert^{\mathfrak{p}_{j}'(\mathfrak{p}_{j}-2)}\left\lvert \nabla_{\mathbb{E}}  v_{j} - \nabla_{\mathbb{E}}  w_{j} \right\rvert^{\mathfrak{p}_{j}'}\right)^{\frac{1}{\mathfrak{p}_{j}'}} \\
				&\stackrel{\eqref{both side bounds on w}}{\leq} \frac{c}{\lambda^{\mathfrak{p}_{j}-2}} \left( c\lambda \right)^{\frac{\mathfrak{p}_{j}-2}{2}} 
				\left( \fint_{\frac{1}{2}B_{j}} \left\lvert \nabla_{\mathbb{E}} w_{j} \right\rvert^{\frac{\mathfrak{p}_{j}'(\mathfrak{p}_{j}-2)}{2}}\left\lvert \nabla_{\mathbb{E}} v_{j} - \nabla_{\mathbb{E}}  w_{j} \right\rvert^{\mathfrak{p}_{j}'}\right)^{\frac{1}{\mathfrak{p}_{j}'}} \\
				&\leq c \lambda^{\frac{2-\mathfrak{p}_{j}}{2}} \left( \fint_{\frac{1}{2}B_{j}} \left\lvert \nabla_{\mathbb{E}} w_{j} \right\rvert^{\mathfrak{p}_{j}-2}\left\lvert \nabla_{\mathbb{E}} v_{j} - \nabla_{\mathbb{E}} w_{j} \right\rvert^{2}\right)^{\frac{1}{2}} \\
				&\leq c \lambda^{\frac{2-\mathfrak{p}_{j}}{2}} \left( \fint_{\frac{1}{2}B_{j}} \left( \left\lvert \nabla_{\mathbb{E}} v_{j} \right\rvert + \left\lvert \nabla_{\mathbb{E}} w_{j} \right\rvert \right)^{\mathfrak{p}_{j}-2}
				\left\lvert \nabla_{\mathbb{E}}  v_{j} - \nabla_{\mathbb{E}} w_{j} \right\rvert^{2}\right)^{\frac{1}{2}}. \end{align*}
			In view of \eqref{constant cv}, we can estimate the right hand side to obtain  
            \begin{align*}
                I_{4} \leq  c \lambda^{\frac{2-\mathfrak{p}_{j}}{2}} \left( \fint_{\frac{1}{2}B_{j}} \left\lvert V_{\mathfrak{p}_{j}}(\nabla_{\mathbb{E}} v_{j}) - V_{\mathfrak{p}_{j}}(\nabla_{\mathbb{E}} w_{j}) \right\rvert^{2}\right)^{\frac{1}{2}} 
				\stackrel{\eqref{comparison lemma 2 main estimate}, \eqref{estimates of j by j-1 for a, p, f}, \eqref{upper bound of f}}{\leq} c\mathfrak{P}_{r_{j-1}, \bar{\sigma}}\lambda . 
            \end{align*}
Combining these estimates for $I_3$ and $I_4$ we prove Claim \ref{linearize wrt pr/2 frozen-double frozen}. 
			
			\par To complete the proof in this case, we first apply H\"older's inequality, utilizing the fact that $\gamma \leq \mathfrak{p}_{r_j/2}$. Then combining \eqref{linearize estimate wrt pr/2 frozen-double frozen} and the general estimate \eqref{grad u -grad w pbigger2} with the choice of $\bar{\mathfrak{p}}= \mathfrak{p}_{r_j/2}$ we derive 
			$$\left( \fint_{\frac{1}{2}B_{j}} \left\lvert \nabla_{\mathbb{E}}  u - \nabla_{\mathbb{E}}  v_{j} \right\rvert^{\gamma}\right)^{\frac{1}{\gamma}} \leq C_{1}\left( \mathfrak{A}_{r_{j-1}, \bar{\sigma}} + \mathfrak{P}_{r_{j-1}, \bar{\sigma}} \right)\lambda 
			+ C_{1}\lambda^{2-\mathfrak{p}_{r_j/2}}\mathfrak{F}_{r_{j-1}, \theta}.$$
			
			Now, \eqref{grad u -grad v pbigger2} follows from the fact $\lambda\geq 1$ and $\mathfrak{p}_0-\mathfrak{p}_{r_j/2}\leq 0$. 
			\begin{case}
				$\mathfrak{p}_0>2$ and $\mathfrak{p}_{r_j/2}\leq 2$ 
			\end{case}
			Since $\gamma\leq \min\{\mathfrak{p}_0', \mathfrak{p}_{r_j/2}\}$, so using H\"older's inequaity and the estimate \eqref{grad u -grad w pbigger2} with the choice of $\bar{\mathfrak{p}}= \mathfrak{p}_{0}$  we have 
			$$\left( \fint_{\frac{1}{2}B_{j}} \left\lvert \nabla_{\mathbb{E}}  u - \nabla_{\mathbb{E}}  w_{j} \right\rvert^{\gamma}\right)^{\frac{1}{\gamma}} \leq C \left( \mathfrak{A}_{r_{j-1}, \bar{\sigma}} + \mathfrak{P}_{r_{j-1}, \bar{\sigma}} \right)\lambda 
			+ C \lambda^{2-\mathfrak{p}_{0}}\mathfrak{F}_{r_{j-1}, \theta}.$$
			On the other hand, again using H\"older's inequality and the estimates \eqref{comparison lemma 2 bar p smaller than 2} with $R= r_j$, \eqref{upper bound of f} and \eqref{estimates of j by j-1 for a, p, f}  we derive 
			$$\left( \fint_{\frac{1}{2}B_{j}} \left\lvert \nabla_{\mathbb{E}}  w_j - \nabla_{\mathbb{E}}  v_{j} \right\rvert^{\gamma}\right)^{\frac{1}{\gamma}} \leq C \mathfrak{P}_{r_{j-1}, \bar{\sigma}} \lambda.$$	
			Combining these two estimates we prove \eqref{grad u -grad v pbigger2} in this case. 	
			\begin{case}
				$\mathfrak{p}_0> \mathfrak{p}_{r_j/2}> 2$
			\end{case}
			Using \eqref{grad u -grad w pbigger2} with the choice of $\bar{\mathfrak{p}}= \mathfrak{p}_{0}$,  $\gamma\leq \mathfrak{p}_0'$ and H\"older's inequality, we have 
			\begin{align}\label{final case linearized frozen}
				\left( \fint_{\frac{1}{2}B_{j}} \left\lvert \nabla_{\mathbb{E}}  u - \nabla_{\mathbb{E}}  w_{j} \right\rvert^{\gamma}\right)^{\frac{1}{\gamma}} \leq C\left( \mathfrak{A}_{r_{j-1}, \bar{\sigma}} + \mathfrak{P}_{r_{j-1}, \bar{\sigma}} \right)\lambda 
				+ C\lambda^{2-\mathfrak{p}_{0}}\mathfrak{F}_{r_{j-1}, \theta}.
			\end{align}
			Denoting $p_j:= \mathfrak{p}_{r_j/2}$ for now, we proceed similarly as in the proof of Claim \ref{linearize wrt pr/2 frozen-double frozen} . 
			\begin{align*}
				&\left( \fint_{\frac{1}{2}B_{j}} \left\lvert \nabla_{\mathbb{E}}  v_{j} - \nabla_{\mathbb{E}}  w_{j} \right\rvert^{\mathfrak{p}_{j}'}\right)^{\frac{1}{\mathfrak{p}_{j}'}} \\
				&\qquad\stackrel{\eqref{both side bounds on w}}{\leq} 
				\frac{c}{\lambda^{\mathfrak{p}_{j}-2}}  \left( \fint_{\frac{1}{2}B_{j}} \left\lvert \nabla_{\mathbb{E}} w_{j-1} \right\rvert^{\mathfrak{p}_{j}'(\mathfrak{p}_{j}-2)}\left\lvert \nabla_{\mathbb{E}}  v_{j} - \nabla_{\mathbb{E}} w_{j} \right\rvert^{\mathfrak{p}_{j}'}\right)^{\frac{1}{\mathfrak{p}_{j}'}} \\
				&\qquad \begin{aligned}
					\leq \frac{c}{\lambda^{\mathfrak{p}_{j}-2}} &\left( \fint_{\frac{1}{2}B_{j}} \left\lvert \nabla_{\mathbb{E}} w_{j} - \nabla_{\mathbb{E}} w_{j-1} \right\rvert^{\mathfrak{p}_{j}'(\mathfrak{p}_{j}-2)}\left\lvert \nabla_{\mathbb{E}} v_{j} - \nabla_{\mathbb{E}}  w_{j} \right\rvert^{\mathfrak{p}_{j}'}\right)^{\frac{1}{\mathfrak{p}_{j}'}}
					\\ &\quad + \frac{c}{\lambda^{\mathfrak{p}_{j}-2}} \left( \fint_{\frac{1}{2}B_{j}} \left\lvert \nabla_{\mathbb{E}} w_{j} \right\rvert^{\mathfrak{p}_{j}'(\mathfrak{p}_{j}-2)}\left\lvert \nabla_{\mathbb{E}} v_{j} - \nabla_{\mathbb{E}} w_{j} \right\rvert^{\mathfrak{p}_{j}'}\right)^{\frac{1}{\mathfrak{p}_{j}'}} \eqqcolon I'_3+I'_4.
				\end{aligned}
			\end{align*}
			The estimate of $I_4'$ is same as the estimate of $I_4$ in Claim \ref{linearize wrt pr/2 frozen-double frozen} which is 
			\begin{align}\label{I4' estimate}
				I'_4=&\frac{c}{\lambda^{\mathfrak{p}_{j}-2}} \left( \fint_{\frac{1}{2}B_{j}} \left\lvert \nabla_{\mathbb{E}} w_{j} \right\rvert^{\mathfrak{p}_{j}'(\mathfrak{p}_{j}-2)}\left\lvert \nabla_{\mathbb{E}}  v_{j} - \nabla_{\mathbb{E}}  w_{j} \right\rvert^{\mathfrak{p}_{j}'}\right)^{\frac{1}{\mathfrak{p}_{j}'}} \leq c\mathfrak{P}_{r_{j-1}, \bar{\sigma}}\, \lambda.
			\end{align}
			We skip the details. Next we estimate $I'_3$. 
			\begin{align}\label{mid estimate of I3'}
				I_3':&= \frac{c}{\lambda^{\mathfrak{p}_{j}-2}}\left( \fint_{\frac{1}{2}B_{j}} \left\lvert \nabla_{\mathbb{E}} w_{j} - \nabla_{\mathbb{E}} w_{j-1} \right\rvert^{\mathfrak{p}_{j}'(\mathfrak{p}_{j}-2)}\left\lvert \nabla_{\mathbb{E}}  v_{j} - \nabla_{\mathbb{E}}  w_{j} \right\rvert^{\mathfrak{p}_{j}'}\right)^{\frac{1}{\mathfrak{p}_{j}'}} \notag \\
				&\leq \frac{c}{\lambda^{\mathfrak{p}_{j}-2}}  \left( \fint_{\frac{1}{2}B_{j}} \left\lvert \nabla_{\mathbb{E}} w_{j} - \nabla_{\mathbb{E}} w_{j-1} \right\rvert^{\mathfrak{p}_{j}}\right)^{\frac{\mathfrak{p}_{j}-2}{\mathfrak{p}_{j}}}
				\left( \fint_{\frac{1}{2}B_{j}} \left\lvert \nabla_{\mathbb{E}}  v_{j} - \nabla_{\mathbb{E}}  w_{j} \right\rvert^{\mathfrak{p}_{j}}\right)^{\frac{1}{\mathfrak{p}_{j}}} \notag \\
				&\stackrel{\eqref{comparison lemma 2 bar p bigger than 2}, \eqref{upper bound of f},\eqref{estimates of j by j-1 for a, p, f}}{\leq}	\frac{c}{\lambda^{\mathfrak{p}_{j}-2}} \left( \fint_{\frac{1}{2}B_{j}} \left\lvert \nabla_{\mathbb{E}} w_{j} - \nabla_{\mathbb{E}} w_{j-1} \right\rvert^{\mathfrak{p}_{j}}\right)^{\frac{\mathfrak{p}_{j}-2}{\mathfrak{p}_{j}}}
				\left[\mathfrak{P}_{r_{j-1},\bar{\sigma}}\right]^{\frac{2}{\mathfrak{p}_{j}}}\lambda \notag \\
				&\leq c\mathfrak{P}_{r_{j-1},\bar{\sigma}} \lambda + c \lambda^{1-\mathfrak{p}_{j}} \fint_{\frac{1}{2}B_{j}} \left\lvert \nabla_{\mathbb{E}} w_{j} - \nabla_{\mathbb{E}} w_{j-1} \right\rvert^{\mathfrak{p}_{j}}.
			\end{align}
			Since, $\mathfrak{p}_0> \mathfrak{p}_j>2$, so we can use Corollary \ref{homogeneous comparison at p0} with $\bar{\mathfrak{p}}=\mathfrak{p}_j$.  Thus using 	\eqref{homogeneous comparison bar p bigger than 2}, \eqref{estimates of j by j-1 for a, p, f}, \eqref{upper bound of f} and \eqref{def lambda1} we deduce
			\begin{align*}
				\fint_{B_{j}} \lvert \nabla_{\mathbb{E}}  w_{j-1} - \nabla_{\mathbb{E}}  w_{j}\rvert^{\mathfrak{p}_j} 
				&\leq c\tau^{-Q_{\mathbb{E}}}\fint_{B_{j-1}}\lvert \nabla_{\mathbb{E}} u  - \nabla_{\mathbb{E}} w_{j-1}\rvert^{\mathfrak{p}_j} + c \fint_{B_{j}}\lvert \nabla_{\mathbb{E}} u  - \nabla_{\mathbb{E}} w_{j}\rvert^{\mathfrak{p}_j} \notag \\
				&\leq \frac{c}{\tau^{2Q_{\mathbb{E}} + \bar{\mathfrak{p}}'Q_{\mathbb{E}}}} \left(\mathfrak{F}^{\frac{\mathfrak{p}_j}{\mathfrak{p}_0-1}}_{r_{j-1},\theta} + \left( \mathfrak{A}_{r_{j-1}, \bar{\sigma}} + \mathfrak{P}_{r_{j-1}, \bar{\sigma}} \right)^2\lambda^{\bar{\mathfrak{p}}}  \right).
			\end{align*}
			Finally, we use this estimate in \eqref{mid estimate of I3'} to derive 
			\begin{align*}
				I'_3&= \frac{c}{\lambda^{\mathfrak{p}_{j}-2}}\left( \fint_{\frac{1}{2}B_{j}} \left\lvert \nabla_{\mathbb{E}} w_{j} - \nabla_{\mathbb{E}} w_{j-1} \right\rvert^{\mathfrak{p}_{j}'(\mathfrak{p}_{j}-2)}\left\lvert \nabla_{\mathbb{E}}  v_{j} - \nabla_{\mathbb{E}}  w_{j} \right\rvert^{\mathfrak{p}_{j}'}\right)^{\frac{1}{\mathfrak{p}_{j}'}} \\ 
				&\leq c\mathfrak{P}_{r_{j-1},\bar{\sigma}} \lambda + c \lambda^{1-\mathfrak{p}_{j}} \left(\mathfrak{F}^{\frac{\mathfrak{p}_j}{\mathfrak{p}_0-1}}_{r_{j-1},\theta} + \left( \mathfrak{A}_{r_{j-1}, \bar{\sigma}} + \mathfrak{P}_{r_{j-1}, \bar{\sigma}} \right)^2\lambda^{\mathfrak{p}_j}\right)  \\ 
				&\stackrel{\eqref{smallness of mathfrak a+mathfrak p}}{\leq} c\left(\mathfrak{A}_{r_{j-1}, \bar{\sigma}} + \mathfrak{P}_{r_{j-1}, \bar{\sigma}}\right)\lambda + c \lambda^{1-\mathfrak{p}_j} \mathfrak{F}^{\frac{\mathfrak{p}_j}{\mathfrak{p}_0-1}-1}_{r_{j-1},\theta} \mathfrak{F}_{r_{j-1},\theta}\\
				&\stackrel{\eqref{smallness of wp}, \eqref{upper bound of f}}{\leq} c\left(\mathfrak{A}_{r_{j-1}, \bar{\sigma}} + \mathfrak{P}_{r_{j-1}, \bar{\sigma}}\right)\lambda + c \lambda^{1-\mathfrak{p}_j} \lambda^{\mathfrak{p}_j-\mathfrak{p}_0+1} \mathfrak{F}_{r_{j-1},\theta}\\
				&\leq c\left(\mathfrak{A}_{r_{j-1}, \bar{\sigma}} + \mathfrak{P}_{r_{j-1}, \bar{\sigma}}\right)\lambda + c \lambda^{2-\mathfrak{p}_0}  \mathfrak{F}_{r_{j-1},\theta}.
			\end{align*}
			Thus combining this estimates of $I'_3$ with the estimate of $I'_4$ in \eqref{I4' estimate} we prove 
			\begin{align}\label{final frozen-double frozen}
				\left( \fint_{\frac{1}{2}B_{j}} \left\lvert \nabla_{\mathbb{E}}  v_j - \nabla_{\mathbb{E}}  w_{j} \right\rvert^{\gamma}\right)^{\frac{1}{\gamma}} 
				&\leq \left( \fint_{\frac{1}{2}B_{j}} \left\lvert \nabla_{\mathbb{E}}  v_j - \nabla_{\mathbb{E}}  w_{j} \right\rvert^{\mathfrak{p}_j'}\right)^{\frac{1}{\mathfrak{p}_j'}}\notag\\
				&\leq  C\left( \mathfrak{A}_{r_{j-1}, \bar{\sigma}} + \mathfrak{P}_{r_{j-1}, \bar{\sigma}} \right)\lambda 
				+ C\lambda^{2-\mathfrak{p}_{0}}\mathfrak{F}_{r_{j-1}, \theta}.
			\end{align}
			Therefore, \eqref{final case linearized frozen} and  \eqref{final frozen-double frozen} together proves \eqref{grad u -grad v pbigger2} in this case. This completes the proof.
		\end{proof}
		\section{Proof of the main results} 
		\begin{theorem}\label{unified main them}
			Let $\Omega \subset \mathbb{E}^{n}$ be open and bounded. Let $\mathfrak{p}:\Omega \rightarrow [\gamma_{1}, \gamma_{2}]$ and $\mathfrak{a}:\Omega \rightarrow [\nu, L]$ be measurable functions  with $1 < \gamma_{1} \leq \gamma_{2} < \infty$ and $ 0 < \nu < L < \infty,$ such that for every $\tau \in (0, 1/4)$ and every $\sigma >0,$ we have 
			\begin{align*}
				\lim\limits_{r \rightarrow 0}	\sup\limits_{x \in K} &\sum\limits_{j=0}^{\infty}\left(\fint_{B_{\tau^{j}r}(x)} \left\lvert \mathfrak{a} - \left(\mathfrak{a}\right)_{x, \tau^{j}r} \right\rvert^{\frac{2(1+\sigma)}{\sigma}}\right)^{\frac{\sigma}{2(1+\sigma)}} = 0, \\
				\lim\limits_{r \rightarrow 0}	\sup\limits_{x \in K}&\sum\limits_{j=0}^{\infty}\log \left(\frac{1}{\tau^{j}r}\right)\left(\fint_{B_{\tau^{j}r}(x)} \left\lvert \mathfrak{p} - \left(\mathfrak{p}\right)_{x, \tau^{j}r} \right\rvert^{\frac{2(1+\sigma)}{\sigma}}\right)^{\frac{\sigma}{2(1+\sigma)}} = 0, 
			\end{align*}
			whenever $K \subset \subset \Omega$ is a compact subset.  If $u \in \mathbb{E}W_{\text{loc}}^{1,p\left(\cdot\right)} \left(\Omega\right)$ is a local weak solution to the equation 		\begin{align}\label{unifired main equation}
				\operatorname{div}_{\mathbb{E}}\left[ \mathfrak{a}(x) \lvert \nabla_{\mathbb{E}} u \rvert^{\mathfrak{p}\left(x\right)-2} \nabla_{\mathbb{E}} u\right]    = f   &&\text{ in } \Omega, 
			\end{align}
			for some $f \in L^{\left(Q_{\mathbb{E}},1\right)}_{\text{loc}}\left( \Omega\right),$ then $\nabla_{\mathbb{E}} u$ admits a continuous representative in $\Omega.$
		\end{theorem}
		
\begin{proof}
	
	Since the result is local, we pick $z_{0} \in \Omega$ and show $\nabla_{\mathbb{E}} u$ admits a continuous representative in a neighborhood of $z_{0}.$ To this end, we 
	choose open subsets $K \subset \subset K_{1} \subset \subset K_{2} \subset \subset K_{3}\subset \subset \Omega$ such that $z_{0} \in K.$ 
	Let 
	\begin{align}\label{distance def}
		d:= \min \left\lbrace 1, \operatorname*{dist}\left( \bar{K}, \partial K_{1}\right), \operatorname*{dist}\left( \bar{K}_{1}, \partial K_{2}\right), \operatorname*{dist}\left( \bar{K}_{2}, \partial K_{3}\right), \operatorname*{dist}\left( \bar{K}_{3}, \partial \Omega\right) \right\rbrace. 
	\end{align}Set 
	\begin{align}\label{total energy def}
		\tilde{\mathcal{K}}:= 1 + \int_{K_3} \left(1 + \left\lvert \nabla_{\mathbb{E}}u\right\rvert\right)^{\mathfrak{p}(x)}\ \mathrm{d}x + \int_{K_3} \left(1 + \left\lvert f\right\rvert\right)^{Q_{\mathbb{E}}}\ \mathrm{d}x. 
	\end{align}
	Since $u \in \mathbb{E}W_{\text{loc}}^{1,p\left(\cdot\right)} \left(\Omega\right),$ we see that $\tilde{\mathcal{K}} < +\infty.$ We set 
	\begin{align}
		\mathcal{K}_{1}=\mathcal{K}_{2} =\tilde{\mathcal{K}}. 
	\end{align}
	Note that by \eqref{energy bound f} and \eqref{energy bound f Global HI}, $\mathcal{K}_{1}$ and $\mathcal{K}_{2}$ satisfies the hypotheses of Theorem \ref{higher integrability f} and Theorem \ref{global_higher_integrability}, respectively, with `$\Omega$ replaced by $K_3$'. 
	Observe that by our assumption and Lemma \ref{local vanishing log holder}, $\mathfrak{p}$ is log-H\"{o}lder in $K_{3}.$ Let $\Lambda_{\log} >0$ be any number greater than or equal to the log-H\"{o}lder constant of $\mathfrak{p}$ in $K_{3}.$ Now, using the set of parameters
	\begin{align}
		\texttt{data} := \left\lbrace Q_{\mathbb{E}}, \gamma_{1}, \gamma_{2}, \nu, L, \Lambda_{\log} \right\rbrace, 
	\end{align} we can determine the constants  $C_4$, $R_{1}$ and $\sigma_{1}$ by Theorem \ref{higher integrability f}.  Using the same parameters with the choice $\tilde{\sigma}= \sigma_{1}$ in Theorem \ref{global_higher_integrability}, we determine the constants $C_{5}$, $R_{2}$ and $\sigma_{2}$. Set $\bar{\sigma}$ and $\bar{R}$ by \eqref{sigma bar and R bar} and $\theta$ by \eqref{theta def}. Thus, we can now determine $R_{4}$ by \eqref{smallness of wp} and \eqref{smallness of mathfrak a+mathfrak p} and use Lemma \ref{comparison lemma 1}, Lemma \ref{comparison lemma 2} to determine the constants $C_{6}$ and $C_{7}$, respectively.   
	
For now, choose any $0 < R < \min \left\lbrace 1, R_{4}/128, d/256 \right\rbrace.$ By our choice of $R$, clearly $B(x_{0}, 16R) \subset \subset K_{3}$ for any $x_{0} \in K_{2}.$ Fix $x_{0} \in K_{2}$ and for $j\geq 0,$ we set 
\begin{align}\label{concentric balls}
		B_{j}:= B(x_{0}, r_{j}), \qquad r_{j}:= \tau^{j}R, \quad \tau \in (0,1/4).
	\end{align} 
 We denote 
	\begin{align*}
		\mathfrak{a}_{j} := \fint_{ \frac{1}{4}B_{j}} \mathfrak{a}, \qquad \text{ and } \qquad 	\mathfrak{p}_{j} := \fint_{ \frac{1}{8}B_{j}} \mathfrak{p}, \qquad \text{ and } \qquad \mathfrak{p}_{0}= \mathfrak{p}(x_0).  
	\end{align*}
\noindent For $j \geq 0,$ we define $w_{j} \in u+\mathbb{E}W_{0}^{1, \mathfrak{p}(\cdot)}\left(\frac{1}{4}B_{j}\right)$ be the unique minimizer of 
	\begin{align*}
		\inf \left\lbrace \mathfrak{a}_{j}\int_{B_{j}} \frac{1}{\mathfrak{p}(x)}\left\lvert \nabla_{\mathbb{E}} w\right\rvert^{\mathfrak{p}\left(x\right)}\ \mathrm{d}x: w \in u + \mathbb{E}W_{0}^{1, \mathfrak{p}(\cdot)}\left(\frac{1}{4}B_{j}\right) \right\rbrace.  
	\end{align*}
	Clearly, $w_{j} \in u+\mathbb{E}W_{0}^{1, \mathfrak{p}(\cdot)}\left(\frac{1}{4}B_{j}\right)$ is the unique solution of 
	\begin{align}\label{quasilinear homogeneous general}
		\left\lbrace \begin{aligned}
			\operatorname{div}_{\mathbb{E}}\left(   \lvert \nabla_{\mathbb{E}} w_{j} \rvert^{\mathfrak{p}(x)-2} \nabla_{\mathbb{E}} w_{j}) \right) &= 0   &&\text{ in } \frac{1}{4}B_{j},\\
			w_{j} &= u &&\text{  on } \partial \left( \frac{1}{4}B_{j}\right).
		\end{aligned} 
		\right. 
	\end{align}
	Similarly, we define $v_{j} \in w_{j}+\mathbb{E}W_{0}^{1,\mathfrak{p}_{j}}\left(\frac{1}{8}B_{j}\right)$ be the unique minimizer of 
	\begin{align*}
		\inf \left\lbrace \int_{\frac{1}{8}B_{j}} \frac{1}{\mathfrak{p}_{j}}\left\lvert \nabla_{\mathbb{E}} w\right\rvert^{\mathfrak{p}_{j}}\ \mathrm{d}x: v \in w_{j} + \mathbb{E}W_{0}^{1,\mathfrak{p}_{j}}\left(\frac{1}{8}B_{j}\right) \right\rbrace.  
	\end{align*}
	Note that $v_{j}\in w_{j} + \mathbb{E}W_{0}^{1,\mathfrak{p}_{j}}\left( \frac{1}{8}B_{j}\right)$ is the unique weak solution to 
	\begin{align}\label{quasilinear homogeneous frozen general}
		\left\lbrace \begin{aligned}
			\operatorname{div}_{\mathbb{H}}\left(\left\lvert \nabla_{\mathbb{E}} v_{j} \right\rvert^{\mathfrak{p}_{j}-2} \nabla_{\mathbb{E}} v_{j} \right) &= 0  &&\text{ in }  \frac{1}{8}B_{j},\\
			v_{j} &= w_{j} &&\text{  on } \partial \left(\frac{1}{8}B_{j}\right).
		\end{aligned} 
		\right.
	\end{align}	 
	For notational ease, for any nonnegative integer $j$, we denote 
	\begin{align}\label{notation for excess and average of u}
		\mathscr{A}_j:= \lvert ( \nabla_{\mathbb{E}} u)_{B_j} \rvert \quad \text{ and } \quad \mathscr{E}_j &:= \left( \fint_{B_j} \lvert \nabla_{\mathbb{E}} u- (\nabla_{\mathbb{E}} u)_{B_j} \rvert^{\gamma} \right)^\frac{1}{\gamma}.
	\end{align}
	For $j\in \mathbb{N}$ we define the composite quantity 
	
	\begin{equation}\label{average on consecutive balls and excess}
		\mathscr{C}_j:= \left ( \fint_{B_{j-1}} \lvert  \nabla_{\mathbb{E}}u \rvert^{\gamma} \right)^\frac{1}{\gamma} + \left ( \fint_{B_{j}} \lvert  \nabla_{\mathbb{E}}u \rvert^{\gamma} \right)^\frac{1}{\gamma} + 2\tau^{-Q_{\mathbb{E}}} \mathscr{E}_j.
	\end{equation}
Finally, for every integer $j\geq 1,$ we set 
	\begin{align}\label{combined a and p}
		\Omega_{j, \sigma} := \mathfrak{A}_{r_{j-1}, \sigma} + \mathfrak{P}_{r_{j-1}, \sigma},
	\end{align}
	where $\mathfrak{A}_{r_{j-1},\sigma}$ and $\mathfrak{P}_{r_{j-1},\sigma}$ are as defined in \eqref{mathfrak a} and \eqref{mathfrak a}, respectively, and 
	\begin{align}\label{sigma choice}
		\sigma = \bar{\sigma}/4.
	\end{align}
	Recalling \eqref{mathfrak a}, \eqref{mathfrak p} and \eqref{mathfrak f}, we see that this implies 
	\begin{align}
		\sum_{j=1}^\infty \Omega_{j, \sigma} 
		&= \sum_{j=1}^\infty \mathfrak{A}_{r_{j-1}, \sigma} + \sum_{j=1}^\infty \mathfrak{P}_{r_{j-1}, \sigma} \notag\\
		&= \mathfrak{S}_{\left[\mathfrak{a}, 2(1+\sigma)/\sigma, \tau, 0\right]}\left(x_{0}, R \right) + \mathfrak{S}_{\left[\mathfrak{p}, 2(1+\sigma)/\sigma, \tau, 1\right]}\left(x_{0}, R \right) \quad \text{ and }\label{notation of the lower bound of mean oscillation sums}\\
		\sum_{j=0}^\infty \mathfrak{F}_{r_{j}, \theta}&= \mathcal{S}_{\left[f, \theta, \tau, 0 \right]}\left(x_{0}, R \right).  	\label{notation of the lower bound of lorentz norm}
	\end{align}
We set 
\begin{align}\label{gamma0 def}
	\gamma_{0}:= \gamma_{1}\gamma_{2}\gamma \cdot \frac{\gamma_{1}}{\gamma_{1}-1}\cdot \frac{\gamma_{2}}{\gamma_{2}-1} \cdot \frac{\gamma}{\gamma-1}. 
\end{align}	
	
	We divide the rest of the proofs in two steps. 
	
	\textbf{Step 1:} \underline{Boundedness} We first show $\nabla_{\mathbb{E}} u \in L^{\infty}(\bar{K}).$ \smallskip 
%
%

 For this, we shall prove the following estimate: There exists $\tau \in (0, 1/4),$ $\theta \in (1, Q_{\mathbb{E}})$ and $R_{0} >0$ such that for any $0 < R < R_{0}/16,$ 		
		\begin{align}\label{sufficient estimate for pointwise gradient bound}
			\lvert \nabla_{\mathbb{E}} u (x_0) \rvert \leq \lambda:= H_1 \left(1+\left( \fint_{B(x_0,R)} \lvert \nabla_{\mathbb{E}} u \rvert^{\gamma}  \right)^\frac{1}{\gamma}\right) + H_2^{\gamma_{0}} \left[ 	\mathcal{S}_{\left[f, \theta, \tau, 0 \right]}\left(x_{0}, R \right) \right ]^\frac{1}{\mathfrak{p}_{0}-1}.
		\end{align}
for some large constants $H_{1}, H_{2} >1$, independent of $x_{0}\in \bar{K}$ and $R_{0}$, whenever $x_{0}$ is a Lebesgue point of $\nabla_{\mathbb{E}} u.$	Note that by a standard covering argument, this implies $\nabla_{\mathbb{E}} u \in L^{\infty}(\bar{K}).$ Also, we may assume that $\lambda>0$, otherwise there is nothing to prove.

By Lebesgue differentiation theorem, in order to establish \eqref{sufficient estimate for pointwise gradient bound} we only need to establish $\mathscr{A}_{j_{i}} \le \lambda$ for a subsequence $\left\lbrace j_i\right\rbrace_{i \in \mathbb{N}}$. We divide the proof in three substeps. 

\textbf{Step 1a: Choice of the constants and exit time.}		

We set 
\begin{equation}\label{constants choice 1}
	A:= 10^{10Q_{\mathbb{E}}}2^{2\gamma_2+6} C_1C_3C_4C_5 \qquad \text{ and } \qquad   \varepsilon:= 10^{-25}. 
\end{equation}
where $C_{1}, C_{3}$ are the constants given by Theorem \ref{Uhlenbeck estimate} and Theorem \ref{homogenous eqn estimate}, respectively. 

Now clearly, for any $j\geq 0$, by minimality of $w_j$ we have 
$$C(\gamma_1,\gamma_2,Q_{\mathbb{E}}) \left(1+ \int_{K_3} \left(1+ \left\lvert \nabla_{\mathbb{E}} u \right \rvert\right)^{p(x)}\, dx,\right) \geq 1+ \int_{\frac{1}{4}B_j} \left(1+ \left\lvert \nabla_{\mathbb{E}} w_j \right \rvert\right)^{p(x)}\, dx,$$
for some constant $C(\gamma_1,\gamma_2,Q_{\mathbb{E}})>1$. This implies $\mathcal{K}_0:= C(\gamma_1,\gamma_2,Q_{\mathbb{E}}) \tilde{\mathcal{K}}$ satisfy
\eqref{energy bound w by u} and hence Theorem \ref{homogenous eqn estimate} determines the radius $R_3$ (independent of $j$), and with $A$ and $\varepsilon$ fixed as above, also the parameter $\tau_{1}.$ Note that, $\eqref{def gamma}$ implies $1\leq \gamma\leq 2$. So, we use Theorem \ref{Uhlenbeck estimate} again to determine the constant $C_{\gamma}$, i.e. the constant $C_{q}$ in  estimate \eqref{excess decay} with the choice $q=\gamma$. Finally, let $\beta \in (0,1)$ and choose $\tau_{2} \in (0, 1/4)$ small enough such that 
\begin{align}\label{choice of tau2}
	C_{\gamma}\tau_{2}^{\beta} < 4^{-(4Q_{\mathbb{E}}+16)}. 
\end{align}
Set 
\begin{align}\label{choice of tau}
	\tau := \frac{1}{16}\min \left\lbrace \tau_{1}, \tau_{2} \right\rbrace.
\end{align} 
Note that, this choice implies that $B_{j+k} \subset 4^{-1}B_{j}$, for all $j\geq 0$ and any $k\geq 1$. Now we use Lemma \ref{linearized comparison} to determine the constant $C_{8}$ with $A$ and $\tau$ as defined above.  
Finally, we define 
\begin{equation}\label{order on C}
	C_0:= C_{1}\cdot C_{\gamma} \cdot \prod_{i=3}^8 C_i.
\end{equation}

We now make the choices of $H_1$, $H_2$ and $R_0$. We first define 

\begin{align}\label{choice of H1 and H2}
	\begin{cases}
		H_1 &:= 10^{10Q_{\mathbb{E}}}\tau^{-2Q_{\mathbb{E}}}\\
		H_2 &:= 10^{Q_{\mathbb{E}} + 20} \tau^{-7Q_{\mathbb{E}}} C_0 . 
	\end{cases}
\end{align}	
Now we choose $0 < R_{0} < \min \left\lbrace R_{1}, R_{2}, R_{3}, R_{4}, d \right\rbrace/256$, small enough such that 
\begin{align}\label{choice of R0}
	\sum_{j=1}^\infty \Omega_{j, \sigma} \leq \frac{\tau^{2Q_{\mathbb{E}}}}{10^{30} H_{2}^{\gamma_{0}}} \qquad \text{ for all }\qquad  0 < R < \frac{R_{0}}{16}.  
\end{align}	
By our choice of $H_{1}$ and $H_{2},$ we have 
		\begin{align*}
			\mathscr{C}_{1} \le 6 \tau^{-Q_{\mathbb{E}}-\frac{Q_{\mathbb{E}}}{\gamma}} \left( \fint_{B(x_0,R)} \lvert \nabla_{\mathbb{E}} u \rvert^{\gamma}  \right)^\frac{1}{\gamma} \le \frac{6 \tau^{-2Q_{\mathbb{E}}}\lambda}{H_{1}} \le \frac{\lambda}{1000}. 
		\end{align*}
		Now without loss of generality, we can assume that there exists an \emph{exit time} $j_{e} \geq 1$ such that 
		\begin{align}\label{exit time}
			\mathscr{C}_{j_{e}} \le \frac{\lambda}{1000} \qquad \text{ and } \qquad \mathscr{C}_{j} > \frac{\lambda}{1000} \quad \text{ for all } j > j_{e}. 
		\end{align}
		Indeed, if not, then there exists a subsequence $\{j_m\}_m$ such that $\mathscr{C}_{j_m} \leq \frac{\lambda}{1000}$  for all $m$. But this implies 
		
		\begin{equation*}
			\lvert \nabla_{\mathbb{E}}u(x_0) \rvert = \lim_{m \to \infty } \mathscr{A}_{j_m} \leq \limsup_{m\to \infty} \mathscr{C}_{j_m} \leq \lambda,  
		\end{equation*}
		whenever $x_0$ is a Lebesgue point of $\nabla_{\mathbb{E}} u$ and we have \eqref{sufficient estimate for pointwise gradient bound} follows. For the rest of the proof,  we assume \eqref{exit time} holds. 

%
		%
		%
		%
		
		We also have the estimate which trivially follows form the defintion of the quantity $\mathcal{S}_{\left[f, \theta, \tau, 0 \right]}\left(x_{0}, R \right)$, \eqref{sufficient estimate for pointwise gradient bound} and \eqref{notation of the lower bound of lorentz norm} 
		
		\begin{equation}\label{lorentz norm lambda upper bound}
			\mathfrak{F}_{r_{j}, \theta} \le \left(\frac{\lambda}{H_{2}^{\gamma_{0}}} \right)^{\mathfrak{p}_{0}-1}  \qquad \text{ for all } j\geq 0.
		\end{equation}
		We are now ready to proceed to the next substep. 
		
		\textbf{Step 1b: Upper and lower bounds, and a decay estimate.} For $j\geq j_e$ we now consider the following condition: 
		
		\begin{equation}\label{condition ind j}
			\rm{Ind}(j): \quad 1+  \max \left\{ \left ( \fint_{B_{j-1}} \lvert \nabla_{\mathbb{E}} u \rvert^{\gamma} \right)^\frac{1}{\gamma}, \left ( \fint_{B_{j}} \lvert \nabla_{\mathbb{E}} u \rvert^{\gamma} \right)^\frac{1}{\gamma} \right\} \leq \lambda.
		\end{equation}
		
		We now prove some consequences of \eqref{condition ind j}. 
		
		\begin{claim}\label{energy bound implies oscillation bound claim} 
			If $\rm{Ind}(j)$ holds, then the following estimates also hold.  
			\begin{equation}\label{upper and lower bound estimates}
				\sup_{\frac{1}{8} B_{j-1}} \lvert \nabla_{\mathbb{E}} w_{j-1} \rvert \leq A\lambda, \ \ \sup_{\frac{1}{8}B_j} \ \lvert \nabla_{\mathbb{E}} w_j \rvert \leq A\lambda, \ \ \frac{\lambda}{A} \leq \inf\limits_{\frac{1}{4}B_j} \  \lvert \nabla_{\mathbb{E}} w_{j-1} \rvert.
			\end{equation}
		\end{claim}
		
		\emph{Proof of Claim \ref{energy bound implies oscillation bound claim}: } 
		
		Let $\bar{x} \in \overline{\frac{1}{4}B_{j-1}}$ be the point such that $\bar{\mathfrak{p}}=\mathfrak{p}(\bar{x}) = \min\limits_{x \in \overline{\frac{1}{4}B_{j-1}}} \mathfrak{p}(x).$ We have 
		\begin{align*}
			\sup_{\frac{1}{8} B_{j-1}} \lvert \nabla_{\mathbb{E}} w_{j-1} \rvert &\stackrel{\eqref{sup bound homogeneous}}{\le} C_3 \fint_{ \frac{1}{4}B_{j-1}} \left(1 + \left\lvert \nabla_{\mathbb{E}} w_{j-1} \right\rvert\right) \\
			&\le C_3 \left(\fint_{ \frac{1}{4}B_{j-1}} \left(1 + \left\lvert \nabla_{\mathbb{E}} w_{j-1} \right\rvert\right)^{\bar{\mathfrak{p}}}\right)^{\frac{1}{\bar{\mathfrak{p}}}}\\
			&\le C_3 \left(\fint_{ \frac{1}{4}B_{j-1}} \left(1 + \left\lvert \nabla_{\mathbb{E}} w_{j-1} \right\rvert\right)^{\mathfrak{p}(x)}\right)^{\frac{1}{\bar{\mathfrak{p}}}}\\
			&\le C_3 \left(\fint_{ \frac{1}{4}B_{j-1}} \left(1 + \left\lvert \nabla_{\mathbb{E}} w_{j-1} \right\rvert\right)^{\mathfrak{p}(x)\left(1 + \sigma\right)}\right)^{\frac{1}{\bar{\mathfrak{p}}(1 + \sigma)}} 
		\end{align*}
        Using \eqref{Global HI}, we estimate  
		\begin{align*}
			\sup_{\frac{1}{8} B_{j-1}} \lvert \nabla_{\mathbb{E}} w_{j-1} \rvert 
			&\le C_3 \left(\fint_{ \frac{1}{4}B_{j-1}} \left(1 + \left\lvert \nabla_{\mathbb{E}} w_{j-1} \right\rvert\right)^{\mathfrak{p}(x)\left(1 + \sigma\right)}\right)^{\frac{1}{\bar{\mathfrak{p}}(1 + \sigma)}}\\
			&\stackrel{\eqref{Global HI}}{\le} C_{3}C_{5} 2^{2\gamma_2+1} \left( 1 + \fint_{ \frac{1}{2}B_{j-1}} \left\lvert \nabla_{\mathbb{E}} u \right\rvert^{\mathfrak{p}(x)\left(1 + \sigma\right)}\right)^{\frac{1}{\bar{\mathfrak{p}}(1 + \sigma)}} \\
			&\stackrel{\eqref{higher integrability estimate f}}{\le}C_{3}C_{4}C_{5} 2^{2\gamma_2+4}  \left[ \left( \fint_{B_{j-1}}\left\lvert \nabla_{\mathbb{E}} u \right\rvert^{\gamma}\ \mathrm{d}x \right)^{\frac{1}{\gamma}}  + \mathfrak{F}_{r_{j}, \theta}^{\frac{1}{\left(\mathfrak{p}_{0}-1\right)}} + 1 \right] \le A\lambda. 
		\end{align*}
		Similarly, we have $\sup_{\frac{1}{8} B_{j}} \lvert \nabla_{\mathbb{E}} w_{j} \rvert \le A\lambda. $ Now  observe that in view of \eqref{exit time}, we have 
		\begin{align*}
			\frac{\lambda}{1000} < \mathscr{C}_{j+k+1} \le \left ( \fint_{B_{j+k}} \lvert  \nabla_{\mathbb{E}}u \rvert^{\gamma} \right)^\frac{1}{\gamma}
		\end{align*}  
		for any $j > j_{e}$ and any $k \ge 1.$ Thus, we have 
		\begin{align}\label{lower bound on energy of u}
			\frac{\lambda}{10^3} \le \left ( \fint_{B_{j+k}} \lvert  \nabla_{\mathbb{E}}u \rvert^{\gamma} \right)^\frac{1}{\gamma} \le \left ( \fint_{B_{j+k}} \lvert  \nabla_{\mathbb{E}}w_{j-1} \rvert^{\gamma} \right)^\frac{1}{\gamma} + \left ( \fint_{B_{j+k}} \lvert  \nabla_{\mathbb{E}}u - \nabla_{\mathbb{E}}w_{j-1} \rvert^{\gamma} \right)^\frac{1}{\gamma}.
		\end{align}
		By \eqref{choice of R0}, we have $\mathfrak{A}_{r_{j-1}, \sigma} + \mathfrak{P}_{r_{j-1}, \sigma} < 10^{-6} H_{2}^{-\gamma_{0}}.$ Thus, if $\mathfrak{p}_{0}:= \mathfrak{p}(x_{0}) >2,$ we have 
		\begin{align*}
			\fint_{\frac{1}{4}B_{j-1}}  \lvert  \nabla_{\mathbb{E}}u -\nabla_{\mathbb{E}} w_{j-1}\rvert^{\mathfrak{p}_{0}} &\stackrel{\eqref{comparison lemma 1 bar p bigger than 2}}{\le} C_{6} \mathfrak{F}_{r_{j-1}, \theta}^{\frac{\mathfrak{p}_{0}}{\mathfrak{p}_{0}-1}} + \frac{C_{6}}{10^{12}H_{2}^{2\gamma_{0}}}\lambda^{\mathfrak{p}_{0}} \\
			&\stackrel{\eqref{lorentz norm lambda upper bound}}{\le} C_{6}\left(\left(\frac{\lambda}{H_{2}^{\gamma_{0}}}\right)^{\mathfrak{p}_{0}} + \left( \frac{\lambda}{H_{2}^{\gamma_{0}}}\right)^{\mathfrak{p}_{0}}\right) \le \left(\frac{\lambda}{10^{6}H_{2}}\right)^{\mathfrak{p}_{0}}. 
		\end{align*} 
		Similarly, if $\mathfrak{p}_{0}:= \mathfrak{p}(x_{0}) \le 2,$ we have 
		\begin{align*}
			\fint_{B_{j-1}}  \lvert  \nabla_{\mathbb{E}}u -\nabla_{\mathbb{E}} w_{j-1}\rvert^{\mathfrak{p}_{0}} 		
			&\stackrel{\eqref{comparison lemma 1 bar p smaller than 2}}{\le} C_{6}\left(2  \lambda\right)^{\mathfrak{p}_{0}\left(2 -\mathfrak{p}_{0}\right)}  \mathfrak{F}_{r_{j-1}, \theta}^{\mathfrak{p}_{0}} + C_{6}\left( \frac{2  \lambda}{H_{2}^{\gamma_{0}}} \right)^{\mathfrak{p}_{0}}\\
			&\stackrel{\eqref{lorentz norm lambda upper bound}}{\le} \left(\frac{\lambda}{10^{6}H_{2}}\right)^{\mathfrak{p}_{0}}.
		\end{align*} 
		Thus, in either case, since $\gamma < \mathfrak{p}_{0}$, by H\"{o}lder inequality, we have 
		\begin{align*}
				\left( \fint_{B_{j-1}}  \lvert  \nabla_{\mathbb{E}}u -\nabla_{\mathbb{E}} w_{j-1}\rvert^{\gamma} \right)^{\frac{1}{\gamma}} \le \left( 	\fint_{B_{j-1}}  \lvert  \nabla_{\mathbb{E}}u -\nabla_{\mathbb{E}} w_{j-1}\rvert^{\mathfrak{p}_{0}}\right)^{\frac{1}{\mathfrak{p}_{0}}}\le \frac{\lambda}{10^{6}H_{2}}. 
		\end{align*}
Hence, we arrive at 
\begin{align*}
	\left( \fint_{B_{j+2}}  \lvert  \nabla_{\mathbb{E}}u -\nabla_{\mathbb{E}} w_{j-1}\rvert^{\gamma} \right)^{\frac{1}{\gamma}} \le \tau^{-3Q_{\mathbb{E}}}\left( \fint_{B_{j-1}}  \lvert  \nabla_{\mathbb{E}}u -\nabla_{\mathbb{E}} w_{j-1}\rvert^{\gamma} \right)^{\frac{1}{\gamma}} \le \frac{\lambda}{10^{6}}. 
\end{align*}
Thus, in view of \eqref{lower bound on energy of u}, we have 
\begin{align*}
	\left ( \fint_{B_{j+2}} \lvert  \nabla_{\mathbb{E}}w_{j-1} \rvert^{\gamma} \right)^\frac{1}{\gamma} \ge \frac{\lambda}{10^{3}} - \frac{\lambda}{10^{6}} \ge \frac{\lambda}{10^{6}}. 
\end{align*}
Thus, we deduce 
\begin{align*}
	\sup\limits_{\frac{1}{4}B_j} \  \lvert \nabla_{\mathbb{E}} w_{j-1} \rvert \ge \sup\limits_{B_{j+2}} \  \lvert \nabla_{\mathbb{E}} w_{j-1} \rvert \ge 	\left ( \fint_{B_{j+2}} \lvert  \nabla_{\mathbb{E}}w_{j-1} \rvert^{\gamma} \right)^\frac{1}{\gamma} \ge \frac{\lambda}{10^{3}} - \frac{\lambda}{10^{6}} \ge \frac{\lambda}{10^{6}}.
\end{align*}
Now, in view of our choice of $\varepsilon$ and $\tau$ (in particular, the fact that $2\tau < \tau_{1}$) and $\sup\limits_{\frac{1}{8}B_{j-1}} \  \lvert \nabla_{\mathbb{E}} w_{j-1} \rvert \le A\lambda$, applying Theorem \ref{homogenous eqn estimate} to $w_{j-1}$, we have 
\begin{align*}
	\sup\limits_{x, y \in \frac{1}{4}B_{j}} \left\lvert \nabla_{\mathbb{E}} w_{j-1} \left( x\right) - \nabla_{\mathbb{E}} w_{j-1} \left( y\right) \right\rvert \leq  \frac{\lambda}{10^{25}}. 
\end{align*}
The lower bound for $\inf\limits_{\frac{1}{4}B_j} \  \lvert \nabla_{\mathbb{E}} w_{j-1} \rvert$ follows from this of the obvious estimate 
\begin{align*}
	\inf\limits_{\frac{1}{4}B_j} \  \lvert \nabla_{\mathbb{E}} w_{j-1} \rvert \ge \sup\limits_{\frac{1}{4}B_j} \  \lvert \nabla_{\mathbb{E}} w_{j-1} \rvert - \sup\limits_{x, y \in \frac{1}{4}B_{j}} \left\lvert \nabla_{\mathbb{E}} w_{j-1} \left( x\right) - \nabla_{\mathbb{E}} w_{j-1} \left( y\right) \right\rvert.  
\end{align*} This completes the proof of Claim \ref{energy bound implies oscillation bound claim}. \smallskip 

In view of the bounds in Claim \ref{energy bound implies oscillation bound claim}, Lemma \ref{linearized comparison}, Theorem \ref{Uhlenbeck estimate}, specifically \eqref{excess decay} (applied to $v_{j}$), and \eqref{choice of tau2}, by standard estimates, we deduce  
	\begin{align}\label{excess decay beyond exit time estimate}
		\rm{Ind}(j) \implies \mathscr{E}_{j+1} \le \frac{1}{4} \mathscr{E}_{j} +2C_{8}\left( \mathfrak{A}_{r_{j-1}, \bar{\sigma}} + \mathfrak{P}_{r_{j-1}, \bar{\sigma}} \right)\lambda 
		+ 2C_{8}\lambda^{2-\mathfrak{p}_0}\mathfrak{F}_{r_{j-1}, \theta}.
	\end{align}

\textbf{Step 1c: Final Induction.} In this step we will prove 

\begin{align}\label{final average plus excess bound}
	\mathscr{A}_{j} +  \mathscr{E}_{j} \leq \lambda , \quad \forall j \geq j_e.
\end{align}
Note that \eqref{final average plus excess bound} would complete the proof since   $\lvert \nabla_{\mathbb{E}} u(x_0) \rvert = \lim\limits_{j\to \infty} \mathscr{A}_{j}$ whenever $x_0$ is a Lebesgue point of $\nabla_{\mathbb{E}}u$. Clearly, \eqref{exit time} implies that \eqref{final average plus excess bound} holds for $j=j_e$. Thus we assume that \eqref{final average plus excess bound} holds for $j\in \{j_e, j_e+1, \dots, i\}$ and will prove that \eqref{final average plus excess bound} holds for $j=i+1$ as well. We now claim the following. 

\begin{claim}\label{Indj holds claim}
	\rm{Ind}(j) holds for all $j\in \{j_e, j_e+1, \dots, i\}$.
\end{claim}

\emph{Proof of Claim \ref{Indj holds claim}:} The validity of $\rm{Ind}(j_e)$ is a direct consequence of \eqref{exit time}. Now if $j>j_e$ then we can use \eqref{final average plus excess bound} and the following trivial estimates to prove that $\rm{Ind(j)}$ holds.

$$
\left ( \fint_{B_{m}} \lvert \nabla_{\mathbb{E}} u \rvert^{\gamma} \right)^\frac{1}{\gamma} \leq 	\mathscr{A}_{m} +  \mathscr{E}_{m} , \mbox{ for } m=j, j-1.
$$
This completes the proof of Claim \ref{Indj holds claim}. Now as a consequence of Claim \ref{Indj holds claim} we can use \eqref{excess decay beyond exit time estimate} for $j\in \{j_e, j_e+1, \dots, i\}.$  Summing up yields
\begin{align*}
\sum\limits_{j=j_e+1}^{i+1}	\mathscr{E}_{j} \le \frac{1}{2}  \sum\limits_{j=j_e}^{i}\mathscr{E}_{j} +2C_{8}\lambda \sum\limits_{j=j_e}^i\left( \mathfrak{A}_{r_{j-1}, \sigma} + \mathfrak{P}_{r_{j-1}, \sigma} \right) 
	+ 2C_{8}\lambda^{2-\mathfrak{p}_0} \sum\limits_{j=j_e}^{i}\mathfrak{F}_{r_{j-1}, \theta}.
\end{align*}
Hence we deduce 
\begin{align}\label{final estimate for pointwise bound}
	\sum\limits_{j=j_e}^{i+1} \mathscr{E}_{j} &\le 2 \mathscr{E}_{j_{e}}+ 4C_{8}\lambda \sum\limits_{j=0}^\infty \left( \mathfrak{A}_{r_{j}, \sigma} + \mathfrak{P}_{r_{j}, \sigma} \right)  + 4C_{8}\lambda^{2-\mathfrak{p}_0} \sum\limits_{j=0}^{\infty}\mathfrak{F}_{r_{j}, \theta}\notag\\
	&\stackrel{\eqref{notation for excess and average of u}, \eqref{choice of R0}}{\leq} \tau^{Q_{\mathbb{E}}} \mathscr{C}_{j_e} +\frac{\tau^{Q_{\mathbb{E}}}\lambda}{10^{25}}+ \frac{4C_8\lambda}{H_2^{\gamma_{0}(\mathfrak{p}_{0}-1)}} \stackrel{\eqref{exit time},\eqref{choice of H1 and H2}}{\leq} \frac{\tau^{Q_{\mathbb{E}}} \lambda}{500}.  
\end{align}
Finally, notice that 
\begin{align*}
	\mathscr{A}_{i+1} - 	\mathscr{A}_{j_e} = \sum\limits_{j=j_e}^i \left[ 	\mathscr{A}_{j+1} - 	\mathscr{A}_{j}\right]
	&\le \sum\limits_{j=j_e}^i \fint_{B_{j+1}} \lvert \nabla_{\mathbb{E}} u- (\nabla_{\mathbb{E}}u)_{B_j} \rvert \\
	&\le \tau^{-Q_{\mathbb{E}}} \sum\limits_{j=j_e}^i \mathscr{E}_{j} \stackrel{\eqref{final estimate for pointwise bound}}{\leq} \frac{\lambda}{500}
\end{align*}
and therefore \eqref{exit time} yields $\mathscr{A}_{i+1}  \le \mathscr{A}_{j_e} +\lambda/500 \leq \mathscr{C}_{j_e} +\lambda/500 \leq \lambda/2$. Combining the preceding estimate with \eqref{final estimate for pointwise bound} we establish $\mathscr{A}_{i+1} + 	\mathscr{E}_{i+1} \le \lambda$. This verifies the induction step and completes our proof of \eqref{final average plus excess bound}. As a consequence, we have established \eqref{sufficient estimate for pointwise gradient bound}, which in turn implies $\nabla_{\mathbb{E}}u \in L^{\infty}\left(K\right).$ \bigskip

\textbf{Step 2:} \underline{Continuity} We now prove $\nabla_{\mathbb{E}} u$ admits a continuous representative in $K$. Since the argument is very similar to Kuusi-Mingione \cite{KuusiMingione_nonlinearStein}, we only provide a brief sketch. Set
\begin{align}\label{lambda in stein theorem}
	\lambda:= \lvert \lvert \nabla_{\mathbb{E}}u\rvert \rvert_{L^\infty(K_{2})}+1.
\end{align}
We shall show 
$\nabla_{\mathbb{E}}u$  agrees a.e. with the uniform limit of a net of continuous functions defined via averages. For this, we show that given any $\bar{\varepsilon} >0,$ there exists  
\begin{equation}\label{radius r*}
	r_{\bar{\varepsilon}} \leq \operatorname{dist} \left( K_{1}, \partial K_{2}\right)/100 =: R_*, 
\end{equation}
depending only on $d, Q_{\mathbb{E}}, \mathfrak{p}, \mathfrak{a}, \gamma_{1}, \gamma_{2}, \nu, L $, $\lvert\lvert f\rvert\rvert_{L^{(Q_{\mathbb{E}},1)}}$ and $\bar{\varepsilon}$, such that 
\begin{equation}\label{uniform limit estimate}
	\lvert (\nabla_{\mathbb{E}}u)_{B_\rho(y_0)}- (\nabla_{\mathbb{E}}u)_{B_\varrho(y_0)} \rvert \le \lambda \bar{\varepsilon} , \mbox{ holds for every } \rho , \varrho \in (0,r_{\bar{\varepsilon}}], 
\end{equation}
whenever $y_0\in K_{1}$. By Lebesgue differentiation theorem, this proves that $\nabla_{\mathbb{E}}u$ agress a.e. with the uniform limit of continuous maps $y_0 \mapsto (\nabla_{\mathbb{E}}u)_{B_\rho(y_0)}$ and hence admits a continuous representative.

For the rest of the proof we fix $\bar{\varepsilon}$ and a point $y_{0} \in K_{1}.$ All the balls are now defined exactly as before, with the only exception being their common centers, which is now $y_{0} \in K_{1}.$ Hence $\mathfrak{p}_{0}= \mathfrak{p}(y_{0})$ now. We divide the proof in two substeps.\smallskip 

\textbf{Step 2a: Smallness of the excess. } First we begin by recalling that the constants $\mathcal{K}_{0},\mathcal{K}_{1}, \mathcal{K}_{2}, \tilde{\mathcal{K}}$ $\Lambda_{\log},$ $C_{1}, C_{3},$  $C_{4}, C_{5}, C_{4}, C_{6}, C_{7}, C_{\gamma}, \beta,  R_{1}, R_{2}, R_{3},$ $\bar{R}, R_{4}$ $ \sigma_{1}, \sigma_{2}, \bar{\sigma}$ and $\theta$ are all determined already. We set $\sigma$ as in \eqref{sigma choice} and $\gamma_{0}$ as in \eqref{gamma0 def}. Now we set 
\begin{equation}\label{constants choice continuity}
	A:= \frac{1}{\bar{\varepsilon}}10^{10Q_{\mathbb{E}}} \max\{C_1C_3\} \qquad \text{ and } \qquad   \varepsilon:= \bar{\varepsilon}\cdot 10^{-25}. 
\end{equation}
With $A$ and $\varepsilon$ fixed as above, Theorem \ref{homogenous eqn estimate} determines the parameter $\tau_{1}.$ We choose $\tau_{2} \in (0,1/4)$ small enough such that 
\begin{align}
	C_{\gamma}\tau_{2}^{\beta} < \bar{\varepsilon}\cdot 4^{-(4Q_{\mathbb{E}}+16)}. 
	\end{align}
	We set $\tau$ by \eqref{choice of tau} with this new $\tau_{1}, \tau_{2}.$ Now we use Lemma \ref{linearized comparison} to determine the constant $C_{8}$ with $A$ and $\tau$ as defined above. We set $C_{0}$ as before by \eqref{order on C}.  

Since $f \in L^{\left( Q_{\mathbb{E}}, 1\right)}\left(K_{3}\right)$, by Corollary \ref{vanishing corollary derivative version}, by choosing the radius small, we can make $\mathcal{S}_{\left[ f, \theta, \tau, 0\right]}\left(y_{0}, R\right)$ as small as we like. Similarly, by our assumptions on $\mathfrak{a}$ and $\mathfrak{p}$, the sums $\mathfrak{S}_{\left[ \mathfrak{a}, 2(1+\sigma)/\sigma, \tau, 0\right]}\left(y_{0}, R\right)$ and $\mathfrak{S}_{\left[ \mathfrak{p}, 2(1+\sigma)/\sigma, \tau, 1\right]}\left(y_{0}, R\right)$ converges to zero as $r$ goes to zero, uniformly in $K_{3}.$ Thus, we can choose $R_{\bar{\varepsilon}} \in (0, R_0/16)$ small enough such that 
\begin{align}
	\mathfrak{S}_{\left[ \mathfrak{a}, 2(1+\sigma)/\sigma, \tau, 0\right]}\left(y_{0}, R\right) &\le \bar{\varepsilon} \cdot \frac{ \tau^{12\gamma_{0}Q_{\mathbb{E}}}}{10^{\gamma_{0}\left(Q_{\mathbb{E}}+20\right)} C_{0}}\cdot \lambda^{\mathfrak{p}_{0}-1} \label{choice of smallness of f sum continuity} \intertext{ and } 
	\sum\limits_{j=1}^{\infty} \Omega_{j, \sigma} &\le \bar{\varepsilon} \cdot \frac{ \tau^{12\gamma_{0}Q_{\mathbb{E}}}}{10^{\gamma_{0}\left(Q_{\mathbb{E}}+20\right)} C_{0}} \label{choice of smallness of a and p sum continuity}, 
\end{align}
for all $0 < R < R_{\bar{\varepsilon}}.$ With all these choices fixed, the choice of $\lambda$ in \eqref{lambda in stein theorem} implies 
\begin{equation}\label{condition ind j st}
	\max \left\{ \left ( \fint_{B_{j-1}} \lvert  \nabla_{\mathbb{E}}u \rvert^{\gamma} \right)^\frac{1}{\gamma}, \left ( \fint_{B_{j}} \lvert  \nabla_{\mathbb{E}}u \rvert^{\gamma} \right)^\frac{1}{\gamma} \right\} \leq \lambda, \qquad \text{ for all } j \ge 1.
\end{equation}
Now for $j \ge 1,$ we define 
	\begin{equation}\label{condition indj*}
	\rm{Ind}^*(j) :  \quad  \left(\fint_{B_{j+1}} \lvert \nabla_{\mathbb{E}} u \rvert^{\gamma} \right)^\frac{1}{\gamma} \ge \frac{\lambda \bar{\varepsilon}}{50}.
\end{equation}
Now, arguing exactly the same way as in Step 1, but at scale $\bar{\varepsilon}$, where the condition \eqref{condition indj*} serves as a replacement of the exit time condition $\mathscr{C}_{j} > \lambda/1000$ in \eqref{exit time}, we can deduce the same upper and lower bounds as in Claim \ref{energy bound implies oscillation bound claim} (with the new choice of $A$) and this implies, in exactly the same manner, the following decay estimate. 
 \begin{align}\label{excess decay without exit time estimate}
 		\rm{Ind}^*(j) \implies \mathscr{E}_{j+1} \le \frac{\bar{\varepsilon}}{4} \mathscr{E}_{j} +2C_{8}\left( \mathfrak{A}_{r_{j-1}, \bar{\sigma}} + \mathfrak{P}_{r_{j-1}, \bar{\sigma}} \right)\lambda 
 	+ 2C_{8}\lambda^{2-\mathfrak{p}_0}\mathfrak{F}_{r_{j-1}, \theta}.
 \end{align}
But this implies that we have 
\begin{align}\label{smallness of the excess}
\mathscr{E}_{j+1} \le \bar{\varepsilon}\lambda \qquad \text{ for all } j \ge 1 \quad \text{ whenever }   0 < R < R_{\bar{\varepsilon}}.
\end{align}
Indeed, for any $j \geq 1,$ if $\rm{Ind}^*(j)$ does not hold, then by \eqref{minimality of mean}, we clearly have  
\begin{align*}
	\mathscr{E}_{j+1} \le 2 \left(\fint_{B_{j+1}} \lvert \nabla_{\mathbb{E}} u \rvert^{\gamma} \right)^\frac{1}{\gamma} < \frac{\lambda \bar{\varepsilon}}{25}. 
\end{align*}
On the other hand, if $\rm{Ind}^*(j)$ holds for some $j \ge 1,$ then summing up the excess decay estimate in \eqref{excess decay without exit time estimate} and using \eqref{choice of smallness of f sum continuity} and \eqref{choice of smallness of a and p sum continuity}, we again have \eqref{smallness of the excess}.

\textbf{Step 2b: Conclusion of the proof.} Now we finish the proof. Note that in Step 2a, $y_{0} \in K_{1}$ was arbitrary and the estimate \eqref{smallness of the excess} holds uniformly in $y_{0} \in K_{1}.$ Thus, we can choose a radius $0 < r_{\bar{\varepsilon}} < R_{\bar{\varepsilon}}$ such that we have 
\begin{align}\label{final excess decay st}
	\sup_{0<\varrho \leq r_{\bar{\varepsilon}}} \sup _{x\in K_{1}} \left( \fint_{B(x,\varrho)} \left\lvert \nabla_{\mathbb{E}} u - \left(\nabla_{\mathbb{E}} u\right)_{B(x,\varrho)}\right\rvert^{\gamma} \right)^{\frac{1}{\gamma}} \leq \frac{\tau^{16Q_{\mathbb{E}}} \lambda \bar{\varepsilon}}{10^{25}}.
\end{align} 
Now we choose any point $\bar{x} \in K.$ Now we consider shrinking balls $B_{j}$ with center $\bar{x}$, the same $\tau$ as Step 2a and the starting radius being equal to $r_{\bar{\varepsilon}}/16.$ We claim 
\begin{claim}\label{difference of average claim}
\begin{equation}\label{difference of average}
		\lvert ( \nabla_{\mathbb{E}} u)_{B_h}- (\nabla_{\mathbb{E}}  u )_{B_k} \rvert \leq \frac{\lambda \varepsilon}{12} \qquad \text{ holds whenever } 2\leq k\leq h.
	\end{equation}
\end{claim}
Claim \ref{difference of average claim} implies \eqref{uniform limit estimate} by a standard interpolation argument.  For Claim \ref{difference of average claim}, we only sketch the idea. We consider the set 
\begin{align*}
	\mathcal{L} := \left\{ j\in \mathbb{N}: \left(\fint_{B_j} \lvert  \nabla_{\mathbb{E}} u\rvert^{\gamma} \right)^\frac{1}{\gamma}< \frac{\lambda\varepsilon}{50} \right\}. 
\end{align*} 
Clearly, if $j \in \mathcal{L}$, then by H\"{o}lder inequality, we have 
\begin{align*}
	( \nabla_{\mathbb{E}} u)_{B_j} \le \left(\fint_{B_j} \lvert  \nabla_{\mathbb{E}} u\rvert^{\gamma} \right)^\frac{1}{\gamma} < \frac{\lambda\varepsilon}{50}. 
\end{align*}
This means the averages are themselves small each time we hit $\mathcal{L}.$ Indeed, if both $h,k \in \mathcal{L}$, then \eqref{difference of average} follows trivially by triangle inequality. On the other hand, if $j \notin \mathcal{L}$, then $\rm{Ind}^*(j)$ holds. If there are $m$ consecutive integers $i, i+1, \ldots, i+m-1$ such that none of them are in $\mathcal{L}$, then we have 
\begin{align*}
	\lvert ( \nabla_{\mathbb{E}} u)_{B_i}- (\nabla_{\mathbb{E}}  u )_{B_{i+m}} \rvert &\le \sum\limits_{l=i}^{i+m-1} \lvert ( \nabla_{\mathbb{E}} u)_{B_{l}}- (\nabla_{\mathbb{E}}  u )_{B_{l+1}} \rvert \\
	&\le \sum\limits_{l=i}^{i+m-1}\fint_{B_{l}} \left\lvert \nabla_{\mathbb{E}} u - (\nabla_{\mathbb{E}}  u )_{B_{l+1}}\right\rvert \\ &\le 2\sum\limits_{l=i}^{i+m-1}\fint_{B_{l}} \left\lvert \nabla_{\mathbb{E}} u - (\nabla_{\mathbb{E}}  u )_{B_{l}}\right\rvert \le 2 \sum\limits_{l=i}^{i+m-1}\mathscr{E}_{l}. 
\end{align*}
But since $\rm{Ind}^*(j)$ holds for every $j \in \left\lbrace i, i+1, \ldots, i+m-1\right\rbrace$, we can sum \eqref{excess decay without exit time estimate} from $j=i$ to $i+m-1$ to estimate the right hand side above. Thus, once we exit from $\mathcal{L}$ (or never enter $\mathcal{L}$), the averages remain close to each other (in $\bar{\varepsilon}$ scale) til we hit $\mathcal{L}$ again for the next time. Combining these arguments, a simple case by case analysis which we skip, establishes Claim \ref{difference of average claim}. This completes the proof.   
\end{proof}	

			\section*{Acknowledgment}  Both the authors acknowledge the support of the ANRF-ARG Project grant ANRF/ARG/2025/000348/MS. S. Sil's research is also partially supported by the ANRF-SERB MATRICS Project grant MTR/2023/000885. The Department of Mathematics, Indian Institute of Science, where this research was carried out, received support of DST FIST
			program-2021 [TPN-700661]. 
			


	\end{document}